\documentclass[reqno,oneside, notitlepage]{amsart}
\usepackage[
  a4paper,
  textwidth=16.3cm,
]{geometry}

\usepackage{mathrsfs}  
\usepackage{mathtools}
\usepackage{esint}
\usepackage{bm}
\usepackage{stackengine}
\usepackage{enumitem}
\usepackage[final]{hyperref}

\usepackage[dvipsnames]{xcolor}
\usepackage{xargs}
\usepackage[textsize=tiny,textwidth=4cm]{todonotes}

\newtheorem{theorem}{Theorem}[section]
\newtheorem{lemma}[theorem]{Lemma}
\newtheorem{proposition}[theorem]{Proposition}

\theoremstyle{definition}
\newtheorem{definition}[theorem]{Definition}

\theoremstyle{remark}
\newtheorem{remark}[theorem]{Remark}
 
\newcommand{\N}{\mathbb{N}}
\newcommand{\R}{\mathbb{R}}

\newcommand{\Id}{\mathbf{Id}}
\newcommand{\id}{\mathbf{id}}
\newcommand{\eps}{\varepsilon}
\newcommand{\vphi}{\varphi}
\newcommand{\weakly}{\rightharpoonup}
\newcommand{\weaklystar}{\stackrel{*}{\rightharpoonup}}
\newcommand{\defas}{\coloneqq}
\newcommand{\sym}{\mathrm{sym}}
\newcommand{\cplen}{\mathcal{W}^{\mathrm{cpl}}}
\newcommand{\mechen}{\mathcal{M}}
\newcommand{\inten}{W^{\mathrm{in}}}
\newcommand{\tempen}{\mathcal{T}}
\newcommand{\diss}{\mathcal{D}}
\newcommand{\elpot}{W^{\mathrm{el}}}
\newcommand{\cplpot}{W^{\mathrm{cpl}}}
\newcommand{\hypot}{H}
\newcommand{\felpot}{W}
\newcommand{\disspot}{R}
\newcommand{\Wid}[1]{\mathcal{Y}_{h(#1)}} 
\newcommand{\Wzero}{W^{2, p}_{\Gamma_D}(\Omega; \R^d)}
\newcommand{\pl}{\partial} 
\newcommand{\yst}[1]{y_{\tau}^{(#1)}}
\newcommand{\ysts}[1]{y_{\tau}^{(#1)}}
\newcommand{\tst}[1]{\theta_{\tau}^{(#1)}}
\newcommand{\tsts}[1]{\theta_{\tau}^{(#1)}}
\newcommand{\lst}[1]{\ell_{\tau}^{(#1)}}
\newcommand{\wst}[1]{w_{\tau}^{(#1)}} 
\newcommand{\fst}[1]{f_\tau^{(#1)}}
\newcommand{\gst}[1]{g_\tau^{(#1)}}
\newcommand{\bt}{\theta_\flat}
\newcommand{\btst}[1]{\theta_{\flat, \tau}^{(#1)}}
\newcommand{\hc}{\mathbb{K}}
\newcommand{\hcm}{\mathcal{K}}
\newcommand{\drate}{\xi}
\newcommand{\haus}{\mathcal{H}}
\newcommand{\ddif}{\delta_\tau}
\newcommand{\aC}{C_0}
\newcommand{\ac}{c_0}
\newcommand{\ny}{\overline{y}_{\tau}}
\newcommand{\py}{\underline{y}_{\tau}}
\newcommand{\ay}{\hat{y}_{\tau}}
 
\newcommand{\nt}{\overline{\theta}_{\tau}}
\newcommand{\pt}{\underline{\theta}_{\tau}}
\newcommand{\at}{\hat{\theta}_{\tau}}
   
\newcommand{\nw}{\overline{w}_{\tau}}
\newcommand{\pw}{\underline{w}_{\tau}}
\newcommand{\aw}{\hat{w}_{\tau}}
\newcommand{\intQ}{\int_I\int_\Omega}

\newcommand{\cdddot}{\mathrel{\Shortstack{{.} {.} {.}}}}
\newcommand*{\di}{\mathop{}\!\mathrm{d}}

\newcommand{\neohooke}{W_{\rm NH}}

\DeclareMathOperator*{\argmin}{argmin}
\DeclareMathOperator{\trace}{tr}
\DeclareMathOperator{\diver}{div}

\DeclarePairedDelimiterX\setof[1]\{\}{#1}
\DeclarePairedDelimiterX\abs[1]\lvert\rvert{#1}
\DeclarePairedDelimiterX\norm[1]\lVert\rVert{#1}
\DeclarePairedDelimiterX\sprod[2]\langle\rangle{#1, #2}

\numberwithin{equation}{section}

\title[Frame-indifferent discretization in nonlinear thermoviscoelasticity]{Frame-indifferent discretization in nonlinear thermoviscoelasticity: Analysis and numerical simulations }


\author[R.~Badal]{Rufat Badal}
\address[Rufat Badal]{
  Department of Mathematics \\
  Friedrich-Alexander Universit\"at Erlangen-N\"urnberg \\
  Cauerstr.~11, D-91058 Erlangen, Germany
}
\email{rufat.badal@icloud.com}

\author[M.~Friedrich]{Manuel Friedrich} 
\address[Manuel Friedrich]{Department of Mathematics, Johannes Kepler Universit\"at Linz. Altenbergerstrasse 69, 4040 Linz,
    Austria}
\email{manuel.friedrich@jku.at}

\author[M.~Hor\'{a}k]{Martin Hor\'{a}k}
\address[Martin Hor\'{a}k]{Faculty of Civil Engineering \\
  Czech Technical University \\
  Th\'akurova 7, CZ-166 29 Praha 6, Czechia\\
  Czech Academy of Sciences \\
  Institute of Information Theory and Automation \\
  Pod vod\'arenskou v\v{e}\v{z}\'i 4, CZ-182 00 Praha 8, Czechia} \email{Martin.Horak@cvut.cz}

\author[M.~Kru\v{z}\'ik]{Martin Kru\v{z}\'ik}
\address[Martin Kru\v{z}\'ik]{
  Czech Academy of Sciences \\
  Institute of Information Theory and Automation \\
  Pod vod\'arenskou v\v{e}\v{z}\'i 4, CZ-182 00 Praha 8, Czechia  \\
  \& Faculty of Civil Engineering \\
  Czech Technical University \\
  Th\'akurova 7, CZ-166 29 Praha 6, Czechia}
\email{kruzik@utia.cas.cz}

\author[L.~Machill]{Lennart Machill}
\address[Lennart Machill]{Institute for Applied Mathematics, University of Bonn, Endenicher Allee 60, 53115 Bonn, Germany}
\email{lmachill@uni-bonn.de}

\begin{document}

\begin{abstract}
	We consider a quasi-static nonlinear model in thermoviscoelasticity at a finite-strain setting in the Kelvin-Voigt rheology where both the elastic and viscous stress tensors comply with the principle of frame indifference under rotations. We refine the discretization schemes in \cite{RBMFMK,MielkeRoubicek20Thermoviscoelasticity} by imposing frame indifference already at a time-discrete level. This is justified both  analytically and numerically.
\end{abstract}

\maketitle

\section{Introduction}
Evolutionary problems in solid mechanics and their variational analysis have attracted increasing attention over the last decades. In these models, the motion of the material is driven by time-dependent external loads or by initial configurations that are not in equilibrium. A natural way to describe the evolution is through the balance of momentum, which, in Lagrangian (referential) coordinates, takes the form
\begin{align}\label{eq:verygeneral}
	\rho \, \partial^2_t  y  - \diver ( \mathbf{T} ) =  f        \ \ \  \text{ in $ [0,T] \times    \Omega$}.
\end{align}
Here, $\Omega \subset \R^d$ represents the reference configuration of the material, while $y\colon [0,T] \times \Omega \to \R^d$ indicates the deformation mapping. The stress tensor $\mathbf{T}$   depends on  the material properties, and $f\colon [0,T] \times \Omega \to \R^d$ indicates a density of body forces. Finally, $\rho$ denotes a mass density.

In nonlinear elastodynamics, the tensor $\mathbf{T}$ is  calculated from a frame-indifferent density $W\colon \R^{d \times d} \to [0,\infty)$ via $\mathbf{T} = \mathbf{T}(\nabla y) = \partial_F W(\nabla y)$.
In contrast to the purely elastic case,  viscoelastic materials additionally depend on the first time derivative of the strain $\nabla y$. One of the fundamental concepts to describe these materials is the Kelvin-Voigt rheology, where $\mathbf{T}$ is determined by the  sum of the elastic stress and viscous stress, leading to
$ \mathbf{T} = \partial_F W(\nabla y)+ \partial_{\dot F} R (\nabla y, \partial_t \nabla y)$.
Here, the function $R \colon  \R^{d \times d} \times \R^{d \times d} \times \R_+   \to \R $ is the density of the viscous stress tensor with $\dot F$ denoting the placeholder for the strain rate $\partial_t \nabla y$.
To comply with dynamical frame-indifference \cite{Antmann}, $R$ has to  take a specific form and particularly   depends   on both the strain and the strain rate. As a result of this principle, $\mathbf{T}$ is not monotone in the strain rate  rendering the problem highly nontrivial, see \cite[Appendix B]{demoulini}  and \cite[Section~4(b)]{sengul2}.
Similarly to the setting of elastodynamics, the existence of a solution  is only known
for initial data appropriately close to a smooth equilibrium, see \cite{potier-ferry-1,potier-ferry-2},  while existence results under different monotonicity assumptions are proved in \cite{demoulini,Lewick,Tvedt}.
Other viscoelastic rheologies including Maxwell elements are  discussed, e.g., in \cite{Poynting}.

To date, one still needs to resort to energy densities with higher-order spatial gradients in order to prove the existence of weak solutions beyond the one-dimensional framework. Such an approach, referred to as nonsimple-material models \cite{Toupin62Elastic,Toupin64Theory}, has proved useful in various settings of continuum mechanics, for instance, we refer to \cite{BallCurrieOlver81Null,Batra76Thermodynamics,Mielke20Roubicek202016Rateindependent,Podio02Contact}. While the existence of quasi-static weak solutions to \eqref{eq:verygeneral} for nonsimple materials in the Kelvin–Voigt rheology (corresponding to $\rho = 0$) has been established in \cite{RBMFMK,positivity24,Machillpvisco,MielkeRoubicek20Thermoviscoelasticity}, inertial effects have been addressed in \cite{Schwarzacher}.

Moreover, the system has been rigorously related to the linearized equations in \cite{FriedrichKruzik18Onthepassage} and to lower-dimensional counterparts in \cite{FK_dimred,MFLMDimension2D1D,MFLMDimension3D1D}. Further models allow for self-contact \cite{gravina,Kroemer}, nontrivial couplings with a diffusion equation \cite{liero} or a Cahn–Hilliard equation \cite{Leonie}, as well as models with accretive growth \cite{Chiesa2,Chiesa3}. We refer to \cite{gahn} for a homogenization result, to \cite{CesikStability} for stability results on in-time approximation schemes, and to \cite{Roubicek23Eulerian,Roubicek23Eulerian2,Roubicek23Eulerian3,Roubicek23Eulerian4} for recent works employing the alternative Eulerian formulation. General information on the model and many further references on the dynamics can be found in \cite{sengul}.

In thermoviscoelasticity, the mechanical equation is coupled with a nonlinear heat equation, which relates the energy dissipated by viscous stresses to heat production in the material. The study of such models flourished following the work of {\textsc{Dafermos}} \cite{dafermos}. For developments in linearized settings, we refer to
\cite{colli-sprekels,Gawinecki1,Gawinecki2,Lazzaroni18,pawlow-zajackowski,pawlow-zajackowski-1,Roubicek09,winkler,winkler2,Yoshikawa-Pawlow-Zajackowski},
while advances in nonlinear thermoviscoelasticity have been made only more recently.  This progress was triggered by the article of {\textsc{Mielke \& Roubíček}} \cite{MielkeRoubicek20Thermoviscoelasticity}, in which global-in-time weak solutions to the  quasi-static nonlinear system were derived by means of a staggered time-discretization scheme. In this scheme, the deformation and the temperature are updated alternately. It was later refined in \cite{RBMFMK}, where the intermediate regularization of the viscosity term was avoided. The existence theory was further extended to the dynamical setting in \cite{BadalSchwarz}, as well as to results focusing on thermodynamical consistency in \cite{positivity24}. We also refer to recent works on linearization \cite{RBMFMK,positivity24} and dimension reduction \cite{RBMFLM}, as well as to the Eulerian perspective \cite{Roubicek23EulerianTHERMO2,Roubicek23Eulerian2}.

Time-discretization schemes play an important role in this context, both for proving existence results and for numerical implementations. To the best of our knowledge, every time-discretization scheme in large-strain thermoviscoelasticity breaks frame indifference and only recovers frame indifference of the original model in the limit of vanishing time-step size. The goal of this article is to propose a physically more accurate time-discrete approximation scheme for the model in \cite{MielkeRoubicek20Thermoviscoelasticity} by imposing frame indifference already at the time-discrete level. As a second contribution of this article, we implement the scheme using a finite element approximation and discuss three prototypical model settings highlighting qualitative features of the model.

We emphasize that, in the  quasi-static and isothermal setting, a time-discrete frame-indifferent scheme has already been employed successfully in \cite{FriedrichKruzik18Onthepassage,MFLMDimension3D1D,Machillpvisco,MielkeOrtnerSenguel14Anapproach}, relying on the theory of gradient flows in metric spaces. In contrast, in thermoviscoelasticity there is no straightforward gradient-flow structure, which makes the problem more delicate. As in \cite{RBMFMK,MielkeRoubicek20Thermoviscoelasticity}, we employ a variational staggered scheme that alternates between deformation and temperature. The frame-indifference principle, however, requires us to adjust both the mechanical and thermal minimization problems in order to obtain the desired compactness properties, see Section~\ref{sec:timediscrete} below.

The key ingredient for proving compactness properties of the (discrete) strain rates in the large-strain setting is a generalization of the rigidity estimate by {\sc Friesecke, James, and Müller} \cite{FrieseckeJamesMueller:02} due to {\sc Ciarlet and Mardare} \cite{CiarletMardare}. In contrast to the isothermal case \cite{Machillpvisco}, it is more challenging to derive bounds on the strain rates, as these do not immediately follow from the energy inequality obtained from the mechanical minimization problem. Indeed, the alternating scheme generates a term without a sign, see \eqref{comparison_yn} below, which requires suitable estimates for matrix roots due to the frame-indifference principle of the energy densities and the polar decomposition theorem, see Section~\ref{sec:matrixcalc}. Moreover, the rigidity estimate has to be suitably adapted to our setting, since we complement the model with time-dependent boundary  conditions, see Section~\ref{sec:korn}.

It turns out that we obtain the same a priori estimates as in \cite{RBMFMK} for the time-discrete approximations. Although the time-discretization scheme generates new (time-discrete) Euler–Lagrange equations, the equations in \cite{RBMFMK} differ only by lower-order terms. Since the a priori bounds allow us to control these terms, we proceed analogously to \cite{RBMFMK} in the limiting passage to recover a solution of the original system.

In the final part of this contribution, we present an implementation of the staggered time-discretization scheme and discuss three numerical simulations. These simulations are partially  related to corresponding analytical results. First, we compare the dynamically frame-invariant and the non-frame-invariant approximation schemes and set up an experiment in which the simulation yields a more physical solution when time-discrete frame indifference is respected. It is clearly visible from the numerical experiments (see Figure~\ref{fig:example_frameNonIvariance}) that the non-frame-invariant scheme from \cite{RBMFMK} inevitably leads to energy dissipation even for rigid body motions, which is nonphysical. Moreover, this experiment is complemented by a quantitative result measuring the dissipated energy for a minimizer of the non-frame-indifferent scheme, see Theorem~\ref{thm:quantitativeerror}.
In the second experiment, we simulate the long-time behavior of the system. We observe an (exponential) convergence to an equilibrium, see Figure~\ref{fig:creep}, which is in accordance with the prediction in the isothermal case \cite[Theorem 2.3]{Machillpvisco}. In the final experiment, we show that applied loadings may induce cooling effects in the material. For instance, such effects occur in stress-induced phase transformations of shape-memory alloys \cite{Reviewelastocaloric,Elastocaloric1} and shape-memory polymers \cite{shapememorypolymerscaloric}. Mathematically, the heat equation features an adiabatic term which has no sign and may absorb heat. This term arises naturally from the thermodynamical derivation of the heat equation via the entropy balance and a free energy depending on both deformation and temperature, cf.\ \cite[(2.11)--(2.13)]{MielkeRoubicek20Thermoviscoelasticity} for details. This experiment visualizes the role of the adiabatic effect in thermomechanical systems and particularly relates it with rate-dependent viscosity, see Figure~\ref{fig:figure5}.

The plan of this paper is as follows. In Section~\ref{sec:themodel}, we present the model in detail.  Sections~\ref{sec:analyticalresult}--\ref{sec:convtimediscrete} are devoted to the analytical convergence result.
Finally, in Section~\ref{sec:Example}, we address the numerical implementation.\\

\textbf{Notation:} In what follows, we use standard notation for Lebesgue and Sobolev spaces. The lower index $_+$ means nonnegative elements, i.e., $L^2_+(\Omega)$ denotes the convex cone of nonnegative functions belonging to $L^2(\Omega)$, and  we  set $\R_+\defas [0,+\infty)$.   Let $a \wedge b \defas \min\setof{a, b}$ and $a \vee b \defas \max\setof{a, b}$ for $a, \, b \in \R$.
Denoting by $d \ge 2$ the   space   dimension, we let $\Id \in \R^{d \times d}$  be the identity matrix, and  $\id(x) \defas x$ stands for the identity map on $\R^d$.
We define  the subsets  $SO(d) \defas \setof{A \in \R^{d \times d} \colon A^T A = \Id, \,  \det A = 1  }$, $GL^+(d) \defas \setof{F \in \R^{d \times d} \colon \det(F) > 0}$,  and  $\R^{d \times d}_\sym \defas \setof{A \in \R^{d \times d} \colon  A= \sym(A)}$, where $\sym(A) = \frac{1}{2}(A+ A^T)$ is the symmetric part of a matrix $A\in \R^{d \times d}$.  
Furthermore, for $F \in GL^+(d)$ we denote by $F^{-T} \defas (F^{-1})^T=(F^T)^{-1}$ the inverse of the transpose of $F$, and given a tensor $G$ (of arbitrary dimension), $\abs{G}$ indicates its Frobenius norm.
The scalar product between vectors, matrices, and third-order tensors will be written as $\cdot$, $:$, and $\cdddot$, respectively.
For $T \in \R^{d \times d \times d \times d}$ and $A \in \R^{d \times d}$,  $TA \in \R^{ d \times d}$ is given by $(TA)_{ij} = T_{ijkl} A_{kl}$  for $1 \leq i, \, j \leq d$, where we employ Einstein's summation convention.  Any fourth-order tensor  $T \in \R^{d \times d \times d  \times d}$ induces a bilinear form $T \colon \R^{d \times d} \times \R^{d \times d} \to \R$ given by $T[A,B] \defas TA : B = T_{ijkl} A_{kl} B_{ij}$ for any $A, \, B \in \R^{d \times d}$.
As usual, generic constants   may vary from line to line.

\section{The model}\label{sec:themodel}
Consider a bounded domain $\Omega \subset \R^d$ with Lipschitz boundary $\Gamma \defas \partial \Omega$.
Let $\Gamma_D, \, \Gamma_N$ be disjoint Borel subsets of $\Gamma$ such that $\haus^{d-1}(\Gamma_D) > 0$ and $\Gamma = \Gamma_D \cup \Gamma_N$, representing Dirichlet and Neumann parts of the boundary, respectively. Given $p>d$ and a time process interval $I \defas [0,T] $ for $T>0$,
the Dirichlet boundary conditions are parametrized by a time-dependent function $h \in W^{1,\infty}( I ; W^{2,p}(  \R^d; \R^d))$ which is bijective for each $t \in I$ and satisfies      $h^{-1} \in W^{1,\infty}( I ; W^{2,p}( \R^d;  \R^d))$ as well as   $\inf_{I \times \Omega}  \min\lbrace  \det(\nabla h), \det(\nabla h^{-1}) \rbrace  \geq \mu > 0$. Here, $h^{-1}$ is the spatial inverse. We introduce the set of \emph{admissible deformations} by
\begin{equation}\label{Yid}
	\Wid{t} \defas \setof*{
		y \in W^{2, p}(\Omega; \R^d) \colon
		y = h(t) \text{ on } \Gamma_D, \,
		\det(\nabla y) > 0 \text{ in } \Omega
	}.
\end{equation}
Notice that in contrast to some previous results, see e.g.\ \cite{RBMFMK,positivity24,Schwarzacher,MielkeRoubicek20Thermoviscoelasticity}, we consider time-dependent  Dirichlet   boundary conditions, in particular as this is relevant later for the numerical simulations.  
We also refer, e.g., to \cite{MielkeFranc} for time-dependent Dirichlet conditions in finite-strain elasticity.

\subsection{The variational setting}
  Let $0 < \ac < \aC     < \infty    $ be some  fixed      constants.

\noindent \textbf{Mechanical energy and coupling energy:}
The \emph{mechanical energy} $\mechen \colon \Wid{t} \to \R_+$ is given by
\begin{equation}\label{mechanical_energy}
	\mechen(y) \defas \int_\Omega \elpot(\nabla y(x)) + H(\nabla^2 y(x) ) \di x.
\end{equation}
This elastic energy corresponds to the concept of  2nd-grade nonsimple materials, depending on both the gradient and the Hessian of the deformation $y$. Here,  $\elpot \colon GL^+(d) \to \R_+$ is a frame-indifferent elastic energy potential satisfying  
\begin{enumerate}[label=(W.\arabic*)]
	\item \label{W_regularity} $\elpot $ is additively decomposed as $\elpot (F) = W^{\rm el}_1(F) + W^{\rm el}_2(\det(F))$ for $F \in GL^+(d)$, with  $W^{\rm el}_1\in C^2(\R^{d\times d})$ and $W^{\rm el}_2 \in C^2((0,+\infty))$;  
	\item \label{W_frame_invariace} Frame indifference: $\elpot(QF) = \elpot(F)$ for all $F \in GL^+(d)$ and $Q \in SO(d)$;
	\item \label{W_lower_bound}   Growth: Given $q \ge \frac{pd}{p-d}$ and $r \geq \frac{2d q}{q-2}$,    for all $F  \in GL^+(d)$ and $a > 0$ it holds that
	      \begin{align*}
		      \qquad  W^{\rm el}_1(F) & \ge \ac |F|^{r}  - \aC, \quad \quad    W^{\rm el}_2(a) \ge c_0 a^{-q},
	      \end{align*}

	\item \label{W_lipschitz}   Local Lipschitz estimate:    For all $F, \tilde F \in GL^+(d)$ and $a,b > 0$ it holds that   
	      \begin{align*} |W^{\rm el}_1(F) - W^{\rm el}_1(\tilde F)| & \leq C_0 (1 + \vert F \vert^{r-1} + \vert \tilde F \vert^{r-1}  ) \vert F - \tilde F \vert, \\   \vert W^{\rm el}_2(a) - W^{\rm el}_2(b) \vert & \leq C_0 (1+ \vert a \vert^{-q-1} + \vert b \vert^{-q-1}) \vert a - b\vert.
	      \end{align*}
\end{enumerate}
The potential $\hypot \colon \R^{d \times d \times d} \to \R_+$ satisfies
\begin{enumerate}[label=(H.\arabic*)]
	\item \label{H_regularity} $\hypot$ is convex and $C^1$;
	\item \label{H_frame_indifference} Frame indifference: $\hypot(QG) = \hypot(G)$ for all $G \in \R^{d \times d \times d}$ and $Q \in SO(d)$;
	\item \label{H_bounds}  Growth: $\ac \abs{G}^p \leq H(G) \leq \aC (1+ \abs{G}^p)$ and $\abs{\pl_G H(G)} \leq \aC \abs{G}^{p-1}$ for all $G \in \R^{d \times d \times d}$.
\end{enumerate}
In contrast to \cite{RBMFMK}, we consider a superquadratic  growth of $W^{\rm el}$, see Remark~\ref{rem:differenceapriori} for a detailed  discussion on the difference of our model to \cite{RBMFMK}.   
Besides the mechanical energy, we introduce a \emph{coupling energy} $\cplen \colon \Wid{t} \times     L^1_+(\Omega)      \to \R$ given by
\begin{equation}\label{coppler}
	\cplen(y, \theta) \defas \int_\Omega \cplpot(\nabla y, \theta) \di x.
\end{equation}
Here, $\theta \in L^1_+(\Omega)$ is the \emph{absolute temperature} and   $\cplpot \colon GL^+(d) \times \R_+ \to \R$  satisfies
\begin{enumerate}[label=(C.\arabic*)]
	\item \label{C_regularity} $\cplpot$ is continuous, and $C^2$ in $GL^+(d) \times (0, \infty)$;
	\item \label{C_frame_indifference} $\cplpot(QF, \theta) = \cplpot(F, \theta)$ for all $F \in GL^+(d)$, $\theta \geq 0$, and $Q \in SO(d)$;
	\item \label{C_zero_temperature} $\cplpot(F, 0) = 0$ for all $F \in GL^+(d)$;
	\item \label{C_lipschitz} $|\cplpot(F,\theta) - \cplpot(\tilde{F}, \theta)| \le \aC(1 + |F| + |\tilde{F}|)|F - \tilde{F}|$ for all $F, \, \tilde F \in GL^+(d)$        and $\theta \geq 0$;
	\item \label{C_bounds} For all $F \in     GL^+(d)   $ and $\theta > 0$ it holds that
	      \begin{align*}
		      \abs{\partial_F^2 W^{\rm cpl}(F,\theta)} & \le \aC,                                               &
		      \abs{\pl_{F \theta} \cplpot(F, \theta)}  & \leq \frac{\aC(1+|F|)}{\max\lbrace \theta,1\rbrace},   &
		      \ac                                      & \leq -\theta \pl_\theta^2 \cplpot(F, \theta) \leq \aC.
	      \end{align*}
\end{enumerate}
Classical examples for densities in the isothermal case are presented in
\cite[Section~2.4]{KruzikRoubicek19mathmodels}, while a specific coupling density
arising in the context
of shape-memory alloys and polymers
which complies
with \ref{C_regularity}--\ref{C_bounds} is given in \cite[Appendix~B]{positivity24}.
More specifically, the latter serves as a starting point for the third simulation below in Section~\ref{sec:Example}.
The advantages of the potential $H$ in the context of elastic stresses are extensively discussed in \cite[Section~2.5]{KruzikRoubicek19mathmodels}.  

For $F \in GL^+(d)$ and $\theta \geq 0$, we define the \emph{total free energy potential}
\begin{align}\label{eq: free energy}
	\felpot(F, \theta) \defas \elpot(F) + \cplpot(F, \theta),
\end{align}
and   the density of the internal energy $\inten \colon GL^+(  d ) \times (0, \infty) \to \R_+$ by
\begin{equation*}
	\inten(F, \theta) \defas \cplpot(F, \theta) - \theta \partial_\theta \cplpot(F, \theta).
\end{equation*}
As discussed in \cite{RBMFMK}, \ref{C_zero_temperature} and the third bound in \ref{C_bounds} imply that $\inten$ can be continuously extended to zero temperatures by setting $\inten(F, 0) = 0$ for all $F \in GL^+(d)$.
Also by the third bound in \ref{C_bounds}, the internal energy is controlled by the temperature in the following sense:
\begin{align}\label{sec_deriv}
	\partial_{\theta} \inten (F, \theta)
	= -\theta \pl_\theta^2 \cplpot(F, \theta) \in [\ac, \aC]
	\qquad \text{for all $F \in GL^+(d)$ and $\theta  >   0$}
\end{align}
which along with \ref{C_zero_temperature} yields
\begin{equation}\label{inten_lipschitz_bounds}
	\ac \theta \leq \inten(F, \theta) \leq \aC \theta.
\end{equation}

\noindent\textbf{Dissipation mechanism:}
We consider the potential of dissipative forces of the form
$R(F,\dot F, \theta) \defas \frac{1}{2} V(C,\theta) [\dot C, \dot C] = \frac{1}{2} \dot C :V(C, \theta) \dot C$, where we set $C \defas F^TF$ and  $\dot C \defas \dot F^T F + F^T \dot F$, and  
$V \in C(\R^{d \times d}_{\rm sym} \times \R_+ ; \R^{d\times d \times d \times d})$ denotes a fourth order tensor  satisfying
\begin{enumerate}[label=(D.\arabic*)]
	\item \label{D_quadratic}  with $V_{ijkl} = V_{jikl}= V_{klij}$ for $1 \le i,\,j,\,k,\,l \le d$;
	\item \label{D_bounds} $c_0 \vert \dot C \vert^2 \leq \frac{1}{2} \dot C :V(C, \theta) \dot C \leq C_0 \vert \dot C \vert^2$ for $0 <c_0< C_0$ for all $C, \, \dot C \in \R^{d \times d}_\sym$, and $\theta \geq 0$.
\end{enumerate}
The resulting viscous stress tensor (see \cite[(2.8)]{RBMFMK}) is then given as
\begin{align}\label{viscousstresstensor}
	\partial_{\dot{F}}R(F,\dot F,\theta)  = 2 F  \, V(C, \theta) \dot C   .
\end{align}
Given $y\in  W^{2, p}(\Omega; \R^d)$, $\tilde y \in  W^{2, p}(\Omega; \R^d) $, and $\theta \in L_+^1(\Omega)$ we consider the \emph{dissipation functional}  $\diss \colon  W^{2, p}(\Omega; \R^d) \times W^{2, p}(\Omega; \R^d) \times L_+^1(\Omega) \to \R_+$ defined as
\begin{equation}\label{dissipation}
	\mathcal D(y, \tilde y,\theta) \defas \frac{1}{2} \int_\Omega D^2(\nabla   y, \nabla \tilde  y,\theta) \di x,
\end{equation}
where we consider two different options for the density $D^2$:

(a) The frame-indifferent option for $D^2$ is given by
\begin{enumerate}[label=(V.1)]
	\item \label{Case2-newDissipation} $D^2(F_1, F_2,\theta)  \defas  V(F_1^T F_1, \theta) [F_2^T F_2 - F_1^T F_1 , F_2^T F_2 - F_1^T F_1]$
\end{enumerate}
for $F_1, \, F_2 \in GL^+(  d )$.  More specifically, this function satisfies \emph{separate frame indifference}, i.e., for all $Q_1, Q_2 \in SO(d)$ and all $F_1, F_2 \in GL^+(d)$ and for all $\theta \geq 0$ we have that
\begin{align}\label{def:seperateframeindifference}
	D^2(F_1, F_2,\theta) = D^2 (Q_1F_1, Q_2 F_2,\theta).
\end{align}
This particularly means that the dissipation potential is invariant under rigid body motions. Indeed, given $y(t) = R(t) y_0$ for some $R \in C^\infty([0,T];SO(d))$ we have $\mathcal D(y(0), y(t),\theta) = 0 $.   The choice \ref{Case2-newDissipation}  was proposed in a similar manner in \cite[Example~2.4]{MielkeOrtnerSenguel14Anapproach}.  

(b) Our second option is different and has been adopted in \cite{RBMFMK,MielkeRoubicek20Thermoviscoelasticity}.  Here, $D^2$ takes the form
\begin{enumerate}[label=(V.2)]
	\item \label{Case1-oldDissipation} $D^2(F_1, F_2,\theta)  \defas 2 R(F_1, F_2 - F_1, \theta)$
\end{enumerate}
for $F_1, \, F_2 \in GL^+(  d )$.
  We note that the choices \ref{Case2-newDissipation} and \ref{Case1-oldDissipation} only differ by a  second-order contribution in terms of $|F_2 - F_1|$. Indeed, an elementary computation yields
\begin{align} \label{calculationhelpful}
	\left(F_2-F_1\right)^{T} F_1+F_1^{T}\left(F_2-F_1\right)-\left(F_2^{T} F_2-F_1^{T} F_1\right)  =-\left(F_2^{T}-F_1^{T}\right)\left(F_2-F_1\right)   .
\end{align}
  As we will see later, both options  in the time-discrete approximation  lead to the same system of PDEs in the limit.

\noindent\textbf{Heat conductivity:}
The map $\hc \colon     \R_+     \to \R^{d \times d}_\sym$ will denote the \emph{heat conductivity tensor} of the material in the deformed configuration.
We require that $\hc$ is continuous, symmetric, uniformly positive definite, and bounded.
More precisely, for all     $\theta \geq 0$     it holds that
\begin{equation*}
	\ac \leq     \hc(\theta)     \leq \aC,
\end{equation*}
where the inequalities are meant in the eigenvalue sense.
We define the pull-back $\hcm \colon     GL^+(d)     \times \R_+ \to \R^{d \times d}_\sym$ of $\hc$ into the reference configuration by (see \cite[(2.24)]{MielkeRoubicek20Thermoviscoelasticity})
\begin{equation}\label{hcm}
	\hcm(F, \theta)     \defas \det(F) F^{-1}     \hc(\theta)     F^{-T}.
\end{equation}

\noindent\textbf{Equations of nonlinear thermoviscoelasticity:} Recall the notation $I  =  [0, T]$.
Let $f \in W^{1, 1}(I; L^2(\Omega; \R^d))$ be a time-dependent \emph{dead load}, $g \in W^{1, 1}(I; L^2(\Gamma_N; \R^d))$  be     a \emph{boundary traction}, and let $\bt \in  L^2 (I; L^2_+(\Gamma))$      be an external temperature. We study the coupled system
\begin{subequations}\label{main_evol_eq}
	\begin{align}
		\hspace{-0.2cm} 	f                                                              & =
		-\diver\big(
		\pl_F \felpot(\nabla y, \theta)
		+ \pl_{\dot F} \disspot(\nabla y, \nabla \partial_t y, \theta)
		-\diver(\pl_G \hypot(\nabla^2 y))
		\big) , \label{main_mechanical_eq}                                                                            \\
		\hspace{-0.2cm} 	-\theta \pl_\theta^2 \cplpot(\nabla y, \theta) \, \dot{\theta} & =
		\diver(\hcm(\nabla y, \theta) \nabla \theta)
		+ \drate(\nabla y, \nabla \partial_t y, \theta)
		+ \theta \pl_{F \theta} \cplpot(\nabla y, \theta) : \nabla \partial_t y         &   & \label{main_thermal_eq}
	\end{align}
\end{subequations}
  defined on $I \times  \Omega$,  
where  $\drate \colon \R^{d \times d} \times \R^{d \times d} \times \R_+ \to \R_+$  is the \emph{dissipation rate} defined as
\begin{align}\label{diss_rate}
	\drate(F, \dot F, \theta)
	\defas \pl_{\dot F} \disspot(F, \dot F, \theta) : \dot F
	= 2 F (V(C,\theta)\dot C) : \dot F
	= V(C, \theta) \dot C : \dot C
	= 2 R(F, \dot F,\theta).
\end{align}
Here, the equalities follow from \eqref{viscousstresstensor} and \ref{D_quadratic}.
The form of the heat equation \eqref{main_thermal_eq} is explained in \cite[(2.11)--(2.13)]{MielkeRoubicek20Thermoviscoelasticity}.
The positive factor $-\theta \pl_\theta^2 \cplpot(\nabla y, \theta)$, see \eqref{sec_deriv}, corresponds to the \emph{heat capacity} and the last term in \eqref{main_thermal_eq} is an \emph{adiabatic term}. The latter has no sign and can play the role as a heat source or sink, respectively. This feature is the starting point of the third experiment in Section~\ref{sec:Experiment3}.

As in \cite{MielkeRoubicek20Thermoviscoelasticity}, the mechanical equations are complemented with the boundary conditions
\begin{align}\label{main_bc}
	\hspace{-0.2cm}	\big(
	\pl_F \felpot(\nabla y, \theta)
	+ \pl_{\dot F} \disspot(\nabla y, \nabla \dot y, \theta)  -\diver(\pl_G \hypot(\nabla^2 y))  
	\big) \nu
	- \diver_S \big( \pl_G \hypot(\nabla^2 y) \nu \big)
	                  & =  g   &  & \hspace{-0.2cm} \text{on }     I \times     \Gamma_N, \notag \\
	\hspace{-0.2cm}	y & = h &  & \hspace{-0.2cm}  \text{on } I \times \Gamma_D,      \notag   \\
	\hspace{-0.2cm}	\pl_G \hypot(\nabla^2 y)  : (\nu \otimes \nu)
	                  & = 0    &  & \hspace{-0.2cm} \text{on }I \times  \Gamma,
\end{align}
and for the heat equation it holds that
\begin{align}\label{main_bc_heat}
	\hcm(\nabla y, \theta) \nabla \theta \cdot  \nu     + \kappa \theta
	= \kappa  \bt \qquad \text{ on } I \times \Gamma.
\end{align}
Here, $\nu$ denotes the outward pointing unit normal on $\Gamma$ and $\kappa \ge 0$ is a \emph{phenomenological heat-transfer coefficient} on $\Gamma$. Moreover, $\diver_S$   represents the \emph{surface divergence}, defined by $\diver_S(\cdot) = \trace(\nabla_S(\cdot))$, where $\trace$ denotes the trace and $\nabla_S \defas (\Id - \nu \otimes \nu) \nabla$ denotes the surface gradient. We refer to \cite[(2.28)--(2.29)]{MielkeRoubicek20Thermoviscoelasticity} for an explanation and derivation of the boundary conditions.

We impose the initial conditions
\begin{equation}\label{initial_cond}
	y(0, \cdot) =  y_0     \qquad \text{and} \qquad \theta(0, \cdot) = \theta_0
\end{equation}
for some  $\theta_0\in L^2_+(\Omega)$ and some $y_0     \in \Wid{0}  $.
We now define weak solutions associated to the initial-boundary-value problem \eqref{main_evol_eq}--\eqref{initial_cond}.

\begin{definition}[Weak solutions]\label{def:weak_formulation}
	A couple $(y, \theta) \colon I \times \Omega \to \R^d \times \R$ is called a \emph{weak solution} to the initial-boundary-value problem \eqref{main_evol_eq}--\eqref{initial_cond} if $y \in L^\infty(I; W^{2,p}(\Omega)) \cap H^1(I; H^1(\Omega; \R^d))$ with $y(0, \cdot)=  y_{0}   $ and $y(t) = h(t)$ a.e.\ on $I \times \Gamma_D$, $\theta \in L^1(I; W^{1,1}(\Omega))$ with $\theta \ge 0$ a.e., and if it satisfies the identities
	\begin{equation}\label{weak_limit_mechanical_equation}
		\begin{aligned}
			 & \intQ \pl_G \hypot(\nabla^2 y) \cdddot \nabla^2 z
			+ \Big(
			\pl_F \felpot(\nabla y, \theta)
			+ \pl_{\dot F} \disspot(\nabla y, \partial_t\nabla  y, \theta)
			\Big) : \nabla z \di x \di t                         \\
			 & \quad=   \intQ f \cdot z \di x \di t
			+   \int_I \int_{\Gamma_N} g \cdot z \di \haus^{d-1} \di t
		\end{aligned}
	\end{equation}
	for any test function $z \in C^\infty(I \times \overline{\Omega}; \R^d)$ with $z = 0$ on $I \times \Gamma_D$, as well as
	\begin{equation}\label{weak_limit_heat_equation}
		\begin{aligned}
			 & \intQ \hcm(\nabla y, \theta) \nabla \theta \cdot \nabla \vphi
			-\big(
			\drate (\nabla y, \partial_t \nabla y, \theta)
			+ \pl_F \cplpot(\nabla y, \theta) : \partial_t \nabla  y
			\big) \vphi
			- \inten(\nabla y, \theta) \partial_t \vphi \di x \di t                                 \\
			 & \quad + \kappa \int_I \int_{    \Gamma} \theta \vphi \di \haus^{d-1} \di t
			= \kappa  \int_I \int_{    \Gamma} \bt \vphi \di \haus^{d-1} \di t
			+ \int_\Omega \inten(\nabla  y_{0},     \theta_{0}) \, \vphi(0) \di x
		\end{aligned}
	\end{equation}
	for any test function $\vphi \in C^\infty(I \times \overline \Omega)$ with $\varphi(T) = 0$.
\end{definition}
  One can indeed show that sufficiently smooth weak solutions lead to the classical formulation \eqref{main_evol_eq} along with the boundary and initial conditions  \eqref{main_bc}--\eqref{initial_cond}, see \cite{RBMFMK,MielkeRoubicek20Thermoviscoelasticity}.

\subsection{Frame-indifferent  discretization in time}\label{sec:timediscrete}
Given a step size $\tau>0$ such that $T/\tau \in \N$,  we divide the time interval $[0,T]$ into $  T / \tau  $ intervals of length $\tau>0$
and solve global minimization problems iteratively, in order to define time-discrete deformations and temperature.  

  Define $\ysts{0} = y_0$ and  $\tsts{0} = \theta_0$. Suppose that we have already constructed $\ysts{0}$, $\tsts{0}$, \dots, $\ysts{k-1}$, and $\tsts{k-1}$ for some $k \in \lbrace 1 , \ldots T/\tau \rbrace$.   
The next mechanical step $\ysts{k}$ is a solution of the   problem
\begin{equation}\label{mechanical_step}
	\min_{y \in \Wid{k\tau}} \hspace{-0.1cm}\left\{
	\mechen(y) + \int_\Omega \hspace{-0.1cm}\cplpot(\nabla y, \tsts{k-1}) \di x
	+ \frac{1}{\tau} \mathcal D(\ysts{k-1}, y, \tsts{k-1})  
	- \int_\Omega\hspace{-0.1cm} \fst{k} \cdot y \di x
	- \int_{\Gamma_N}\hspace{-0.1cm} \gst{k} \cdot y \di \haus^{d-1}
	\right\},
\end{equation}
where $\fst k \defas \tau^{-1} \int_{(k-1)\tau}^{k\tau} f(t) \di t$  and  $g^{(k)}_\tau \defas \tau^{-1} \int_{(k-1)\tau}^{k\tau} g(t) \di t$. Here, $D^2$ either satisfies  \ref{Case2-newDissipation} or \ref{Case1-oldDissipation}.  
Considering a solution $\ysts k$ of the above minimization problem,
the next thermal step $\tsts{k}$ is a solution of the  problem
\begin{align}\label{thermal_step}
	 & \min_{\theta \in H^1_+(\Omega)}
	\Bigg\{
	\int_\Omega \int_0^\theta \frac{1}{\tau}\big(
	\inten(\nabla \ysts{k},s)
	- \inten(\nabla \ysts{k-1},\tsts{k-1})
	\big) \di s \di x                          \\
	 & + \int_\Omega \frac{1}{2} \nabla \theta
	\cdot \hcm(\nabla \ysts{k-1}, \tsts{k-1}) \nabla \theta \di x
	- \int_\Omega h_\tau(\ysts{k}, \ysts{k-1}, \tsts{k-1}) \, \theta \di x
	+ \frac{\kappa}{2} \int_\Gamma (\theta - \btst k)^2 \di \haus^{d-1} \notag
	\Bigg\},
\end{align}
where $h_\tau$ features source and sink contributions, defined below in \eqref{htaucase2} and  \eqref{htaucase1}, respectively,  and $ \btst k  \defas \tau^{-1} \int_{(k-1)\tau}^{k\tau} \theta_\flat(t) \di t$. To ensure that   both  the mechanical and thermal step satisfy time-discrete frame indifference (in the case $f = g \equiv 0$),  we need to guarantee that the minimization problems can be rewritten in terms of the multiplicative symmetrized strain $(\nabla \ysts{j})^T \nabla \ysts{j}$  for $j = k-1,k$.  Clearly, $W^{\rm el}$, $H$, $W^{\rm cpl}$, $W^{\rm in}$, and $\mathcal{K}$ satisfy frame indifference due to \ref{W_frame_invariace}, \ref{H_frame_indifference}, \ref{C_frame_indifference}, \eqref{hcm}, and the invariance of the determinant under rotations.
By the discussion in \eqref{def:seperateframeindifference}, the mechanical step is frame indifferent for option   \ref{Case2-newDissipation}.
The choice \ref{Case1-oldDissipation} from \cite{RBMFMK, MielkeRoubicek20Thermoviscoelasticity} does not  comply with  frame  indifference in the sense of \eqref{def:seperateframeindifference}: otherwise, given two rotations $S_1,S_2 \in SO(d)$ with $S_1 \neq S_2$, \ref{D_bounds} implies that
\begin{align*}
	0  =  D^2 (S_1^{-1} S_1, S_2^{-1} S_2,\theta) = D^2(S_1, S_2,\theta) \geq 2 c_0 \vert 2 \, \sym (S_1^T (S_2-S_1))\vert^2 > 0,
\end{align*}
  a contradiction.  
Similarly, the term $h_\tau$ from \cite{RBMFMK},  defined by  
\begin{align}\label{htaucase1}
	 & \quad \ h_\tau(\ysts{k}, \ysts{k-1}, \tsts{k-1})   \tag{$\mathrm{V}^{h}\hspace{-0.15cm}\mathrm{.\hspace{0.02cm}2}$} \\ &=
	\tau^{-2}\drate(\nabla \ysts{k-1}, \nabla \ysts k -\nabla \ysts{k-1} , \tsts{k-1}) +
	\tau^{-1} \pl_F \cplpot(\nabla \ysts{k-1}
	, \tsts{k-1}) : \big( \nabla \ysts k -\nabla \ysts{k-1} \big), \notag
\end{align}
does \emph{not} satisfy frame indifference. Therefore, for  our frame-indifferent approximation scheme, we  use a different definition of $h_\tau$, namely 
\begin{align}\label{htaucase2}
	 & h_\tau(\ysts{k}, \ysts{k-1}, \tsts{k-1})  =  \tau^{-2} D^2(\nabla \ysts{k-1}, \nabla \ysts{k}, \tsts{k-1}) \notag                                                                                                                                                 \\
	 & \qquad   + \tau^{-1} \partial_C \hat{W}^{\rm cpl} \big( (\nabla \ysts{k-1})^T \nabla \ysts{k-1}, \tsts{k-1} \big) : \big( (\nabla \ysts{k} )^T \nabla \ysts{k} - (\nabla \ysts{k-1})^T \nabla \ysts{k-1} \big), \tag{$\mathrm{V}^{h}\hspace{-0.15cm}\mathrm{.1}$}
\end{align}
where $D^2$ is as in \ref{Case2-newDissipation} and  $\hat{W}^{\rm cpl}$ is defined by ${\cplpot}(F, \theta ) = \hat{W}^{\rm cpl}(C, \theta )$ with $C = F^TF$, which is possible due to   \ref{C_frame_indifference}. Clearly, $\partial_C \hat{W}^{\rm cpl}$ is symmetric which implies   with  the chain rule     that
\begin{align}\label{eq: strange deriv}
	\partial_F {\cplpot}(F, \theta  ) =    F \big(\partial_C\hat{W}^{\rm cpl}(C, \theta  ) + (\partial_C\hat{W}^{\rm cpl}(C, \theta  ))^T\big)    = 2F\partial_C\hat{W}^{\rm cpl}(C,  \theta
	),
\end{align}
  see \cite[(3.17)]{RBMFMK}.  
Using again   \eqref{calculationhelpful} and the symmetries of $\partial_C \hat W^{\rm cpl}$, we  observe that  \eqref{htaucase2} and \eqref{htaucase1}   differ by lower order terms. This will be exploited in Lemma~\ref{lem:strong_strain_rates_conv} and Lemma~\ref{thm:convergence_heat_vanishing_tau} below.

\subsection{Main result}

We now proceed with the main result of our contribution, the convergence of frame-indifferent time-discrete solutions to solutions in the sense of Definition~\ref{def:weak_formulation}. Supposing that the steps $\ysts 0, \ldots,  \ysts{T / \tau}$ and $\tsts 0, \ldots,  \tsts{T / \tau}$ as described above exist, we define  interpolations as follows: for $k \in \setof{0, \ldots,  T / \tau}$,   we let $\overline{y}_\tau(k\tau) =   \underline{y}_\tau(k\tau) = \hat{y}_\tau(k\tau)  \defas  \ysts k$   and for  $t \in ((k-1)\tau, k\tau)$
\begin{align}\label{y_interpolations}
	\overline{y}_\tau(t)  & \defas \ysts k,    &
	\underline{y}_\tau(t) & \defas \ysts{k-1}, &
	\hat{y}_\tau(t)   \defas \frac{k \tau - t}{\tau} \ysts{k-1} + \frac{t - (k-1)\tau}{\tau} \ysts k.
\end{align}
A similar notation is employed for $\overline{\theta}_\tau$, $\underline{\theta}_\tau$, and $\hat{\theta}_\tau$.
We  formulate our main result concerning the convergence of solutions to the staggered scheme towards a weak solution in the sense of Definition~\ref{def:weak_formulation}.  
\begin{theorem}[Convergence of solutions]\label{thm:van_tau}
	Let $T>0$. Suppose that either \ref{Case2-newDissipation} and \eqref{htaucase2} or  \ref{Case1-oldDissipation} and  \eqref{htaucase1} hold. Then, there exists $\tau_0 \in (0, 1]$ such that for any $\tau \in (0, \tau_0)$  we have the   following: \\
	{\rm (i)}   (Well-definedness of the scheme)  The sequences $\ysts 0, \ldots,  \ysts{T / \tau}$ and $\tsts 0, \ldots,  \tsts{T / \tau}$ satisfying  \eqref{mechanical_step} and \eqref{thermal_step} exist.\\
	{\rm (ii)}  (Convergence)  There exist $y \in L^\infty(I; W^{2,p}(\Omega;\R^d))  \cap H^1(I; H^1(\Omega; \R^d)) $ and $\theta \in L^1(I; W^{1,1}(\Omega))$ such that  the couple $(y, \theta)$   is a solution in the sense of Definition~\ref{def:weak_formulation}  and,  up to selecting a subsequence, it holds  that
	\begin{align}
		\hat{y}_\tau                         & \to y \text{ in } L^\infty(I; W^{1,\infty}(\Omega;\R^d))
		                                     &                                                                                 & \text{and} &
		   \partial_t {\hat{y}}_\tau & \to      \partial_t y \text{ strongly in } L^2(I; H^1(\Omega; \R^d)),
		\label{van_tau_y_conv}                                                                                                                \\
		\hat{\theta}_\tau                    & \to \theta \text{ in }   L^s(I \times \Omega)
		                                     &                                                                                 & \text{and} &
		\hat{\theta}_\tau                    & \weakly \theta \text{ weakly in } L^r(I; W^{1,r}(\Omega))
		\label{van_tau_theta_conv}
	\end{align}
	as $\tau \to 0$  for any $r \in [1, \tfrac{d+2}{d+1})$ and $s \in [1, \frac{d+2}{d})$.
	The same holds true if we replace $\hat{y}_\tau$ by $\overline{y}_\tau$ or $\underline{y}_\tau$   in the first part of \eqref{van_tau_y_conv},  and $\hat{\theta}_\tau$ by $\overline{\theta}_\tau$ or $\underline{\theta}_\tau$  in \eqref{van_tau_theta_conv}.
\end{theorem}
As discussed, the  novelty is the convergence result in the case \ref{Case2-newDissipation} and \eqref{htaucase2}.  
Convergence in the case \ref{Case1-oldDissipation} and \eqref{htaucase1} was proved in \cite{RBMFMK} for the special case of non-varying boundary conditions. To this end, in the next sections we will only focus on the case   \ref{Case2-newDissipation}  and \eqref{htaucase2}.

\section{Preliminaries}\label{sec:analyticalresult}

In this section we  state and prove technical estimates that are needed later.

\subsection{Preliminary energy estimates}
Let us start with some fundamental auxiliary results.  The first lemma derives bounds on the deformation stemming from the mechanical energy.  
\begin{lemma}[A priori estimates, positivity of determinant]\label{lem:pos_det}
	Given $M > 0$ there exists a constant $C_M > 0$ such that for all $y \in \Wid{t}$ with $\mechen(y) \leq M$ (where $\mechen$ is defined in \eqref{mechanical_energy}) it holds that
	\begin{align}\label{pos_det}
		\hspace{-0.2cm}\norm{y}_{W^{2, p}(\Omega)} \leq C_M, \ \
		\norm{y}_{C^{1, 1-d/p}(\Omega)} \leq C_M, \ \
		\norm{(\nabla y)^{-1}}_{C^{1 - d/p}(\Omega)} \leq C_M, \ \
		\det(\nabla y) \geq \frac{1}{ C_M} \text{ in } \Omega.
	\end{align}
\end{lemma}

\begin{proof}
	The statement is proved in \cite[Theorem~3.1]{MielkeRoubicek20Thermoviscoelasticity} relying on a result in \cite{HealeyKroemer09Injective}.   Notice that the time-dependent boundary conditions $h$ only play a role in the first two estimates. More precisely, the norm  $\Vert y - h(t)  \Vert_{L^p(\Omega)}$ is controlled by the Poincaré inequality. Hence, the triangle inequality and the imposed regularity of $h$ in \eqref{Yid} allow  to take a uniform constant $C_M>0$ such that \eqref{pos_det} holds for all $t \in I$.  
\end{proof}

The next two lemmas address bounds which immediately follow from the assumptions on the potentials in Section~\ref{sec:themodel}, where we note that the proof of \eqref{bound_hcm} additionally relies on Lemma~\ref{lem:pos_det}.  
\begin{lemma}[Estimate on coupling potentials]
	Then, there exists a constant $C>0$ such that for all $F \in  GL^+(d)$,   $\dot F \in \R^{d \times d}$,  and $\theta \geq 0$ it holds that
	\begin{align}
		\vert \partial_F W^{\rm cpl} (F,   \theta   ) \vert + \vert \partial_F W^{\rm in}(F,  \theta   ) \vert & \leq C (  \theta    \wedge 1 ) (1 + \vert F \vert).  \label{C_locally_lipschitz}
	\end{align}
\end{lemma}

\begin{proof}
	The statement is proved in  \cite[Lemma~4.4]{RBMFLM}.
\end{proof}

\begin{lemma}[Heat conductivity and dissipation]\label{lem:bound_hcm}
	For any $M > 0$ there exist constants $c_M, \, C_M > 0$ such that for $y \in \Wid{t}$ satisfying  $\mechen(y) \leq M$ and $\theta \in L^{1}(\Omega)$ we have that $\hcm(\nabla y, \theta)$ is well-defined and
	\begin{equation}\label{bound_hcm}
		c_M \leq \hcm(\nabla y, \theta) \leq C_M.
	\end{equation}
	Consider the case  \ref{Case2-newDissipation}.
	Then, there exists some $0< c_0 <C_0$ such that
	\begin{equation}\label{bound_V1V2dissipation}
		c_0 \vert F_2^T F_2 - F_1^T F_1 \vert^2 \leq D^2(F_1, F_2,\theta,x) \leq C_0 \vert F_2^T F_2 - F_1^T F_1 \vert^2
	\end{equation}
	for $F_1, \, F_2 \in  GL^+(d)$.
\end{lemma}

\begin{proof}
	\eqref{bound_hcm} is proved in \cite[Lemma~3.3]{RBMFMK},   where we note that the constant is uniform in $t \in I$ as discussed in the proof of \eqref{pos_det}. \eqref{bound_V1V2dissipation} follows from the assumptions on the tensor $V$, see \ref{D_bounds}.
\end{proof}

\subsection{Matrix calculus}\label{sec:matrixcalc}
The following lemma is   one of the main ingredients to derive compactness properties by using the frame-indifferent approximation scheme. In fact, we will need a bound on $\vert \sqrt{C_k} -  \sqrt{C_{k-1}} \vert$, where $C_j = (\nabla \ysts{j})^T \nabla \ysts{j}$ for $j=k-1,k$.  
\begin{lemma}
	Let $D,A \in \R^{d \times d}$ be symmetric matrices with positive eigenvalues.  Then, we have
	\begin{align}\label{ineq:matrixcalc}
		|A-D| \leq \frac{1}{\alpha+\tilde{\alpha}}\left|A^2-D^2\right|,
	\end{align}
	where $\alpha, \tilde{\alpha}>0$ are the minimal eigenvalues of $D$ and $A$, respectively.
\end{lemma}

\begin{proof}
	Without loss of generality, we can assume that $D$  is a diagonal matrix. Indeed, if $D$ is not a diagonal matrix, we find  a diagonal matrix $\tilde D$ and a rotation $\tilde S \in SO(d)$ such that $D = \tilde S^T \tilde D \tilde S$. Since the Frobenius norm is invariant under rotations, we thus have
	\begin{align*}
		\vert A - D \vert = \vert \tilde A - \tilde D \vert, \qquad \vert A^2 - D^2 \vert = \vert (\tilde A)^2 - (\tilde D)^2 \vert,
	\end{align*}
	where $\tilde A = \tilde S A \tilde S^T$.

	Consider  symmetric matrices $D,A \in \R^{d \times d}$ with positive eigenvalues such that $D$ is a diagonal matrix. Let $\alpha \leq \lambda_1 \leq \ldots \leq \lambda_d $ be the eigenvalues of $D$ and $\tilde{\alpha} \leq \tilde{\lambda}_1 \leq \ldots \leq \tilde{\lambda}_d $ be the eigenvalues of $A$.
	Let $M \defas A-D$ and let $\hat{A}$ be the diagonalization of $A$, i.e., there exists some rotation $S \in SO(d)$ such that $S^T \hat A S = A$. Setting $B \defas SM$ with entries $b_{ij}$ for $i,j \in \{1, \ldots, d\}$,
	the symmetry of $A$ and the diagonality of $D$ and $\hat{A}$ imply that
	\begin{align}\label{eq:computation1matrix}
		(A-D) D : A (A-D) & = AMD:M
		= \hat A S M D : SM
		= \hat A B D: B \notag                                                                                                                   \\
		                  & = \begin{pmatrix}
			                      \tilde \lambda_1 \lambda_1 b_{11} & \tilde \lambda_1 \lambda_2 b_{12} & \ldots & \tilde \lambda_1 \lambda_d b_{1d} \\
			                      \tilde \lambda_2 \lambda_1 b_{21} & \ldots                            & \ldots & \tilde \lambda_2 \lambda_d b_{2d} \\
			                      \ldots                            & \ldots                            & \ldots & \ldots                            \\
			                      \tilde \lambda_d \lambda_1 b_{d1} & \ldots                            & \ldots & \tilde \lambda_d \lambda_d b_{dd} \\
		                      \end{pmatrix} : B \notag \\
		                  & \geq \alpha \tilde \alpha \vert B \vert^2
		= \alpha \tilde \alpha \vert A - D \vert^2.
	\end{align}
	In a similar spirit, it follows that
	\begin{align}\label{eq:computation2matrix}
		\vert A M \vert^2 =\hat{A}^2 S M :  SM  \geq \alpha^2 \vert M \vert^2.
	\end{align}
	Thus, \eqref{eq:computation1matrix} and \eqref{eq:computation2matrix} imply that  
	\begin{align*}
		\vert A^2 - D^2 \vert^2 & = \vert A (A-D) + (A-D) D \vert^2
		= \vert A (A-D) \vert^2 + \vert (A-D) D\vert^2 + 2 (A-D) D: A (A-D)                                 \\
		                        & \geq \left(\alpha^2+2 \alpha \tilde \alpha+\tilde{\alpha}^2\right)|A-D|^2
		=|\alpha+\tilde{\alpha}|^2|A-D|^2 .
	\end{align*}
	We conclude by taking the square root.
\end{proof}

\subsection{Nonlinear Korn inequalities and time-dependent Dirichlet boundary conditions}\label{sec:korn}
To derive coercivity of the strain rates in the case \ref{Case2-newDissipation}, we rely on a generalization of the   celebrated nonlinear rigidity estimate by {\sc Friesecke, James and Müller} \cite{FrieseckeJamesMueller:02} obtained by {\sc Ciarlet and Mardare} \cite{CiarletMardare},
while one uses linearized versions (nonlinear Korn inequalities \cite{Neff,Pompe}) in the case \ref{Case1-oldDissipation}, see \cite{RBMFMK,MielkeRoubicek20Thermoviscoelasticity}.
More precisely, we revisit these estimates in the versions of \cite[Lemma~4.2]{Machillpvisco} and \cite[Theorem 3.1 (i)]{RBMFLM} by implementing time-dependent Dirichlet boundary conditions.  
We first state the main result of this section.
\begin{theorem}[Nonlinear Korn inequalities for time-dependent Dirichlet boundary conditions]\label{coercivitystrainrates}
	Let $y_i \in \Wid{t_i}$ such that $\det (\nabla y_i ) > \mu$ and $\Vert \nabla y_i \Vert_{  W^{1,p}(\Omega)} \leq 1/\mu$ for some   $\mu>0$, $i = 1,2$,  and some $0\leq t_1 \leq t_2 \leq T$. Assume that $h \in W^{1,\infty}( I ; W^{2,p}(\Omega;\R^d))$ with $  \inf_{I \times \Omega}\det(\nabla h)  \geq \mu > 0$. Then there exist   constants $C_\mu=C(\Omega, \Gamma_D, h, \mu)>0$ and   $\tilde C = C(\Omega, \Gamma_D, h)>0$ such that
	\begin{subequations}
		\begin{equation}\label{nonlinkornwithtimeboundary}
			\Vert y_i - y_j \Vert_{W^{1,2}(\Omega)}  \leq C_\mu   \Vert (\nabla y_i)^T \nabla y_i - (\nabla y_j)^T \nabla y_j \Vert_{L^2(\Omega)} + \tilde C \vert t_j - t_i\vert,
		\end{equation}
		\begin{equation}\label{pompewithtimeboundary}
			\Vert y_i - y_j \Vert_{W^{1,2}(\Omega)}  \leq C_\mu   \Vert \sym \big( (\nabla y_i ) ^T (\nabla y_i - \nabla y_j) \big) \Vert_{L^2(\Omega)} + \tilde C \vert t_j - t_i\vert.
		\end{equation}
	\end{subequations}
\end{theorem}
Notice that
\eqref{pompewithtimeboundary} boils down to   a  classical Korn's inequality  (with nonconstant Dirichlet boundary conditions) if $\nabla y_i = \Id$,
while \eqref{nonlinkornwithtimeboundary} becomes an estimate in terms of the Green-St-Venant strain tensor.
As shown in \eqref{calculationhelpful}, the right-hand side of \eqref{nonlinkornwithtimeboundary} only differs from the right-hand side in \eqref{pompewithtimeboundary} by a term of lower order.  

  Estimate \eqref{nonlinkornwithtimeboundary} is proved in \cite[Theorem 3(b)]{CiarletMardare} without considering  time-dependent boundary conditions and without showing that  the constant $C_\mu$ can be chosen uniformly among function such that $\det (\nabla y_i ) > \mu$ and $\Vert \nabla y_i \Vert_{  W^{1,p}(\Omega)} \leq 1/\mu$ for some   $i = 1,2$, $\mu>0$.
The proof of \eqref{nonlinkornwithtimeboundary} relies on the following result while we will use \cite[Theorem 3.1 (i) ]{RBMFLM} for \eqref{pompewithtimeboundary}.
\begin{lemma}[Generalized rigidity estimate]\label{nonlinearKornboundary}
	Let $t \in I$ and let $y_2 \in \Wid{t}$ be such that $\det (\nabla y_2 ) \geq \mu >0$. Then there exists a constant $C=C(\Omega, y_2)$ such that for every $y_1 \in W^{2,p}(\Omega;\R^d)$ with $\det(\nabla y_1) \geq \mu >0$ there exists a rotation $R(y_1) \in SO(d)$ such that we have
	\begin{align}\label{ineq:Ciarlet-Mardare1}
		\Vert \nabla y_1 - R(y_1) \nabla y_2 \Vert_{L^2(\Omega)}  \leq C  \Vert  (\nabla y_1)^T \nabla y_1 - (\nabla y_2)^T \nabla y_2 \Vert_{L^2(\Omega)}.
	\end{align}

\end{lemma}
\begin{proof}
	The statement is proved in \cite[Theorem~1(b)]{CiarletMardare}.
\end{proof}
The general idea for the proof of \eqref{nonlinkornwithtimeboundary} is to replace the rotation in \eqref{ineq:Ciarlet-Mardare1} with $\Id$ by using the boundary conditions. Then we can estimate the right-hand side of  \eqref{ineq:Ciarlet-Mardare1} by the right-hand side of \eqref{nonlinkornwithtimeboundary}.
To this end, we first show a generalization of  \cite[Lemma 3.3]{MaNePe} to control the norm of matrices by boundary integrals.
We define $S_0 \subset \Gamma_D$ as the set of points $x \in \Gamma_D$ such that $\mathcal{H}^{d-1}(\Gamma_D \cap B_\eps(x) ) >0$ for all $\eps >0$.
\begin{lemma}\label{gendalmaso}
	Let $\Gamma_D \subset \R^d$ be a bounded $\mathcal{H}^{d-1}$-measurable set with $0 < \mathcal{H}^{d-1}(\Gamma_D) < + \infty$. Moreover, let $h \in W^{1,\infty}(I; W^{2,p}(\Omega;\R^d))$ be such that
	${\rm dim} ({\rm aff} (h(t)(S_0) )  ) \geq d-1 $ for all $t \in I$, where ${\rm aff} (h(t)(S_0) )  $ is the affine space generated by linear combinations of points in $h(t)(S_0)$.
	Then there exists a constant $C= C(\Omega, \Gamma_D, h) >0$ such that it holds that
	\begin{align}\label{integralestimate}
		\vert A \vert^2 \leq C \min\limits_{\zeta \in \R^d}\int_{\Gamma_D} \vert A h(t) - \zeta \vert^2 \, {\rm d}\mathcal{H}^{d-1} \quad \text{for any } A \in K \text{ and } t \in I.
	\end{align}
	where $K \subset \R^{d \times d}$ is a closed cone satisfying $ {\rm dim } ( {\rm ker} (A)  )< d-1$ for $0 \neq A \in K$.
\end{lemma}

\begin{proof}

	The proof consists of two parts. We first show the estimate for a fixed $t \in I$ (Step 1). Then, we show that the constant $C>0$ can be chosen to be independent of $t$ (Step 2).  

	\emph{Step 1:} Fixing $t\in I$, we  follow the lines of \cite[Lemma~3.3]{MaNePe}. By elementary arguments one can check that the minimizer $\zeta$ in \eqref{integralestimate} is attained for $\zeta = \fint_{\Gamma_D} A h(t) \, {\rm d} \mathcal{H}^{d-1}$. We argue by contradiction. Suppose that for every $k \in \N$ there exists a matrix $A_k\in K$ such that
	\begin{align*}
		\frac{\vert A_k \vert^2}{k} >  \int_{\Gamma_D} \vert A_k h(t) - \zeta_k \vert^2 \, {\rm d}\mathcal{H}^{d-1} \qquad \text{ for } \qquad \zeta_k =  \fint_{\Gamma_D} A_k h(t) \, {\rm d} \mathcal{H}^{d-1}.
	\end{align*}
	As the sequence $(A_k)_k$ lies in a cone, we can assume that $\vert A_k \vert = 1$ by normalization. Since the cone is closed, by passing to a subsequence we find $A \in K$ with $\vert A \vert = 1$ such that $A_k \to A$ as $k \to \infty$ and $\int_{\Gamma_D} \vert A h(t) - \zeta \vert^2 \, {\rm d}\mathcal{H}^{d-1} = 0$ for $\zeta =  \fint_{\Gamma_D}  A h(t) \, {\rm d} \mathcal{H}^{d-1}$.
	In particular, $A h(t)(x) = const$ for all $x \in S_0$. By linearity this implies that $Ay = const$ for all $y \in {\rm aff}(h(t)(S_0))$.
	  Therefore,  ${\rm dim }  ( {\rm ker} (A)  )\geq d-1$. Due to the assumption on the cone, we have $A = 0$, which is impossible as $\vert A \vert = 1$.

	\emph{Step 2:}
	By normalization, Step 1 shows the existence of a minimal constant $C_t = C_t(h,\Omega,\Gamma_D)>0$ such that
	\begin{align}\label{eq:minimalconstant}
		1  \leq C_t  \int_{\Gamma_D} \Big| B h(t) - \fint_{\Gamma_D} B h(t) \, {\rm d} \mathcal{H}^{d-1}   \Big|^2 \, {\rm d}\mathcal{H}^{d-1}
	\end{align}
	for all $B = \frac{A}{\vert A \vert}$ with $0 \neq A \in K$.
	We show that $t \mapsto C_t$ attains its maximum on the interval $I$.
	Given $\delta >0$ with $t_1,t_2 \in I$ and $\vert t_1 - t_2 \vert < \delta$, the reverse triangle inequality,   $\vert B \vert \leq 1$, a trace estimate, and the Lipschitz regularity of $h$ imply that
	\begin{align*}
		 & \quad \ \left\vert  \int_{\Gamma_D} \vert B h(t_1) - \fint_{\Gamma_D} B h(t_1) \, {\rm d} \mathcal{H}^{d-1}   \vert^2 \, {\rm d}\mathcal{H}^{d-1} - \vert B h(t_2) - \fint_{\Gamma_D} B h(t_2) \, {\rm d} \mathcal{H}^{d-1}   \vert^2 \, {\rm d}\mathcal{H}^{d-1}  \right\vert \\
		 & \leq  C \Vert h \Vert_{L^\infty(I; L^\infty(\Omega;\R^d))}  \Vert h(t_1) - h(t_2) \Vert_{W^{1,1}(\Omega;\R^d)} \leq  C \delta \Vert h \Vert_{L^\infty(I; L^\infty(\Omega;\R^d))}  \Vert \partial_t h \Vert_{L^\infty(I; W^{2,p}(\Omega;\R^d))}.
	\end{align*}
	This shows the continuity of $t \mapsto C_t$ as $C_t$ is the minimal constant in \eqref{eq:minimalconstant}.
	As $I$ is a compact interval, the maximum is attained by the extreme value theorem, concluding the proof.  
\end{proof}

\begin{proof}[Proof of Theorem~\ref{coercivitystrainrates}]
	The proof is divided into four steps. In the first two steps, we prove \eqref{nonlinkornwithtimeboundary}.
	Step 1 adapts the abstract Lemma~\ref{gendalmaso} to our setting, arguing similarly to \cite[(3.14)]{MaNePe},
	while the estimate is shown in Step 2, up to the fact that the constant $C_\mu>0$ is uniform
	among functions in the set
	\begin{align}\label{def:weierstrassextremum}
		W_\mu \defas \Big\{ y_0 \in W^{2,p}(\Omega; \R^d) \colon \, \min\limits_{x \in \Omega} \det (\nabla y_0  (x)) \geq \mu, \ \Vert y_0 \Vert_{W^{2,p}(\Omega)} \leq \mu^{-1} \Big\}.
	\end{align}
	This  is shown in Step 3. Finally, Step 4 addresses \eqref{pompewithtimeboundary}.

	\emph{Step 1:}
	Let $K$ be the closed cone generated by $SO(d) - \Id$. Using the surjective exponential map $\exp\colon \R^{d \times d}_{\rm skew} \to SO(d)$, one can check that $K = {\rm cone}( SO(d) - \Id) \cup \R^{d \times d}_{\rm skew}$. Notice that for $0 \neq A \in \R^{d \times d}_{\rm skew}$ we have that ${\rm dim}({\rm ker} (A)) = d - {\rm rank}(A) \leq d-2 < d-1$. Therefore,
	${\rm dim}({\rm ker}(A))<d-1$ if $0 \neq A \in K$.
	In order to apply Lemma~\ref{gendalmaso}, it remains to check that
	${\rm dim}({\rm aff}\big(h(t)(S_0)\big) ) \geq d-1 $.
	For this purpose, let use check that the points $h(t)(S_0)$ have positive density, i.e., for all $y \in h(t)(S_0)$ it holds that
	\begin{align}\label{5X}
		\mathcal{H}^{d-1} (h(t)(S_0)  \cap B_\eps(y) ) > 0 \quad   \text{for all } \eps >0.
	\end{align}
	  As $\partial \Omega$ is Lipschitz and  $ h(t) \in W^{2,p}(  \Omega;\R^d)$ with $\det(\nabla   h) \geq \mu$, we observe that $h(t)(S_0)  \cap B_\eps(y)$ can be represented as a graph of a Lipschitz function. This along with $S_0 \subset \partial \Omega$ and $\mathcal{H}^{d-1}(S_0 \cap B_\eps(x) ) >0$ for all $\eps >0$ and $x \in S_0$ shows \eqref{5X}.  
	  Thus, Lemma~\ref{gendalmaso} implies that
	\begin{align}\label{ineq:boundary20}
		\vert Q - \Id \vert^2 \leq C \int_{\Gamma_D} \vert (Q - \Id) h(t_i) - \xi \vert ^2 \, {\rm d} \mathcal{H}^{d-1}  \quad  \text{for all } Q \in SO(d)  
	\end{align}
	for a constant $C = C(\Omega, \Gamma_D, h) >0$, where we choose $\xi = \vert \Omega \vert^{-1} \int_\Omega Qy_i - y_j \di x$.

	\emph{Step 2:}
	\eqref{ineq:Ciarlet-Mardare1} shows that there exists a rotation $  Q \in SO(d)$ such that we have
	\begin{align}\label{ineq:Ciarlet-Mardareproof}
		\Vert \nabla y_j - Q \nabla y_i \Vert_{L^2(\Omega)}  \leq C  \Vert  (\nabla y_j)^T \nabla y_j - (\nabla y_i)^T \nabla y_i \Vert_{L^2(\Omega)},
	\end{align}
	where $C>0$ depends on $y_i$.
	Using the boundary conditions on $\Gamma_D$, a trace estimate, Poincaré-Wirtinger inequality, \eqref{ineq:Ciarlet-Mardareproof}, and the regularity of $h$, we derive that
	\begin{align}
		\int_{\Gamma_D} \vert (Q-\Id)h(t_i) - \xi \vert ^2 \, {\rm d} \mathcal{H}^{d-1} & = \int_{\Gamma_D} \vert y_i - y_j + (Q-\Id)y_i - \xi + h(t_j) - h(t_i)\vert ^2 \, {\rm d} \mathcal{H}^{d-1} \notag                                                                     \\
		                                                                                & \leq C \int_{\Gamma_D} \vert  Q y_i - y_j  - \xi \vert ^2 \, {\rm d} \mathcal{H}^{d-1} + C \int_{\Gamma_D} \vert  h(t_j) -h(t_i)\vert ^2 \, {\rm d} \mathcal{H}^{d-1} \notag \\
		                                                                                & \leq C \int_{\Omega} \vert \nabla y_j - Q\nabla y_i  \vert ^2 \di x +  C  \int_{\Gamma_D} \vert  h(t_j) -h(t_i)\vert ^2 \, {\rm d} \mathcal{H}^{d-1} \notag                  \\
		                                                                                & \leq C \Vert  (\nabla y_j)^T \nabla y_j - (\nabla y_i)^T \nabla y_i \Vert_{L^2(\Omega)}^2 + \tilde C \vert t_j -t_i \vert^2. \label{ineq:proofboundary3}
	\end{align}
	Eventually, the triangle inequality,  \eqref{ineq:boundary20}, \eqref{ineq:Ciarlet-Mardareproof},   \eqref{ineq:proofboundary3},    and   $\Vert \nabla y_i \Vert_{L^\infty(\Omega)} \leq 1/\mu$ imply that
	\begin{align*}
		\int_\Omega \vert \nabla y_j - \nabla y_i \vert^2 \di x & \leq C \int_\Omega \vert \nabla y_j - Q \nabla  y_i \vert^2 \di x  +   C  \vert Q - \Id \vert^2 \int_\Omega \vert \nabla y_i \vert^2 \di x \\
		                                                        & \leq C \Vert  (\nabla y_j)^T \nabla y_j - (\nabla y_i)^T \nabla y_i \Vert_{L^2(\Omega)}^2 + \tilde C \vert t_j - t_i \vert^2.
	\end{align*}
	  This estimate along with Poincaré's inequality on $\Vert y_j - y_i - (h(t_i) - h(t_j)) \Vert _{L^2(\Omega)}$, the triangle inequality, and the regularity of $h$ yields estimate \eqref{nonlinkornwithtimeboundary}, where, for the moment, the constant $C$ still depends on $y_i$. We denote the minimal constant such that \eqref{nonlinkornwithtimeboundary} holds for fixed $y_i$ (and arbitrary $y_j \in W_\mu$) by $C_{y_i}>0$.

	\emph{Step 3:}
	We consider the compact space $(W_\mu,\Vert   \cdot  \Vert_{W^{1,\infty}(\Omega)})$ (see  \eqref{def:weierstrassextremum}) and show that   the mapping
	\begin{align}\label{eq:mapping}
		W_\mu \to \R_+, \quad y_i \mapsto C_{y_i}
	\end{align}
	is continuous.  Then, the extreme value theorem implies that the mapping in \eqref{eq:mapping} attains its maximum. (A similar argument has been employed in \cite[Theorem 3.3]{MielkeRoubicek20Thermoviscoelasticity}.)
	To show the continuity, consider $\tilde w , \hat w \in W_\mu $ such that $  \Vert  \hat w -  \tilde w \Vert_{W^{1,\infty}(\Omega)} < \delta$.
	  For arbitrary $y_j \in W_\mu$, the reverse triangle inequality shows that
	\begin{align*}
		\left\vert \Vert  y_j -  \tilde w \Vert_{W^{1,2}(\Omega)} -  \Vert y_j -  \hat w \Vert_{W^{1,2}(\Omega)} \right\vert \leq C \delta,
	\end{align*}
	as well as
	\begin{align*}
		 & \left\vert \Vert (\nabla y_j)^T \nabla y_j - (\nabla \tilde w)^T \nabla \tilde w \Vert_{L^2(\Omega)} -  \Vert (\nabla y_j)^T \nabla y_j - (\nabla \hat w)^T \nabla \hat w\Vert_{L^2(\Omega)}  \right\vert \\ &  \hspace{1.5cm} \leq  \Vert  (\nabla \tilde w)^T \nabla \tilde w - (\nabla \hat w)^T \nabla \hat w  \Vert_{L^2(\Omega)}  \\
		 & \hspace{1.5cm} = \Vert  (\nabla \tilde w - \nabla \hat w)^T \nabla \tilde w - (\nabla \hat w)^T ( \nabla \hat w - \nabla \tilde w ) \Vert_{L^2(\Omega)}    \leq C \mu^{-1}  \delta.
	\end{align*}
	This shows the continuity of the mapping in \eqref{eq:mapping}.  

	\emph{Step 4:}
	The proof of \eqref{pompewithtimeboundary} is similar: \cite[Theorem 3.1  (i) ]{RBMFLM} implies that there exists a constant matrix $A \in \R^{d \times d}_{\rm skew}$ such that
	\begin{align}\label{ineq:pompenew}
		\Vert (\nabla y_i -  \nabla y_j) - A \nabla y_i \Vert_{L^2(\Omega)}  \leq C   \Vert \sym \big( (\nabla y_i ) ^T (\nabla y_i - \nabla y_j) \big) \Vert_{L^2(\Omega)}.
	\end{align}
	(Note that this result is only proved in dimension $3$ as it was proved  in the context of a dimension-reduction problem,   but the arguments also hold in arbitrary dimensions.)
	Lemma~\ref{gendalmaso} (as before, $K$  is the closed cone generated by $SO(d) - \Id$) implies
	\begin{align}\label{ineq:boundary2}
		\vert A \vert^2 \leq C \int_{\Gamma_D} \vert A h(t_i) - \xi \vert ^2 \, {\rm d} \mathcal{H}^{d-1},
	\end{align}
	where we choose $\xi \defas \vert \Omega \vert^{-1} \int_\Omega (y_j -   y_i + A y_i) \di x$. Using the boundary conditions on $\Gamma_D$, a trace estimate, Poincaré-Wirtinger inequality, \eqref{ineq:pompenew}, and the regularity of $h$, we derive that
	\begin{align*}
		\int_{\Gamma_D} \vert A h(t_i) - \xi \vert ^2 \, {\rm d} \mathcal{H}^{d-1} & = \int_{\Gamma_D} \vert y_j - y_i + Ay_i - \xi - h(t_j) + h(t_i)\vert ^2 \, {\rm d} \mathcal{H}^{d-1}                                                                                 \\
		                                                                           & \leq C \int_{\Gamma_D} \vert y_j - y_i + Ay_i  - \xi \vert ^2 \, {\rm d} \mathcal{H}^{d-1} +  C  \int_{\Gamma_D} \vert  h(t_j) -h(t_i)\vert ^2 \, {\rm d} \mathcal{H}^{d-1} \\
		                                                                           & \leq C \int_{\Omega} \vert \nabla y_j - \nabla y_i + A \nabla y_i   \vert ^2 \di x + C  \int_{\Gamma_D} \vert  h(t_j) -h(t_i)\vert ^2 \, {\rm d} \mathcal{H}^{d-1}          \\
		                                                                           & \leq C \Vert \sym \big( (\nabla y_i ) ^T (\nabla y_i - \nabla y_j) \big) \Vert_{L^2(\Omega)}^2 + C \vert t_j -t_i \vert^2.
	\end{align*}
	  Now, repeating the argument at the end of Step 2,  Poincar\'e's inequality, the triangle inequality, \eqref{ineq:pompenew}, \eqref{ineq:boundary2},  $\Vert \nabla y_i \Vert_{L^\infty(\Omega)} \leq 1/\mu$, and the previous estimate  yield \eqref{pompewithtimeboundary}.
	An argument similarly to Step 3 shows that the constant can be chosen uniformly among functions $y_i \in W_\mu$, see \eqref{def:weierstrassextremum}.  
\end{proof}

\subsection{Discrete version of Gronwall's lemma}
We will need the following  Gronwall-type lemma.
\begin{lemma}[Discrete version of Gronwall's inequality]\label{lem:discretegronwall}
	Let $(a_l)_{l \ge 0}$ be a sequence, $\beta >0$ and $(b_l)_{l \ge 0} \geq 0$ such that $a_0 \le \beta$ and  
	\begin{align*}
		a_k \leq \beta + \sum_{l = 0}^{k-1} b_l a_l \quad \text{ for } k \geq 1.
	\end{align*}
	Then we have
	\begin{align*}
		a_k \leq 2\beta \exp \left( \sum_{l=0}^{k-1} 2b_l\right) \quad \text{ for } k \geq 0.
	\end{align*}
\end{lemma}
\begin{proof}
	The proof is standard if $a_l \geq 0$ for all $l \in \N$, see e.g.\ \cite[Lemma~2.11]{CesikStability}.
	The estimate still holds if $(a_l)_l$ attains values smaller than $0$. Indeed, as $(b_l)_l \geq 0$, it holds that
	\begin{align*}
		a_k \leq \beta + \sum_{l = 0}^{k-1}  b_l ( a_l \vee 0 ) \text{ for } k \geq 1.
	\end{align*}
	As the right-hand side is positive, the estimate remains true if we replace $a_k$ by $a_k \vee 0$ on the left-hand side.
	By using Gronwall's inequality for nonnegative coefficients, we conclude the statement.
\end{proof}

Notice that Lemma~\ref{lem:discretegronwall} only controls the positive part of $a_k$, i.e., there is no bound on $(-a_k \vee 0)_{k}$.

\section{Existence of time-discrete solutions}
In this section, we show that for   $\tau \in (0, \tau_0]$  the staggered time-discretization scheme   introduced in \eqref{mechanical_step}--\eqref{thermal_step}     is well-defined. As mentioned before, without further notice, we focus on the option \ref{Case2-newDissipation} and \eqref{htaucase2}. For convenience, we write
\begin{equation}\label{forces_mech_step}
	\langle \lst{k}, y \rangle
	\defas \int_\Omega \fst k \cdot y \di x
	+ \int_{\Gamma_N} \gst k \cdot y \di \haus^{d-1},
\end{equation}
where we recall that $\fst k  = \tau^{-1}   \int_{(k-1)\tau}^{k\tau} f(t) \di t$  and   $g^{(k)}_\tau =  \tau^{-1}  \int_{(k-1)\tau}^{k\tau} g(t) \di t$ and $k \in \{1, \ldots,  T / \tau\}$.
Further, the initial steps
$\yst 0 = y_{0} \in \Wid{0}$ and $\tst 0 = \theta_{0} \in L^2_+(\Omega)$   with $y_{0}$ and $\theta_{0}$ were defined in \eqref{initial_cond}.
Moreover,   recall that $f \in W^{1, 1}(I; L^2(\Omega; \R^d))$, $g \in W^{1, 1}(I; L^2(\Gamma_N; \R^d))$, $\theta_\flat \in   W^{  1,1  }  (I; L^2_+(\Gamma))$. We also recall the definition   of the total free energy potential in \eqref{eq: free energy}.  
We first investigate the existence of the $k$-th mechanical step.
\begin{proposition}[Mechanical step]\label{prop:existence_mechanical_step}
	Suppose that we have already constructed $ \yst{l} \in \Wid{l\tau}$ and $   \tst{l} \in L^2_+(\Omega)$  for $l =0, \dots, k-1$ for some $k \in   \{1, \ldots,  T / \tau \} $.
	Then, there exists $\tau_0 \in (0, 1]$ such that for all $\tau \in (0, \tau_0)$, the minimization problem \eqref{mechanical_step} is well-posed, i.e.,
	\begin{equation}\label{mechanical_step2}
		\min_{y \in \Wid{k\tau}} \Big\{
		\mechen(y) + \cplen\big(y, \tst{k-1})
		+ \frac{1}{\tau} \diss(\yst{k-1}, y , \tst{k-1})
		-     \langle \lst k , y \rangle
		\Big\}
	\end{equation}
	attains a solution.
	Furthermore, such a minimizer $\yst{k}$ solves the corresponding Euler-Lagrange equation, i.e.,  for all $z \in \Wzero  \defas  \{ z \in W^{2,p}(\Omega;\R^d) : z = 0 \ \text{on} \ \Gamma_D \}$     it holds  that
	\begin{align}\label{mechanical_step_single}
		 & \quad \ - \frac{1}{\tau}     \int_\Omega V(  C^{(k-1)}_\tau, \tst{k-1}) \Big[(  \nabla z)^T  \nabla \yst{k} +   (\nabla \yst{k})^T\nabla z ,  C^{(k)}_\tau   - C^{(k-1)}_\tau \Big] \di x \notag \\
		 & = \int_\Omega
		\pl_F \felpot(\nabla \yst{k}, \tst{k-1})   : \nabla z
		+ \pl_G \hypot(\nabla^2 \yst{k}) \cdddot \nabla^2 z \di x
		-     \langle \ell^{(k)}_\tau , z \rangle,
	\end{align}
	where we set  $  C^{(i)}_\tau  \defas (\nabla \yst{i})^T \nabla \yst{i}$ for $i = k-1,k$.
\end{proposition}
As a preparation we introduce the quantity
\begin{align}\label{def:subtractforce}
	E^{(k)}_\tau(y) \defas \mechen(y) - \sprod{\lst k}{y}
\end{align}
for $k \geq 1$ and  $E^{(0)}_\tau(y) \defas E^{(1)}_\tau(y)$.
Recalling the time-dependent boundary conditions in \eqref{Yid}, the following holds.
\begin{lemma}\label{lem:EFnew}
	There exists a constant $C > 0$ only depending on $g$, $f$, and $h$ such that for all  $k \in \setof{1, \ldots,  T / \tau}$  and all $y \in \Wid{t}$ it holds that
	\begin{equation*}
		|  \sprod{\lst k}{y} |
		\le \min \big\{ E^{(k)}_\tau(y),  \frac{1}{2}   \mechen(y) \big\}
		+  C.
	\end{equation*}
\end{lemma}

\begin{proof}
	Let $y \in \Wid{t}$.
	By \eqref{forces_mech_step}, a trace estimate, and Poincaré's inequality, we have
	\begin{align*}
		 & \quad \  |  \sprod{\lst k}{y} |                                                                                                                                                                                                                          \\
		 & \leq \left\vert \int_\Omega \fst k \cdot ( y - h(t) ) \di x
		+ \int_{\Gamma_N} \gst k \cdot ( y - h(t) ) \di \haus^{d-1} \right\vert +  \left\vert \int_\Omega \fst k \cdot   h(t) \di x
		+ \int_{\Gamma_N} \gst k \cdot  h(t) \di \haus^{d-1} \right\vert                                                                                                                                                                                            \\
		 & \leq \Vert f \Vert_{L^\infty(I;L^2(\Omega))} \big( \Vert y- h(t) \Vert_{L^2(\Omega)} + \Vert  h(t) \Vert_{L^2(\Omega)} \big) + \Vert g \Vert_{L^\infty(I;L^2(\Gamma_N))}  \big( \Vert y- h(t) \Vert_{L^2(\Gamma)} + \Vert h(t) \Vert_{L^2(\Gamma)} \big) \\
		 & \leq C \big( \Vert f \Vert_{L^\infty(I;L^2(\Omega))} +\Vert g \Vert_{L^\infty(I;L^2(\Gamma_N))} \big)  \left(\Vert  \nabla y \Vert_{L^2(\Omega)} +\Vert  h(t) \Vert_{W^{1,2}(\Omega)} \right)    .
	\end{align*}
	Since $2dq/(q-2)\geq 2$, \ref{W_lower_bound}  and Young's inequality give
	\begin{align*}
		\int_\Omega \vert \nabla y \vert^2 \di x \leq C \int_\Omega W^{\rm el}( \nabla y ) \di x + C \leq C \mechen(y) + C.
	\end{align*}
	Combining the previous estimates with Young's inequality we thus get
	\begin{align*}
		|  \sprod{\lst k}{y} |  \leq \frac{1}{2} \mechen (y) + \tilde C,
	\end{align*}
	where $\tilde C>0$ depends on $f$, $g$, $h$, and the constants $c_0$ and $C_0$ in \ref{W_lower_bound}. Eventually, the triangle inequality implies that
	\begin{align*}
		\mechen(y) \leq E^{(k)}_\tau(y) + |  \sprod{\lst k}{y} | \leq E^{(k)}_\tau(y) +  \frac{1}{2} \mechen (y) + \tilde C.
	\end{align*}
	By absorbing   $\frac{1}{2}\mechen (y)$ on the left-hand side, the statement follows from the previous two inequalities.
\end{proof}

\begin{proof}[Proof of Proposition~\ref{prop:existence_mechanical_step}]
	  \textit{Step 1 (Existence):}  We first address compactness  in   $W^{2, p}(\Omega; \R^d)$.  In the sequel, we use the notation $h^{(k)}_\tau =  h(k\tau)$, which is unrelated to the one in \eqref{htaucase2}.   In view of \eqref{Yid}, $ \tilde y_\tau^{(k-1)} \defas  h^{(k)}_\tau \circ (h^{(k-1)}_\tau)^{-1} \circ \yst{k-1} \in \mathcal{Y}_{h^{(k)}_\tau}$ is a valid competitor for the problem satisfying, by the chain rule,  
	\begin{align}\label{computationgradient2}
		\nabla \tilde y_\tau^{(k-1)}(x) & = \big((\nabla h^{(k)}_\tau)\circ (h^{(k-1)}_\tau)^{-1}  \circ \yst{k-1} \big)(x) \big(\nabla (h^{(k-1)}_\tau)^{-1} \circ \yst{k-1} \big)(x) \nabla \yst{k-1} (x) \notag \\
		                                & \eqcolon A^k_\tau(x) B^k_\tau (x)  \nabla \yst{k-1} (x)
	\end{align}  
	for a.e.\  $ x \in \Omega$, where by the inverse function theorem it holds that   $ (B^k_\tau)^{-1}(x)  =  ( \nabla  h^{(k-1)}_\tau  ) \circ (h^{(k-1)}_\tau)^{-1} \circ \yst{k-1}  ( x)    $ for a.e.\ $x \in \Omega$.
	Thus, we obtain
	\begin{align}\label{comparison_yn}
		 & \quad \ E^{(k)}_\tau(y_n)
		+ \frac{1}{\tau} \diss(\yst{k-1}, y_n ,  \tst{k-1})                                                                                                                                                                                                     \\
		 & \leq E^{(k)}_\tau(\yst{k-1}) + \vert \cplen(\yst{k-1}, \tst{k-1}) -  \cplen(y_n, \tst{k-1}) \vert + \vert  \langle \lst k ,  y_\tau^{(k-1)} -\tilde y_\tau^{(k-1)} \rangle \vert \notag                                                              \\
		 & \  + \vert \mechen(\tilde y_\tau^{(k-1)} ) - \mechen( y_\tau^{(k-1)} ) \vert + \vert \cplen\big(\tilde y_\tau^{(k-1)} , \tst{k-1}) - \cplen\big( y_\tau^{(k-1)} , \tst{k-1}) \vert + \frac{1}{\tau} \diss(\yst{k-1}, \tilde y_\tau^{(k-1)} ). \notag
	\end{align}
	We show that all the terms   on the right-hand side except for   $E^{(k)}_\tau(\yst{k-1})$ are small. Specifically,   
	the second term will be controlled via the dissipation, while the remaining terms are of order $\tau$, estimated via the regularity of $h$.  
	In view of the positivity of the determinants, see
	\eqref{Yid}, the Polar decomposition theorem yields rotations $S_n = S_n(x),  S^{(k-1)} =  S^{(k-1)}(x) \in SO(d)$ such that $\nabla y_n = S_n \sqrt{C_n}$ a.e.~in $\Omega$ for $C_n \defas (\nabla y_n)^T \nabla y_n$ and $\nabla \yst{k-1} =  S^{(k-1)} \sqrt{ C^{(k-1)}_\tau}$ a.e.~in $\Omega$ for $ C^{(k-1)}_\tau = (\nabla \yst{k-1})^T \nabla \yst{k-1}$.
	Thus, \ref{C_frame_indifference}, \ref{C_lipschitz}, the invariance of the Frobenius norm under rotations, and \eqref{ineq:matrixcalc} imply that
	\begin{align}\label{est:cplinproof}
		 & \vert W^{\rm cpl}(\nabla y_n, \tst{k-1}) - W^{\rm cpl}( \nabla \yst{k-1}, \tst{k-1}) \vert =   \vert W^{\rm cpl}(\sqrt{C_n}, \tst{k-1}) - W^{\rm cpl}(\sqrt{C^{(k-1)}_\tau}, \tst{k-1}) \vert \notag \\
		 & \leq C_0 (1+ \vert \nabla y_n \vert + \vert \nabla \yst{k-1} \vert ) \vert \sqrt{C_n} - \sqrt{C^{(k-1)}_\tau} \vert \notag                                                                           \\
		 & \leq C_0 (1+ \vert \nabla y_n \vert + \vert \nabla \yst{k-1} \vert ) (\alpha_{n,1} + \alpha^{(k-1)}_{1} )^{-1}\vert C_n -  C^{(k-1)}_\tau \vert
	\end{align}
	for a.e.~$x \in \Omega$, where $\alpha_{n,i}$ and $ \alpha^{(k-1)}_{i}$, $i=1,\ldots,d$, denote the eigenvalues of $\sqrt{C_n}$ and $\sqrt{ C^{(k-1)}_\tau}$, respectively, in increasing order.  
	Due to the invariance of the determinant and the Frobenius norm under rotations we find
	\begin{align*}
		\vert \det(\nabla y_n) \vert = \det (\sqrt{C_n} ) = \Pi_{i=1}^d \alpha_{n,i} \leq \alpha_{n,1} \vert \sqrt{C_n} \vert^{d-1} = \alpha_{n,1} \vert \nabla y_n \vert^{d-1}.
	\end{align*}
	Let $c_0>0$ be the constant in \eqref{bound_V1V2dissipation}.
	Then, an integration of \eqref{est:cplinproof}, Young's inequality with a constant $(  c_0 /\tau)^{1/2}$ and powers $2$,  $\alpha_{n,1} + \alpha^{(k-1)}_{1} \geq \alpha_{n,1}$, \eqref{coppler}, and \eqref{dissipation} imply that
	\begin{align}\label{estimate:complicated}
		\big|
		\cplen(  y_n, \tst{k-1})
		- \cplen(  \yst{k-1}, \tst{k-1})
		\big|
		 & \leq  \frac{1}{2\tau} \diss(\yst{k-1}, y_n ,  \tst{k-1})                                                                                                                  \\
		 & \  + \frac{\tau  C_0^2}{2  c_0 } \int_\Omega (1+ \vert \nabla y_n \vert + \vert \nabla \yst{k-1} \vert )^2 (\alpha_{n,1} + \alpha^{(k-1)}_{1} )^{-2} \di x \notag \\
		 & \leq  \frac{1}{2\tau} \diss(\yst{k-1}, y_n ,  \tst{k-1}) \notag                                                                                                           \\ &  \ \  \  + \frac{\tau  C_0^2}{2  c_0 } \int_\Omega (1+ \vert \nabla y_n \vert + \vert \nabla \yst{k-1} \vert )^2 \frac{\vert \nabla y_n \vert^{2d-2}}{\vert \det(\nabla y_n ) \vert^2} \di x. \notag
	\end{align}
	We employ Young's inequality (first with powers $d$ and $d /(d-1)$ and then with powers $ q / (q - 2)$ and $ q/2$ for $q \ge \frac{pd}{p-d}$), and use the inequality $(a+b)^z \leq 2^{z-1} (a^z + b^z)$ for $z \geq 1$ which together
	with \ref{W_lower_bound} yields
	\begin{align}\label{estimatecomplicated2}
		 & \int_\Omega (1+ \vert \nabla y_n \vert + \vert \nabla \yst{k-1} \vert )^2 \frac{\vert \nabla y_n \vert^{2d-2}}{\vert \det(\nabla y_n ) \vert^2} \di x
		\leq C \int_\Omega (1+ \vert \nabla y_n \vert^{2d} + \vert \nabla \yst{k-1} \vert^{2d} )  \frac{1}{\vert \det(\nabla y_n ) \vert^2} \di x \notag                                          \\
		 & \leq C\int_\Omega  \big(     \vert \nabla y_n \vert^{2d q /(q-2)} + \vert \nabla \yst{k-1} \vert^{2d q /(q-2)}   +  \vert \det(\nabla y_n ) \vert^{-q}  \big) \di x + C \notag \\
		 & \leq C\int_\Omega W^{\rm el} (\nabla y_n)  \di x  + C\int_\Omega W^{\rm el} (\nabla \yst{k-1})  \di x + C ,
	\end{align}
	where the constant depends on $c_0$, $C_0$, $d$, $q$, and $\Omega$.

	Notice that $\nabla h$ and $\nabla h^{-1}$ are uniformly bounded in $L^\infty(I\times \Omega)$, see \eqref{Yid}.
	Recalling \eqref{computationgradient2}, multiplication with $(B^k_\tau)^{-T}$ from the left and $(B^k_\tau)^{-1}$ from the right, we have
	\begin{align}\label{helpfulcomputation2}
		\vert (B^k_\tau)^T (A^k_\tau)^T A^k_\tau B^k_\tau - \Id \vert \leq  C \vert (A^k_\tau)^T (A^k_\tau -   (B^k_\tau)^{-1}) + ((A^k_\tau)^T -(B^k_\tau)^{-T}) (B^k_\tau)^{-1}  \vert \leq C \vert A^k_\tau - (B^k_\tau)^{-1} \vert
	\end{align}
	a.e.\ in $\Omega$.
	By \eqref{bound_V1V2dissipation},
	\eqref{computationgradient2}, \eqref{helpfulcomputation2}, the fundamental theorem of calculus, Young's inequality, \ref{W_lower_bound}, $s^2 \le 1+s^r$ for all $s \ge 0$,  and the regularity of $h$ we find that
	\begin{align}\label{dissipationboundary}
		\frac{1}{\tau} \diss(\yst{k-1}, \tilde y_\tau^{(k-1)} ) & \leq C \tau^{-1} \int_\Omega \vert (\nabla \yst{k-1})^T ((B^k_\tau)^T (A^k_\tau)^T A^k_\tau B^k_\tau - \Id )\nabla \yst{k-1} \vert^2 \di x \notag                                                  \\
		                                                        & \leq C \tau \Vert \partial_t \nabla h \Vert_{L^\infty(I\times \Omega)} \int_\Omega   \vert \nabla \yst{k-1} \vert^2 \di x   \leq C\tau  \int_\Omega W^{\rm el} (\nabla \yst{k-1})  \di x + C \tau.
	\end{align}
	Further, similarly to \eqref{helpfulcomputation2}, we have
	\begin{align*}
		\vert \nabla \tilde y_\tau^{(k-1)} - \nabla y_\tau^{(k-1)} \vert \leq \vert \nabla y_\tau^{(k-1)} \vert \vert A^k_\tau - (B^k_\tau)^{-1} \vert
	\end{align*}
	a.e.\ in $\Omega$.
	Using \ref{C_lipschitz}, we thus proceed as in \eqref{dissipationboundary} and obtain
	\begin{align}\label{couplingeasy}
		\vert \cplen\big(\tilde y_\tau^{(k-1)} , \tst{k-1}) - \cplen\big( y_\tau^{(k-1)} , \tst{k-1}) \vert \leq C\tau  \int_\Omega W^{\rm el} (\nabla \yst{k-1})  \di x + C \tau.
	\end{align}
	As $f \in L^\infty(I;L^2(\Omega))$ and $g \in L^\infty(I; L^2(\Gamma_D))$, $ \vert h^{(k-1)}_\tau (z) - h^{(k)}_\tau (z) \vert\leq \tau \Vert\partial_t h (z) \Vert_{L^\infty(I)}$ for all  $z \in \overline{\Omega}$,  the regularity of $h$ implies that
	\begin{align}\label{forces2estimate}
		\vert  \langle \lst k ,  y_\tau^{(k-1)}  - \tilde y_\tau^{(k-1)}  \rangle \vert \leq C\tau.
	\end{align}
	Notice that we have $\vert \det(A^k_\tau B^k_\tau) - 1 \vert \leq C (1+ \vert A^k_\tau \vert^{d-1} + \vert B^k_\tau \vert^{d-1} ) \vert A^k_\tau - (B^k_\tau)^{-1} \vert$.
	In view of \ref{W_lipschitz}, the uniform bounds of  $\nabla h$ and $(\nabla h)^{-1}$, and the fact that their determinants are bounded away from zero we obtain a.e.\ in $\Omega$  
	\begin{align*}
		 & \quad \big| W^{\rm el}_2\big(\det(A^k_\tau B^k_\tau  \nabla \yst{k-1}  ) \big) -  W^{\rm el}_2\big(\det(\nabla \yst{k-1}  ) )\big) \big|            \\
		 & \leq C (1 + \vert \det (  \nabla \yst{k-1} ) \vert^{-q-1} \vert )  \det (A^k_\tau B^k_\tau \nabla \yst{k-1}  ) - \det ( \nabla \yst{k-1}   )  \vert \\
		 & \leq C (\vert \det (  \nabla \yst{k-1} )^{-q} \vert + \vert \det (  \nabla \yst{k-1}  ) \vert)  \vert A^k_\tau - (B^k_\tau)^{-1} \vert  .
	\end{align*}
     Note that by \ref{W_lower_bound} the right-hand side can be controlled by $C\tau  \int_\Omega W^{\rm el} (\nabla \yst{k-1})  \di x + C \tau$. 
	  Next, \ref{H_regularity} and \ref{H_bounds} imply that $ \vert H(G) - H(\tilde G) \vert \leq C (1 + \vert G \vert^{p-1} + \vert \tilde G \vert^{p-1}) \vert G - \tilde G \vert$ for all $G, \tilde G \in \R^{d\times d \times d}$.
	The term involving $W^{\rm el}_1$ is estimated as in \eqref{couplingeasy} by using  \ref{W_lipschitz}, which then yields	\begin{align}\label{mechanicestimatetough}
		\vert \mechen(\tilde y_\tau^{(k-1)} ) - \mechen( y_\tau^{(k-1)} ) \vert \leq C\tau  \int_\Omega W^{\rm el} (\nabla \yst{k-1})  \di x + C \tau .
	\end{align}
	  Eventually, \eqref{def:subtractforce}, Lemma~\ref{lem:EFnew}, \eqref{comparison_yn}, \eqref{estimate:complicated}--\eqref{estimatecomplicated2}, and \eqref{dissipationboundary}--\eqref{mechanicestimatetough} imply that  
	\begin{align}\label{eq:corollaryforlater}
		E^{(k)}_\tau(y_n)
		+  \frac{1}{  2\tau }  \diss(\yst{k-1}, y_n ,  \tst{k-1})
		 & \leq E^{(k)}_\tau(\yst{k-1}) + C \tau \big(1+ E^{(k)}_\tau(y_n) + E^{(k-1)}_\tau(\yst{k-1}) \big).
	\end{align}
	Choosing $\tau_0>0$ such that $C \tau_0 < 1/2$, Lemma~\ref{lem:EFnew} and Lemma~\ref{lem:pos_det} then show the desired coercivity in $W^{2,p}(\Omega;\R^d)$ for $\tau \leq \tau_0$.
	The functional is weakly lower semicontinuous on $W^{2,p}(\Omega;\R^d)$ by the convexity of $\hypot$, see \ref{H_regularity}, the compact embedding $W^{2,p}(\Omega; \R^d) \subset W^{1,\infty}(\Omega; \R^d)$, and the continuity of $\elpot$,   $\cplpot$, and  $D^2$. This proves the existence of a minimizer in $\Wid{t}\subset W^{2,p}(\Omega;\R^d)$.

	  \textit{Step 2 (Euler-Lagrange equation):}  The   treatment of the convex term $\hypot$ is standard due to   \ref{H_bounds}.    
	Next, we compute the G\^ateaux derivative of $y \mapsto \frac{1}{\tau} \diss(\yst{k-1}, y , \tst{k-1})$ in $\Wid{k \tau}$. Let $z \in W^{2,p}_{\Gamma_D}(\Omega;\R^d)$.  Recalling that $ C^{(k-1)}_\tau = (\nabla \yst{k-1})^T \nabla \yst{k-1}$ and $ C^{(k)}_\tau = (\nabla \yst{k})^T \nabla \yst{k}$, we find
	\begin{align*}
		 & \qquad \lim\limits_{\delta \to 0} \frac{1}{\delta}   \left(\frac{1}{\tau} \diss(\yst{k-1}, \yst{k} + \delta z , \tst{k-1}) - \frac{1}{\tau} \diss(\yst{k-1}, \yst{k}
		, \tst{k-1}) \right)                                                                                                                                                                                                 \\
		 & = \frac{1}{2\tau} \lim\limits_{\delta \to 0} \frac{1}{\delta}   \Bigg( \int_\Omega V( C^{(k-1)}_\tau, \tst{k-1}) \Big[(\nabla \yst{k} + \delta \nabla z)^T (\nabla \yst{k} + \delta \nabla z) -  C^{(k-1)}_\tau , \\
		 & \qquad \qquad  \qquad \qquad \qquad \qquad  \qquad \qquad(\nabla \yst{k} + \delta \nabla z)^T (\nabla \yst{k} + \delta \nabla z) -  C^{(k-1)}_\tau \Big] \di x                                                    \\
		 & \qquad \qquad   - \int_\Omega V( C^{(k-1)}_\tau, \tst{k-1}) [C^{(k)}_\tau  -  C^{(k-1)} _\tau,  C^{(k)}_\tau  -  C^{(k-1)}_\tau] \di x   \Bigg)                                                                   \\
		 & = \frac{1}{\tau}     \int_\Omega V( C^{(k-1)}_\tau, \tst{k-1}) [(  \nabla z)^T  \nabla \yst{k} +   (\nabla \yst{k})^T\nabla z , C^{(k)}_\tau  -  C^{(k-1)}_\tau] \di x ,
	\end{align*}
	where we have used the symmetry of $V$ in the last step.
	The G\^ateaux differentiability of the mechanical energy relies on the uniform bound on gradients and the control on the determinant, see \eqref{pos_det}.
	We refer  to \cite[Proposition 3.2]{MielkeRoubicek20Thermoviscoelasticity} and  \cite[Proposition 4.1]{MielkeRoubicek20Thermoviscoelasticity} for   further   details on the other terms.
\end{proof}
From the previous proof, we directly deduce  the following.   For notational convenience, we extend $f,g$ by setting $f(t) =g(t) = 0$ for $t < 0$.  
\begin{lemma}[Bound on mechanical energy and dissipation]\label{lem:bad_mechen_bound}
	Let  $\tau_0 \in (0,1]$,  $k \in   \{1, \ldots,  T / \tau \} $, $\yst{k}\in \Wid{k\tau}$ , $\yst{k-1}\in \Wid{(k-1)\tau}$, and $\tst{k-1}\in L^2_+(\Omega)$
	be as in Proposition~\ref{prop:existence_mechanical_step}.  Then, for  $\tau \in (0, \tau_0)$ it holds that
	\begin{align*}
		 & \quad \   E^{(k)}_\tau(\yst{k})
		+ \frac{1}{   2\tau }  \diss(\yst{k-1}, \yst{k} ,  \tst{k-1})                                                                                                                                                                                 \\
		 & \quad\quad \quad \quad \quad \quad \quad \leq  \frac{1+C \left(\tau + \int_{(k-2)\tau }^{k\tau} \Vert f'(t) \Vert_{L^2(\Omega)} + \Vert g'(t) \Vert_{L^2(\Gamma_N)} \di t \right)}{1-C\tau}  E^{(k-1)}_\tau(\yst{k-1})+ C \tau .
	\end{align*}
\end{lemma}

\begin{proof}
	From the previous proof, see \eqref{eq:corollaryforlater}, we deduce that
	\begin{align}\label{ineq:badbalance}
		(1- C\tau) E^{(k)}_\tau(\yst{k})
		+ \frac{1}{  2\tau } \diss(\yst{k-1}, \yst{k} ,  \tst{k-1})
		\leq E^{(k)}_\tau(\yst{k-1})+ C \tau \big(1 + E^{(k-1)}_\tau(\yst{k-1}) \big).
	\end{align}  
	Observe that 
	\begin{align}\label{forcecomputation}
		E^{(k-1)}_\tau(\yst{k-1}) - E^{(k)}_\tau(\yst{k-1}) & =   \langle \lst{k}, \yst{k-1} \rangle -   \langle \lst{k-1}, \yst{k-1} \rangle                                                   \\
		                                                    & = \langle \lst{k} - \lst{k-1}, \yst{k-1} - h((k-1)\tau)  \rangle  +    \langle \lst{k} - \lst{k-1}, h((k-1)\tau)  \rangle  \notag
	\end{align}
	  for $k \geq 2$, while the difference for $k = 1$ vanishes by definition, see \eqref{def:subtractforce}.
	We proceed with estimating the  forces.    
	By the fundamental theorem of calculus and the Minkowski's integral inequality it holds that
	\begin{align}\label{estimatebodyforce}
		\Vert  \fst k - \fst {k-1} \Vert_{L^2(\Omega)} & = \left(\int_\Omega  \bigg( \tau^{-1} \int_{(k-1)\tau}^{k\tau} f(t) \di t -  \tau^{-1} \int_{(k-2)\tau}^{(k-1)\tau} f(t) \di t \bigg)^2 \di x \right)^{1/2} \notag \\
		                                               & \leq \int_{(k-2)\tau}^{k\tau} \Vert f'(t) \Vert_{L^2(\Omega)} \di t.
	\end{align}
	Similarly, we derive that
	\begin{align}\label{estimateNeumannforce}
		\Vert  \gst k - \gst {k-1} \Vert_{L^2(\Gamma_N)} \leq \int_{(k-2)\tau}^{k\tau} \Vert g'(t) \Vert_{L^2(\Gamma_N)} \di t.
	\end{align}
	  As in the proof of Lemma~\ref{lem:EFnew}, we use \eqref{forcecomputation}, Hölder's inequality, a trace estimate, and Poincaré's inequality to obtain for $k \ge 2$  
	\begin{align*}
		 & \quad \ \left\vert  E^{(k-1)}_\tau(\yst{k-1}) - E^{(k)}_\tau(\yst{k-1}) \right\vert \\
		 & \leq
		C  \Big(  \Vert\fst k - \fst {k-1} \Vert_{L^2(\Omega)} +  \Vert  \gst k - \gst {k-1} \Vert_{L^2(\Gamma_N)} \big) \big( \Vert \nabla \yst{k-1} \Vert_{L^2(\Omega)} + \Vert h \Vert_{L^\infty(I;W^{1,2}(\Omega))} \big).
	\end{align*}
	Thus,  \eqref{estimatebodyforce}, \eqref{estimateNeumannforce}, Young's inequality, \ref{W_lower_bound}, and the regularity of $h$ give
	\begin{align*}
		\left\vert  E^{(k-1)}_\tau(\yst{k-1}) - E^{(k)}_\tau(\yst{k-1}) \right\vert  \leq C \big(1+ \mechen(\yst{k-1}) \big) \int_{(k-2)\tau}^{k\tau} \big(\Vert f'(t) \Vert_{L^2(\Omega)} + \Vert g'(t) \Vert_{L^2(\Gamma_N)}  \big)\di t .
	\end{align*}
	Thus,   \eqref{ineq:badbalance} and Lemma~\ref{lem:EFnew} yield, for  $k \ge 1$,  
	\begin{align*}
		 & \quad \  (1- C\tau) E^{(k)}_\tau(\yst{k})
		+  \frac{1}{2 \tau} \diss(\yst{k-1}, \yst{k} ,  \tst{k-1})                                                                                                                                         \\
		 & \leq E^{(k-1)}_\tau(\yst{k-1})+ C \left(\tau + \int_{(k-2)\tau}^{k\tau} \Vert f'(t) \Vert_{L^2(\Omega)} + \Vert g'(t) \Vert_{L^2(\Gamma_N)} \di t \right) \big(1 + E^{(k-1)}_\tau(\yst{k-1}) \big) .
	\end{align*}
	We divide the previous equation by $(1-C\tau)$, and  get the desired estimate, up to changing the constant $C$, where we recall that $\tau_0$ was chosen small enough such that $C\tau < \frac{1}{2}$.  
\end{proof}

\begin{remark}[A priori bounds in \cite{RBMFMK}]\label{rem:differenceapriori}
	In the setting of \cite{RBMFMK}, one proves   a priori bounds on the mechanical energy by summing up the weak formulations of the (time-discrete) mechanical and heat equations, tested with the discrete strain rate and the constant function, respectively. As our time-discrete approximation scheme leads to different time-discrete Euler-Lagrange equations, we therefore proceed differently and prove the mechanical bounds without resorting to the thermal step. As a consequence, a larger growth condition in \ref{W_lower_bound} is needed, allowing for a better control of the coupling potentials in \eqref{comparison_yn}, see the proof of Proposition~\ref{prop:existence_mechanical_step}. Indeed, in contrast  to \cite[Lemma~3.6]{RBMFMK}, there is no dependence on the temperature on the right-hand side of the bound in Lemma \ref{lem:bad_mechen_bound}.  
	Further, the (rather mild) assumption on the additive decomposition of the elastic energy density in \ref{W_regularity} is only used in estimate \eqref{mechanicestimatetough}, exploiting the Lipschitz estimate in \ref{W_lipschitz},     which ensures that the constant $C>0$ is independent of the previous time-step.  
\end{remark}
Given the minimizer $ \yst{k} \in \Wid{k\tau}$ from Proposition~\ref{prop:existence_mechanical_step},
$\tst{k} \in L^2_+(\Omega)$ is defined as the minimizer of the functional in \eqref{thermal_step}.
Before we address its existence in Proposition~\ref{prop:existence_thermal_step} below, as a consequence of Lemma~\ref{lem:bad_mechen_bound}, we derive global bounds for $\yst{k}$
that are independent of $k \in \{   0, \ldots,   T/\tau   \} $.  For $l \in \N$, and a sequence $(a_k)_{k\geq 0}$, we define the discrete differences as
\begin{align}\label{discretedifference}
	\ddif a_l \defas \frac{a_l - a_{l-1}}{\tau} \quad \text{for } l \ge 1.  
\end{align}
\begin{theorem}[Global mechanical bounds]\label{thm:apriori_toten_velo_bound}
	  Suppose that  $\yst 0, \ldots,   \yst{T / \tau}$  from Proposition~\ref{prop:existence_mechanical_step} exist Then,  for all $k \in \{   0, \ldots,   T/\tau   \} $ we  have that
	\begin{subequations}\label{apriorimech}
		\begin{align}
			\mechen(\yst{k}) & \leq C, \label{mechfirstbound}     \\
			\sum_{l=1}^{k}  \tau \Vert \ddif  \yst l \Vert_{W^{1,2}(\Omega)}^2 \di x
			                 & \leq C, \label{velocityfirstbound}
		\end{align}
	\end{subequations}
	  where $C>0$ is a constant only depending on   the data $f,g,h$ and $T$.  
\end{theorem}

\begin{proof}
	  By Lemma~\ref{lem:bad_mechen_bound} we have
	\begin{align*}
		 & \qquad  E^{(l)}_\tau(\yst{l})
		+ \frac{1}{  2\tau }  \diss(\yst{l-1}, \yst{l} ,  \tst{l-1})                                                                                                                                                  \\
		 & \leq  E^{(l-1)}_\tau(\yst{l-1}) +2C \left(C\tau  + \int_{(l-2)\tau  }^{l\tau} \Vert f'(t) \Vert_{L^2(\Omega)} + \Vert g'(t) \Vert_{L^2(\Gamma_N)} \di t \right)   E^{(l-1)}_\tau(\yst{l-1})+ C \tau .
	\end{align*}
	for $l = 1, \dots, k$, where we used that $\tau \in (0,\tau_0]$, $C\tau_0 <1/2$, and $\frac{1}{1-C\tau} \le C\tau$, possibly passing to a bigger constant.     
	Summing the above inequality over all $l = 1,...,k$, we derive that
	\begin{align}\label{ineq:summedup}
		 & \quad \ E^{(k)}_\tau(\yst{k}) +  \sum_{l =1}^k    \frac{1}{  2\tau } \diss(   \yst{l-1}, \yst{l} ,  \tst{l-1} )                                                                                                                                    \\
		 & \leq E^{(0)}_\tau(\yst{0}) +  \sum_{l = 1}^{k}  2 C\left(  C \tau + \int_{(l-2)\tau }^{l\tau} \Vert f'(t) \Vert_{L^2(\Omega)} + \Vert g'(t) \Vert_{L^2(\Gamma_N)} \di t \right)    E^{(l-1)}_\tau(\yst{l-1}) + C T . \notag
	\end{align}
	Using the nonnegativity of $\diss$ and the  discrete version of Gronwall's lemma,  see Lemma~\ref{lem:discretegronwall}, we find that
	\begin{align}\label{ineq:mechenforcebound}
		E^{(k)}_\tau(\yst{k}) \leq 2 \big( \vert E^{(0)}_\tau(\yst{0}) \vert + CT \big) \exp\left( 4 C^2 T + 8 C  \Vert f' \Vert_{L^1(I;L^2(\Omega))} + 8 C   \Vert g' \Vert_{L^1(I;L^2(\Gamma_N))} \right).
	\end{align}
	  Then, Lemma~\ref{lem:EFnew} and  the regularity of $\yst{0}$, see \eqref{initial_cond}, imply that \eqref{mechfirstbound}.

	We now prove \eqref{velocityfirstbound}.  \eqref{mechfirstbound} allows to use Lemma~\ref{lem:pos_det}, implying that all the assumption of Theorem~\ref{coercivitystrainrates} are satisfied.
	Thus, by \eqref{nonlinkornwithtimeboundary} we deduce that there exists some $C>0$ such that
	\begin{align*}
		\Vert \yst{k} - \yst{k-1} \Vert_{W^{1,2}(\Omega)}^2  \leq C   \Vert (\nabla \yst{k})^T \nabla \yst{k} - (\nabla \yst{k-1})^T \nabla \yst{k-1} \Vert_{L^2(\Omega)}^2 + C \vert \tau \vert^2.
	\end{align*}
	In view of \eqref{bound_V1V2dissipation}, \eqref{ineq:summedup}, \eqref{ineq:mechenforcebound} and the regularity of $g$ and $f$  we thus find
	\begin{align*}
		\tau  \sum_{l=1}^k \tau^{-2} \Vert \yst{l} - \yst{l-1} \Vert_{W^{1,2}(\Omega)}^2  \leq CT + C  \sum_{l=1}^k \frac{1}{   2\tau }  \diss(  \yst{l-1}, \yst{l} ,  \tst{l-1} )
		\leq C.
	\end{align*}
	This concludes the proof.
\end{proof}

In the next lemma we discuss the well-definedness of the thermal step. Recall the definition of  $h_\tau$   in  \eqref{htaucase2}.

\begin{proposition}[Thermal step]\label{prop:existence_thermal_step}
	There exists $\tau_0 \in (0,1]$ such that for $\tau \in (0, \tau_0)$ and all $k \in \{1, \ldots,  T / \tau \} $ with  $\yst{k}\in \Wid{k\tau}$ , $\yst{k-1}\in \Wid{(k-1)\tau}$, $\tst{k-1}\in L^2_+(\Omega)$
	as in Proposition~\ref{prop:existence_mechanical_step},  
	the minimization problem \eqref{thermal_step} is well-posed on $H^1_+(\Omega)$.
	  More precisely,
	\begin{align*}
		\mathcal{T}(\theta)
		\defas & \int_\Omega \int_0^\theta \frac{1}{\tau} \big(
		\inten(\nabla \yst k, s)
		- \inten(\nabla \yst{k-1}, \tst{k-1})
		\big) \di s \di x
		+ \frac{1}{2} \int_\Omega \nabla \theta \cdot
		\hcm(\nabla \yst{k-1}, \tst{k-1}) \nabla \theta \di x                    \\
		       & -\int_\Omega h_\tau(\yst{k}, \yst{k-1}, \tst{k-1}) \theta \di x
		+ \frac{\kappa}{2} \int_\Gamma (\theta -      \btst k)^2 \di \haus^{d-1}
	\end{align*}
	is finite on $H^1_+(\Omega)$  and attains a unique minimizer $\tst k$ on $H^1_+(\Omega)$.
	Moreover, $\tst k$ satisfies
	\begin{align}\label{el_thermal_step}
		 & \int_\Omega \Bigg(
		\frac{\inten(\nabla \yst k, \tst k) - \inten(\nabla \yst{k-1}, \tst{k-1})}{\tau}
		  - h_\tau(\ysts{k}, \ysts{k-1}, \tsts{k-1})
		\Bigg) \vphi \di x \notag \\
		 & \quad + \int_\Omega
		\hcm(\nabla \yst{k-1}, \tst{k-1}) \nabla \tst k \cdot \nabla \vphi \di x
		+ \kappa \int_\Gamma (\tst k -       \btst k) \vphi \di \haus^{d-1} = 0
	\end{align}
	for any $\vphi \in H^1(\Omega)$.
\end{proposition}

\begin{proof}
	We divide the proof into three parts. In the first step, we prove that $\mathcal{T}(\theta)<+\infty$ for $\theta \in H^1_+(\Omega)$.
	In Step 2, the direct method is employed to prove the existence of a unique  minimizer in $H_+^1(\Omega)$. Finally, \eqref{el_thermal_step} is derived in Step 3.  

	\textit{Step 1 (Finiteness):}
	We start by showing that all terms of $\mathcal{T}$ are well-defined and integrable for $\theta \in H^1_+(\Omega)$.
	First, by \eqref{inten_lipschitz_bounds} we find that
	\begin{equation}\label{ul}
		\int_0^\theta \inten(\nabla \yst k, s) \di s
		\in [\tfrac{\ac}{2} \theta^2, \tfrac{\aC}{2} \theta^2]   \quad  \text{ and }\quad  \int_0^\theta \inten(\nabla \yst{k-1}, \tst{k-1}) \di s \le \aC \theta \tst{k-1}
	\end{equation}
	a.e.~on $\Omega$   which both are integrable by H\"older's inequality.   By Lemma~\ref{lem:bound_hcm}, $\hcm(\nabla \yst{k-1}, \tst{k-1})$ is well-defined in $\Omega$, and the corresponding term   in   $\mathcal{T}$ is   integrable.
	Moreover, by  \eqref{eq: strange deriv},  \eqref{C_locally_lipschitz}, \eqref{bound_V1V2dissipation}, and   \eqref{pos_det} it holds that $h_\tau(\yst k, \yst {k-1}, \tst {k-1}) \in L^\infty(\Omega)$, i.e., the third term is also well-defined.
	Finally, a trace estimate and the regularity of  $\theta_{\flat,\tau}^{(k)}$ complete  the proof of the well-definedness of $\tempen$.

	\textit{Step 2 (Existence):}
	The functional is coercive on $H^1_+(\Omega)$ due to $  \int_0^\theta \inten(\nabla \yst k, s) \di s   \ge \frac{\ac}{2} \theta^2$ by \eqref{ul}, the estimate $\nabla \theta \cdot \hcm(\nabla \yst{k-1}, \tst{k-1}) \nabla \theta \ge c_M |\nabla \theta|^2$ by \eqref{bound_hcm},   and   the fact that all other terms are either nonnegative or linear in $\theta$.
	Moreover, the functional is weakly lower semicontinuous on $H^1_+(\Omega)$.
	Indeed, the first two terms of $\mathcal{T}$ are strictly convex in $\theta$ and $\nabla \theta$.
	The term related to $h_\tau$ is linear, and the weak lower semicontinuity of the term involving $ \btst k$ follows again by (strict) convexity and the weak continuity of the trace operator in $H^1(\Omega)$.
	This shows that a   minimizer $\tst k \in H_+^1(\Omega)$ exists.  Due to the strict convexity of $\mathcal{T}$, this minimizer is indeed unique.  

	\textit{Step 3 (Euler-Lagrange equation):}
	In order to prove (\ref{el_thermal_step}) for test functions $\vphi \in H^1(\Omega)$ which are not constrained to be nonnegative, we extend the minimization problem \eqref{thermal_step} to possibly negative functions $\theta \in H^1(\Omega)$ and show that $\tst k$ minimizes $\mathcal{T}$ on $H^1(\Omega)$.
	To this end, one needs to prove that for all $\theta \in H^1(\Omega)$ it holds
	$ \mathcal{T}(\theta) \ge \mathcal{T}(\theta^+)  $
	where   $\theta^+ \defas \max\setof{\theta, 0}$ denotes the nonnegative part of $\theta$.
	Due to the similarities to \cite[Proposition~3.8]{RBMFMK}, we only address the terms
	involving $h_\tau$ and $\wst{k-1} \defas \inten(\nabla \yst{k-1}, \tst{k-1})$.

	By Lemma \ref{lem:pos_det}, $\nabla \yst{k-1}$ is invertible at every point in $\Omega$.
	Hence,by  \eqref{eq: strange deriv}, \eqref{C_locally_lipschitz}, and \eqref{pos_det}    we derive
	\begin{align}
		\big|\pl_C \hat W^{\rm cpl}((\nabla \yst{k-1})^T \nabla \yst{k-1}, \tst{k-1})\big|
		 & =   \frac{1}{2}   \Big|
		(
		\nabla \yst{k-1})^{-1}
		\pl_F \cplpot(\nabla \yst{k-1}, \tst{k-1}
		)\Big| \nonumber                            \\
		 & \leq   C   |  (\nabla \yst{k-1})^{-1}  |
		(\tst{k-1} \wedge 1)
		(1 +   |   \nabla \yst{k-1}  |  ) \notag    \\
		 & \leq  C (\tst{k-1} \wedge 1),
		\label{partial_C_cplpot_bound}
	\end{align}
	  where the constant $C>0$ only depends on the one given in Theorem \ref{apriorimech}.     By  \eqref{inten_lipschitz_bounds},  \eqref{partial_C_cplpot_bound}, Young's inequality, \eqref{bound_V1V2dissipation}, and  the fact that $t \wedge 1 \leq \sqrt{t}$ for all $t \geq 0$     it follows that
	\begin{align}
		 & \quad \vert  \tau^{-1} \partial_C \hat{W}^{\rm cpl} \big( (\nabla \ysts{k-1})^T \nabla \ysts{k-1}, \tsts{k-1} \big) : \big( (\nabla \ysts{k} )^T \nabla \ysts{k} - (\nabla \ysts{k-1})^T \nabla \ysts{k-1} \big) \vert \notag \\
		 &
		\leq c_0\frac{(\tst{k-1} \wedge 1)^2}{\tau} + C \tau^{-1} D^2(\nabla \ysts{k-1}, \nabla \ysts{k}) \leq\frac{\wst{k-1}}{\tau} + C\tau^{-1} D^2(\nabla \ysts{k-1}, \nabla \ysts{k}) . \label{inequalitylater}
	\end{align}
	In view of \eqref{htaucase2}, we choose $\tau_0$ sufficiently small such that $ C \leq \tau^{-1}$ for $\tau \in (0,\tau_0)$  and thus $\tau^{-1}\wst{k-1} + h_\tau  (\yst{k}, \yst{k-1}, \tst{k-1}) \geq 0$ a.e.~on $\Omega$.
	From this we deduce  
	\begin{align*}
		-   \int_\Omega \Big(
		\frac{\wst {k-1}}{\tau}
		+ h_\tau(\yst{k}, \yst{k-1}, \tst{k-1})
		\Big) \, \theta \di x
		\ge - \int_\Omega \Big(
		\frac{\wst {k-1}}{\tau}
		+ h_\tau(\yst{k}, \yst{k-1}, \tst{k-1})
		\Big) \, \theta^+ \di x.
	\end{align*}
	As in \cite[(3.14)--(3.16)]{RBMFMK}, one can check that $\theta_+$ is a better competitor for the remaining terms in $\mathcal{T}$.
	By taking first variations, we obtain \eqref{el_thermal_step}.
\end{proof}

\begin{theorem}[Global thermal bounds]\label{thm:thermal}
	   Suppose that  $\tst 0, \ldots,   \tst{T / \tau}$ from Proposition \ref{prop:existence_thermal_step} exist. Then, for all $k \in \{   0, \ldots,   T/\tau   \} $ we  have that
	\begin{subequations}\label{aprioritemp}
		\begin{align}
			\sup_{k = 1,\ldots, T/\tau} \Vert \tst{k} \Vert_{L^1(\Omega)}+ \sum_{k=0}^{T/\tau} \tau
			\int_\Omega \big( \abs{\tst k}^q + \abs{\wst k}^q \big) \di x & \leq C, \label{temp_inten_Lq_bound}            \\
			\sum_{k=1}^{T/\tau} \tau
			\int_\Omega \big( \abs{\nabla \tst k}^r + \abs{\nabla \wst k}^r \big) \di x
			                                                              & \leq C , \label{nablatemp_nablainten_Lr_bound} \\
			\sum_{k = 1}^{T/\tau} \tau \norm{\ddif \wst k}_{(H^{d}(\Omega))^*}
			                                                              & \leq C \label{dot_temp_apriori_bound}
		\end{align}
	\end{subequations}
	for any $q \in [1,  \frac{d+2}{d})$ and $r \in [1, \frac{d+2}{d+1})$, where $\wst k \defas \inten(\nabla \yst k, \tst k)$ and $\ddif $ is the discrete difference, see \eqref{discretedifference}.  
\end{theorem}

\begin{proof}
	We first test \eqref{el_thermal_step} with $\varphi = 1 $ which lies in $ H^1(\Omega)$ for every $k = 1, \ldots, T /\tau$. Thus, by summing the resulting equations, we find
	\begin{align*}
		\frac{1}{\tau} \sum_{l=1}^{k} \int_\Omega \wst{l} \di x & \leq  \frac{1}{\tau} \sum_{l=1}^{k} \int_\Omega \wst{l-1} \di x + \sum_{l=1}^{k} \int_\Omega  h_\tau(\ysts{l}, \ysts{l-1}, \tsts{l-1})  \di x \\
		                                                                  & \ \ \ - \kappa \sum_{l=1}^{k}  \int_\Gamma (\tst l -       \btst l)  \di \haus^{d-1}.
	\end{align*}
	Recalling the definition of $\btst l$ below \eqref{thermal_step} and using
	$\tst l \geq 0$, we thus get
	\begin{align}\label{forl1estimate}
		\int_\Omega \wst{k} \di x \leq  \int_\Omega \wst{0} \di x + \tau \sum_{l=1}^{k} \int_\Omega  h_\tau(\ysts{l}, \ysts{l-1}, \tsts{l-1})  \di x+   \kappa    \int_\Gamma     \int_I \theta_\flat(t) \di t  \di \haus^{d-1} .
	\end{align}
	The first and last term on the right-hand side can be bounded by using \eqref{inten_lipschitz_bounds}, \eqref{initial_cond}, and the fact that $\bt \in  L^{2} (I; L^2_+(\Gamma))$.
	Recalling \eqref{htaucase2}, we use a variant of the first inequality in \eqref{inequalitylater} by additionally including the factorization $1 = \tau \tau^{-1}$ in Young's inequality, and obtain  
	\begin{align}\label{htaucase2check}
		\vert h_\tau(\ysts{l}, \ysts{l-1}, \tsts{l-1}) \vert  \leq C \tau^{-2} D^2(\nabla \ysts{l-1}, \nabla \ysts{l})    + C.
	\end{align}
	Thus, \eqref{forl1estimate}, \eqref{bound_V1V2dissipation}, \eqref{calculationhelpful}, \eqref{pos_det},  \eqref{velocityfirstbound}, and \eqref{inten_lipschitz_bounds} imply that $\sup_{k} \Vert \tst{k} \Vert_{L^1(\Omega) } \leq C$.
	The derivation of the remaining bounds is delicate as $h_\tau$ is only bounded in $L^1$. Thus, one needs to employ special test functions developed by Boccardo and Gallouët for parabolic equations with a measure-valued right-hand side. Yet, at this stage of the proof one can almost verbatim follow the lines of \cite[Section 3.4]{RBMFMK}. Indeed, the only difference lies in the exact form of $ h_\tau$, but  an inspection of the proof shows that we only employ a  bound of the form \eqref{htaucase2check} on this term. We omit details and refer to \cite[Lemma 3.19, Theorem 3.20]{RBMFMK}.  
\end{proof}

\section{Convergence of time-discrete solutions}\label{sec:convtimediscrete}
In this section, we prove that the time-discrete solutions converge to a weak solution in the sense of Definition~\ref{def:weak_formulation}.
  We do not provide all details of the proofs, but focus on highlighting the differences compared to   \cite[Section 4]{RBMFMK}.  
Throughout this section, we again only handle the case \ref{Case2-newDissipation} and \eqref{htaucase2}.  
The following convergences for the interpolations defined in \eqref{y_interpolations} essentially follow from weak compactness arguments.
\begin{lemma}[Compactness]\label{lem:uniform_def_conv}
	There exists $y  \in L^\infty(I; W^{2,p}(\Omega;\R^d)) \cap  H^1(I; H^1(\Omega; \R^d))$ with $y(0, \cdot) =  y_0 $ and $y(t) \in \Wid{t}$ for a.e.\ $t \in I$ such that, up to a subsequence (not relabeled), it holds that
	\begin{subequations}
		\begin{align}
			\ay        & \weaklystar y \text{ weakly* in } L^\infty(I; W^{2,p}(\Omega;\R^d))                             &  & \text{and} &
			\ay        & \weakly y \text{ weakly in }  H^1(I;H^1(\Omega;\R^d)), \label{uniform_def_conv1}                                  \\
			\nabla \ay & \to \nabla y \text{ in } L^\infty(I; L^\infty(\Omega;\R^{d \times d}))\label{uniform_def_conv2}
		\end{align}
	\end{subequations}
	as $\tau \to 0$.
	In the first convergence of \eqref{uniform_def_conv1} and in \eqref{uniform_def_conv2}, the same holds true if we replace $\ay$ by $\py$ or $\ny$.
	Moreover, there exists $\theta \in L^1(I; W^{1,1}(\Omega))$ with $\theta \ge 0$ a.e.\ in $I \times \Omega$ such that, up to a subsequence (not relabeled), it holds that
	\begin{subequations}\label{pointwise_temp_conv}
		\begin{align}
			\nt & \rightharpoonup \theta                                                                                               &   & \text{and} &
			\nw & \rightharpoonup w                                                                                                    &
			    & \text{weakly in } L^r(I; W^{1,r}(\Omega)) \text{ for any } r \in [1, \tfrac{d+2}{d+1}), \label{pointwise_temp_conv1}                                  \\
			\at & \to \theta                                                                                                           &   & \text{and} & \aw & \to w &
			    & \text{in } L^s(I \times \Omega)  \text{ for any } s \in [1,  \tfrac{d+2}{d}), \label{pointwise_temp_conv2}
		\end{align}
	\end{subequations}
	as $\tau \to 0$,
	where we define $\aw$, $\pw$ or $\nw$ as the interpolations of $\wst k = \inten(\nabla \yst k, \tst k)$, similarly to \eqref{y_interpolations}.
	The limit $w$ satisfies the identification $w = \inten(\nabla y, \theta)$ a.e.~in $I \times \Omega$, and the convergences
	in \eqref{pointwise_temp_conv2} also hold if we replace $\at$ with $\pt$ or $\nt$ and $\aw$ with $\pw$ or $\nw$, respectively.  
\end{lemma}

\begin{proof}
	The statement is proved as in  \cite[Lemma~4.1, Lemma~4.2]{RBMFMK} using that the  time-discete approximations enjoy the same a priori bounds, compare \eqref{apriorimech} and \eqref{aprioritemp} with \cite[Lemma~3.18 and Theorem~3.20]{RBMFMK}.  
	Here, we highlight that we choose $\tau_0 \in (0,1)$ as the minimal $\tau_0$ of Proposition~\ref{prop:existence_mechanical_step} and Proposition~\ref{prop:existence_thermal_step}.  To show that $y(t) $ satisfies the  (time-dependent) boundary conditions  $h(t)$,  we use \eqref{uniform_def_conv1},  $\yst{l}\in \Wid{l\tau}$, \eqref{Yid},  and  \eqref{y_interpolations} to estimate
	\begin{align*}
		\Vert y - h \Vert_{L^1(I;L^1(\Gamma_D))} & \leq \liminf\limits_{\tau \to 0} \Vert \ay  - h \Vert_{L^1(I;L^1(\Gamma_D))}                                                                                                                                   \\
		                                         & \leq \liminf\limits_{\tau \to 0} \sum_{k=1}^{T/\tau} \int_{(k-1)\tau}^{k\tau} \int_{\Gamma_D}  \vert h((k-1)\tau) - h  (t) \vert +   \vert h(k\tau) - h (t)  \vert  \di \haus^{d-1}  \di t \\
		                                         & \leq \liminf\limits_{\tau \to 0} \sum_{k=1}^{T/\tau} \int_{(k-1)\tau}^{k\tau} \int_{\Gamma_D}  2 \int_{(k-1)\tau}^{k\tau}  \vert \partial_t h(s) \vert \di s  \di \haus^{d-1}  \di t                           \\
		                                         & \leq \liminf\limits_{\tau \to 0} 2 \tau  \Vert \partial_t h \Vert_{L^1(I;L^1(\Gamma_D))} \to 0.
	\end{align*}
	  This concludes the  proof.  
\end{proof}
The previous lemma suffices to pass to the limit in the time-discrete mechanical evolution.
\begin{proposition}[Convergence of the mechanical equation]\label{thm:vanishing_tau_mech_nonlinear}
	Let $(y,\theta)$ be as in Lemma \ref{lem:uniform_def_conv}.
	Then, for any   test function   $z \in C^\infty(I \times \overline{\Omega})$ with $z = 0$ on $I \times \Gamma_D$ we have that \eqref{weak_limit_mechanical_equation} holds.
\end{proposition}

\begin{proof}
	The proof relies on \cite[Proof of Proposition 5.1, Step 2]{MielkeRoubicek20Thermoviscoelasticity} and \cite[Proposition~4.4]{RBMFMK}. Indeed,  the only  difference lies in the choice of the viscous term  in  the  time-discrete problem
	\eqref{mechanical_step_single}, compare \eqref{htaucase2} with \eqref{htaucase1}.  
	Compared to \cite{RBMFMK,MielkeRoubicek20Thermoviscoelasticity}, we give slightly more details on the application of Minty's trick for monotone operators which allows to handle the nonlinearity induced by $\partial_G H(\nabla^2 y_\tau^{(k)})$ in the limiting passage, see \eqref{mechanical_step_single}.
	To this end, on $ X \defas L^2(I;  \Wzero )$ we consider the functional  $\mathbf{H}$ defined by
	\begin{equation*}
		\langle \mathbf{H}(w), z\rangle = \intQ \pl_G \hypot\big(\nabla^2 w(t) + \nabla^2 h(t)\big) \cdddot \nabla^2 z  \di x \di t \qquad \text{for } w,z \in X.
	\end{equation*}
	  The incorporation of boundary conditions in $X$ is necessary as test functions in the weak formulation \eqref{weak_limit_mechanical_equation} require zero Dirichlet boundary conditions.   The addend  $\nabla^2 h$ ensures that the function $w(t) + h(t)$ lies in $\mathcal{Y}_{h(t)}$ for $t \in I$.       
	By Proposition \ref{prop:existence_mechanical_step}, there exists $b_{\tau} \in X^*$ defined via summation of \eqref{mechanical_step_single} for $k = 1,\ldots, T /\tau  $   and multiplication with $\tau$ such that
	$\langle \mathbf{H}(\ny - \overline{h}_\tau),z\rangle = \langle b_{\tau}, z \rangle$
	for all $z \in X $, where $\overline{h}_\tau$ denotes the piecewise constant-in-time approximation of $h$ (cf.\ \eqref{y_interpolations}) such that $\ny - \overline{h}_\tau \in X$.  (The notation is unrelated to the one in \eqref{htaucase2}.)  
	Defining $b \in X^*$ as  
	\begin{align*}
		\langle b,z\rangle & \defas     \intQ f \cdot z \di x \di t
		+   \int_I \int_{\Gamma_N} g \cdot z \di \haus^{d-1} \di t \\
		                   & \qquad - \intQ   \Big(
		\pl_F \felpot(\nabla y, \theta)
		+ \pl_{\dot F} \disspot(\nabla y, \nabla \partial_t y, \theta)
		\Big) : \nabla z \di x \di t ,
	\end{align*}
	  our goal is to  show that $b$ satisfies $ \langle \mathbf{H}(y-h),z\rangle = \langle b,z\rangle $
	for all $z \in X$,
	where $y$ is the limit from Lemma~\ref{lem:uniform_def_conv}. This   then implies that
	\eqref{weak_limit_mechanical_equation} holds (even for a larger class of test functions).

	Let us verify the assumptions for Minty's trick. $\mathbf{H}$ is a hemicontinuous and a monotone operator as $\hypot$ is convex due to \ref{H_regularity}.
	By Lemma~\ref{lem:uniform_def_conv}, it holds that $\ny \rightharpoonup y$ weakly in $X$.
	As $\langle \mathbf{H}(\ny - \overline{h}_\tau),z\rangle = \langle b_{\tau}, z \rangle$ for all $z \in X$,  we need to show that
	$b_{\tau} \weaklystar  b$ weakly* in $X^*$, and $\langle b_{ \tau} , \ny - \overline{h}_\tau \rangle \to \langle b , y - h \rangle $.
	For the terms involving $f$, $g$, and $\partial_F W$, we refer to \cite{RBMFMK,MielkeRoubicek20Thermoviscoelasticity},
	and we only prove convergence of the viscous part of $\langle b_\tau,z \rangle$,
	which we denote by
	\begin{align*}
		 & \langle b_\tau^{\rm vis} , z \rangle \defas  \int_I  \int_\Omega
		  \tau^{-1} V\big( \underline{C}_\tau, \pt \big) \Big[(  \nabla z  )^T  \nabla \ny +   (\nabla \ny)^T\nabla z  ,   \overline{C}_\tau  - \underline{C}_\tau \Big]  \di x    \di t,
	\end{align*}
	where we use $\underline{C}_\tau$ as the notation for the piecewise constant-in-time  approximation  similarly to \eqref{y_interpolations}.  Then, \eqref{calculationhelpful} and the triangle inequality yield
	\begin{align*}
		 & \quad \ \Bigg\vert \intQ \partial_{\dot F} R (\nabla y, \nabla  \partial_t y , \theta) : \nabla z \di x \di t - \langle b_\tau^{\rm vis} , z \rangle  \Bigg\vert                                                            \\
		 & \leq  \Bigg\vert \intQ \partial_{\dot F} R (\nabla y, \nabla \partial_t y , \theta) : \nabla z \di x \di t  - \int_I \int_\Omega
		V(  \underline{C}_\tau, \pt ) \Big[(  \nabla z )^T  \nabla \ny +   (\nabla \ny)^T\nabla z  ,                                                                                                                                     \\
		 & \qquad \qquad \qquad  \quad   \qquad \qquad  \qquad \qquad \qquad  \qquad \qquad \quad \qquad \left(\partial_t \nabla \ay \right)^{T} \nabla \py + (\nabla \py)^{T}\left(\partial_t \nabla \ay  \right)\Big] \di x   \di t \Bigg\vert \\
		 & \quad \quad \ \ + \Bigg\vert \int_I \int_\Omega
		V( \underline{C}_\tau, \pt  ) \Big[(  \nabla z )^T  \nabla \ny +   (\nabla \ny)^T\nabla z ,  \left(\partial_t \nabla \ay \right)^{T} \big(\nabla \ny-\nabla \py \big)  \Big]  \di x \di t \Bigg\vert .
	\end{align*}
	We show that both terms on the right-hand side converge to zero as $\tau \to 0$.
	As in \eqref{diss_rate},   \eqref{viscousstresstensor} and \ref{D_quadratic} yield
	\begin{align}\label{identityviscousstress}
		\pl_{\dot F} \disspot(F, \dot F, \theta) : G
		= V(F^TF, \theta) [ G^T F + F^T G  , \dot F^T F + F^T \dot F ] \quad \text{for } F, \dot F, G \in \R^{d \times d}.
	\end{align}
	By \eqref{identityviscousstress}, weak convergence of $(  \partial_t  \hat{y})_\tau$ in $L^2(I; H^1(\Omega; \R^{d \times d}))$ (see \eqref{uniform_def_conv1}), convergence of $(\nabla \ny)_\tau, (\nabla \py)_\tau$  on   $I \times \Omega$ (see \eqref{uniform_def_conv2}),   and pointwise a.e.~convergence of $(\pt)_\tau$   on  $I \times \Omega$ (up to a subsequence, see \eqref{pointwise_temp_conv2}), we get that the first term converges to zero.
	Using additionally that $\Vert \nabla \ny-\nabla \py  \Vert_{L^\infty(I ; L^\infty(\Omega; \R^{ d \times d}))} \to 0$, the second term also converges to zero.  Summarizing, we have checked that $b_{\tau} \weaklystar  b$ weakly* in $X^*$.

	Eventually, due to uniform convergence of the gradients and  $\overline{h}_\tau \to h$ in $L^\infty(I; W^{2,p}(\Omega;\R^d))$, we also obtain the convergence $\langle b_{ \tau} , \ny - \overline{h}_\tau \rangle \to \langle b , y - h \rangle $.  This concludes the proof.
\end{proof}

For the limit passage in the time-discrete heat equation, we will need   strong convergence of the strain rates $(\nabla  \partial_t \hat{y}  )_\tau$ in $L^2(I; L^2(\Omega;\R^{d \times d}))$ since the dissipation rate $\drate(\nabla \py, \nabla \partial_t \hat{y}, \pt)$ is quadratic in $\nabla \partial_t \hat{y}$.
Note that our a priori bounds currently only guarantee weak convergence.
The next lemma improves this convergence:
\begin{lemma}[Strong convergence of the strain rates]\label{lem:strong_strain_rates_conv}
	For $y$ as in Lemma \ref{lem:uniform_def_conv}, we have that, up to taking a subsequence,
	\begin{equation*}
		  \partial_t \hat{y} \to \partial_t y \text{ strongly in } L^2(I; H^1(\Omega; \R^d)) \text{ as } \tau \to 0.
	\end{equation*}
\end{lemma}

\begin{proof}
	The proof essentially follows from  \cite[Proposition~5.1, Proposition~6.4]{MielkeRoubicek20Thermoviscoelasticity} and \cite[Lemma~4.5]{RBMFMK}, additionally taking the time dependent boundary conditions into account. We provide the main steps and highlight the main difference to \cite[Lemma~4.5]{RBMFMK}:
	as we work purely with the mechanical equation, we mainly need to take care of the new time-discrete version of the dissipation in \eqref{mechanical_step_single}.
	First, in the time-continuous setting, one derives the energy balance
	\begin{equation}\label{cont_energy_balance}
		\begin{aligned}
			 & \quad \ \mechen(y(T))
			\hspace{-0.05cm} +   \hspace{-0.05cm} 2 \int_I    \int_\Omega R (\nabla y , \nabla\partial_t y , \theta )  \di x\di t \\& =  \mechen  ( y_{0}   )
			\hspace{-0.05cm}  +  \hspace{-0.1cm}  \int_I   \langle \ell(t),\partial_t y  \rangle \di t    \hspace{-0.05cm} -  \hspace{-0.1cm} \int_I  \int_\Omega
			\pl_F \cplpot(\nabla y , \theta ) : \nabla\partial_t y  \di x \di t + \Phi_h,  
		\end{aligned}
	\end{equation}
	where the term $\Phi_h$ is related to the boundary condition $h$ and takes the form  
	\begin{align}\label{Dhlimit}
		\hspace{-0.2cm} \Phi_h & =
		  \int_I \int_\Omega V( (\nabla y)^T \nabla y, \theta ) \Big[( \partial_t \nabla  h)^T   \nabla y +   (\nabla y)^T  \partial_t \nabla  h ,  ( \partial_t \nabla y)^T  \nabla y +   (\nabla  y)^T    \partial_t \nabla y \Big]  \di x \di t \\
		                       & \quad  + \int_I \int_\Omega \pl_F \cplpot(\nabla y, \theta)
		: \partial_t \nabla  h + \pl_G \hypot(\nabla^2 y)
		\cdddot \partial_t \nabla^2  h
		+\pl_F \elpot(\nabla y ) : \partial_t \nabla  h \di x \di t - \int_I  \langle \ell(t), \partial_t h \rangle \di t . \notag
	\end{align}
	This follows by testing the balance of momentum in \eqref{weak_limit_mechanical_equation} derived in Proposition \ref{thm:vanishing_tau_mech_nonlinear} with $\partial_t {y} - \partial_t h  \in L^2(I; H^1_0(\Omega;\R^d) )$,  employing \eqref{diss_rate} and \eqref{identityviscousstress},  and using a chain rule for the $\Lambda$-convex functional  $\mathcal{M}$, see \cite[Proposition 3.6]{MielkeRoubicek20Thermoviscoelasticity}.
	Our next goal is to show a similar balance in the time-discrete setting.   Using the notation $h^{(k)}_\tau =  h(k\tau)$, which is unrelated to the one in \eqref{htaucase2}, the test function $ \tau (\ddif \yst k - \ddif h^{(k)}_\tau)$ is admissible in the Euler-Lagrange equation \eqref{mechanical_step_single} of the $k$-th mechanical step, which gives  
	\begin{align}\label{energy_balance_discrete_stepk}
		 & 0 = - \int_\Omega V(C^{(k-1)}_\tau, \tst{k-1}) [( \ddif \nabla \yst k)^T  \nabla \yst{k} +   (\nabla \yst{k})^T  \ddif \nabla \yst k , C^{(k)}_\tau  - C^{(k-1)}_\tau ] \di x \notag \\
		 & \qquad  + \tau   \langle \lst k, \ddif \yst k \rangle
		- \tau \int_\Omega \pl_F \cplpot(\nabla \yst k, \tst{k-1})
		: \ddif \nabla \yst k \di x \notag                                                                                                                                                                \\
		 & \qquad \quad - \tau \int_\Omega \pl_G \hypot(\nabla^2 \yst k)
		\cdddot \ddif \nabla^2 \yst k
		-\pl_F \elpot(\nabla \yst k) : \ddif \nabla \yst k   \di x + \Phi_{h,\tau}^{(k)},  
	\end{align}
	where $\Phi_{h,\tau}^{(k)}$ consists of every term of the right-hand side with $\ddif \yst k $ replaced by $\ddif h^{(k)}_\tau$, i.e.,
	\begin{align*}
		\hspace{-0.2cm} \Phi_{h,\tau}^{(k)} & =
		\int_\Omega V(C^{(k-1)}_\tau, \tst{k-1}) [( \ddif \nabla  h^{(k)}_\tau)^T  \nabla \yst{k} +   (\nabla \yst{k})^T  \ddif \nabla  h^{(k)}_\tau ,  C^{(k)}_\tau  - C^{(k-1)}_\tau ] \di x   - \tau   \langle \lst k, \ddif h^{(k)}_\tau \rangle \\
		                                    & \quad  + \tau \int_\Omega \pl_F \cplpot(\nabla \yst k, \tst{k-1})
		: \ddif \nabla   h^{(k)}_\tau + \tau  \pl_G \hypot(\nabla^2 \yst k)
		\cdddot \ddif \nabla^2  h^{(k)}_\tau
		+\pl_F \elpot(\nabla \yst k) : \ddif \nabla  h^{(k)}_\tau  \di x . \notag
	\end{align*}
	By the $\Lambda$-convexity of  $\mathcal{M}$ (see \cite[Lemma~3.8]{BadalSchwarz}, \ref{H_regularity}, and  \eqref{apriorimech}),   we can find $\Lambda > 0$   independent of  $\tau$  and $k$ such that
	\begin{align*}
		\mechen(\yst {k-1})
		 & \ge \mechen(\yst k)
		- \Lambda \norm{\nabla \yst{k-1} - \nabla \yst k}_{L^2(\Omega)}^2
		+ \int_\Omega \pl_G \hypot(\nabla^2 \yst k)
		\cdddot (\nabla^2 \yst {k-1} - \nabla^2 \yst k) \di x           \\
		 & \phantom{\ge}\quad + \int_\Omega \pl_F \elpot(\nabla \yst k)
		: (\nabla \yst {k-1} - \nabla \yst k) \di x.
	\end{align*}  
	  Writing $V_{k-1} \defas V(C^{(k-1)}_\tau,  \tst{k-1})$ for shorthand, by \eqref{energy_balance_discrete_stepk} and \eqref{calculationhelpful} we thus derive
	\begin{align*}
		     & \mechen(\yst k) - \mechen(\yst{k-1})
		- \Lambda \tau^2 \norm{\ddif \nabla \yst k}^2_{L^2(\Omega)}                                                                                                                                                                                                \\
		     & +  \tau  \int_\Omega V_{k-1} \Big[( \ddif \nabla \yst k)^T  \nabla \yst{k} +   (\nabla \yst{k})^T  \ddif \nabla \yst k ,  ( \ddif \nabla \yst k)^T  \nabla \yst{k-1}  +   (\nabla  \yst{k-1})^T  \ddif \nabla \yst k\Big] \di x \\
		     & +   \tau \int_\Omega V_{k-1} \Big[( \ddif \nabla \yst k)^T  \nabla \yst{k} +   (\nabla \yst{k})^T  \ddif \nabla \yst k ,   ( \ddif \nabla \yst k )^T  \big(  \nabla \yst k - \nabla \yst {k-1} \big) \Big] \di x                                    \\
		\leq & \, \tau   \langle \lst k, \ddif \yst k \rangle
		- \tau \int_\Omega \pl_F \cplpot(\nabla \yst k, \tst{k-1}) : \ddif \nabla \yst k \di x    + \Phi_{h,\tau}^{(k)}.
	\end{align*}
	Summing the above inequality over $k \in \setof{1, \ldots,  T/\tau}$ we arrive at a discrete analog of \eqref{cont_energy_balance}, namely,
	\begin{align*}
		 & \mechen(\ny(T))
		-\Lambda \tau \int_I \int_\Omega
		\abs{\nabla \partial_t  \hat{y}_\tau }^2 \di x \di t \notag                                                                                                                                                                                                                                                                                                                                                           \\
		 & +  \int_I  \int_\Omega V(\underline{C}_\tau, \underline{\theta}_\tau ) \Big[( \nabla \partial_t \hat{y}_\tau)^T   \nabla \overline{y}_\tau +   ( \nabla \overline{y}_\tau)^T \nabla   \partial_t \hat{y}_\tau, ( \nabla \partial_t \hat{y}_\tau)^T  \nabla \underline{y}_\tau +   ( \nabla \underline{y}_\tau )^T \nabla \partial_t \hat{y}_\tau \Big] \di x \di t  \notag \\
		 & +   \int_I \int_\Omega V(\underline{C}_\tau, \underline{\theta}_\tau ) \Big[( \nabla \partial_t \hat{y}_\tau)^T   \nabla \overline{y}_\tau +   ( \nabla \overline{y}_\tau)^T  \nabla \partial_t \hat{y}_\tau , ( \nabla \partial_t \hat{y}_\tau)^T  \big(   \nabla \overline{y}_\tau -  \nabla \underline{y}_\tau \big) \Big] \di x \di t  \notag                                                   \\
		 & \quad \leq  \mechen   ( y_ {0}   )
		+  \int_I \langle \ell(t), \partial_t \hat{y}_\tau \rangle \di t
		- \int_I \int_\Omega \pl_F \cplpot(\nabla \ny, \pt) : \nabla \partial_t \hat{y}_\tau \di x \di t + \Phi_{h,\tau},  
	\end{align*}
	where in the integral for the   force   terms we used the definition in \eqref{forces_mech_step}, and as before $\underline{C}_\tau$ is defined analogously to $\underline{y}_\tau$ using the discretizations $(C^{(k)}_\tau)_k$, see \eqref{y_interpolations}. We also set $\Phi_{h,\tau} = \sum_{k=1}^{ T/\tau}\Phi^{(k)}_{h,\tau}$.

	A careful inspection of the proof of \cite[Lemma~4.5]{RBMFMK}, see \cite[(4.13)]{RBMFMK}, reveals that the energy balances are the same up to the  terms $\Phi_{h,\tau}$ and  
	\begin{align}\label{7Y}
		\int_I \int_\Omega V(\underline{C}_\tau, \tst{k-1}) [( \nabla \partial_t \hat{y}_\tau)^T  \ \nabla \overline{y}_\tau +   ( \nabla \overline{y}_\tau)^T  \nabla \partial_t \hat{y}_\tau , ( \nabla \partial_t \hat{y}_\tau)^T  \big(   \nabla \overline{y}_\tau -  \nabla \underline{y}_\tau \big) ] \di x \di t .  
	\end{align}
	Thus, strong convergence follows  exactly as in \cite[Lemma~4.5]{RBMFMK} once   we show that the previous term vanishes and  that $\Phi_{h,\tau} \to \Phi_h$ as $\tau \to 0$.  
	Using \eqref{uniform_def_conv2} for $\overline{y}_\tau$ and $\underline{y}_\tau$, respectively, we have  $ \Vert  \nabla \overline{y}_\tau -  \nabla \underline{y}_\tau \Vert_{L^\infty(I \times \Omega)} \to 0  $ as $\tau \to 0 $. Moreover,   \eqref{uniform_def_conv1}--\eqref{uniform_def_conv2}    imply that  the term in \eqref{7Y}  is controlled in absolute value by $ C\Vert  \nabla \overline{y}_\tau -  \nabla \underline{y}_\tau \Vert_{L^\infty(I \times \Omega)}$.  
	Denoting by $\hat h_\tau$ the piecewise affine interpolations of $h^{(k)}_\tau$, we have $\hat h_\tau \to h$ in $L^\infty(I; W^{2,p}(\Omega))$ and $\partial_t \hat h_\tau \to \partial_t h$ in $L^s(I \times \Omega)$ for all $s \in [1, +\infty)$ which follows from dominated convergence as the difference quotient $\big(\frac{h^{(k)}_\tau - h^{(k-1)}_\tau}{\tau}\big)_\tau$ is bounded in $L^\infty(I; W^{2,p}(\Omega))$ and converges for a.e.~$I \times \Omega$.
	This convergence is even stronger than the convergence of $(\hat y_\tau)_\tau$, giving $\Phi_{h,\tau} \to \Phi_h$.
	This concludes the proof.
\end{proof}

Finally, we prove that the heat equation is satisfied by the limiting variables $(y,\theta)$.

\begin{proposition}[Convergence of the heat-transfer equation]\label{thm:convergence_heat_vanishing_tau}
	Let $(y,\theta)  $ be as in Lemma \ref{lem:uniform_def_conv}.
	Then, for any  test function $\varphi \in C^\infty(I \times \overline \Omega)$ with $\varphi(T) = 0$,  we have that $(y  , \theta  )$ satisfies \eqref{weak_limit_heat_equation}.
\end{proposition}

\begin{proof}
	Our goal is to resort to the proof of \cite[Proposition~4.6]{RBMFMK}. Indeed, the  discrete weak formulations  \eqref{el_thermal_step} and \cite[(3.11)]{RBMFMK} differ only in the different definition of $h_\tau$, cf.\   \eqref{htaucase2} and \eqref{htaucase1}. Therefore, it suffices to shows that, along the evolution,  \eqref{htaucase2} and \eqref{htaucase1}  differ only by a term vanishing for $\tau \to 0$.  We address the dissipation and the coupling term separately.

	First, we  analyze the difference between $D^2$ in case \ref{Case2-newDissipation} and $\xi = 2R$ from \eqref{diss_rate}. (The latter coincides with $D^2$ in   option \ref{Case1-oldDissipation}.)
	Setting $C_i = F_i^TF_i$ for $F_i \in \R^{d\times d}$, $i =1,2$,
	\eqref{calculationhelpful}, \ref{D_quadratic}, \ref{D_bounds}, and triangular inequalities imply that
	\begin{align*}
		 & \quad \ \big|  V(C_1, \theta) [C_2 - C_1 , C_2 - C_1]  -  2 R(F_1, F_2 - F_1, \theta)\big|                                                                                    \\
		 & = \Big\vert  2 V(C_1, \theta) \big[C_2 - C_1 , ( F_2 - F_1)^T (F_2 - F_1) \big]  - V(C_1, \theta) \big[( F_2 - F_1)^T (F_2 - F_1), ( F_2 - F_1)^T (F_2 - F_1) \big] \Big\vert \\
		 & \le  2C_0   \vert C_2 - C_1 \vert \vert F_2 - F_1 \vert^2  + C_0 \vert F_2 - F_1 \vert^4                                                                            \\
		 & \le 2C_0  \vert (F_2-F_1)^T F_1 + F_1^T (F_2 - F_1)  \vert  \vert F_2 - F_1 \vert^2  + 3 C_0 \vert F_2 - F_1 \vert^4                                                \\
		 & \leq C_0  ( 7\vert F_1 \vert +  3  \vert F_2 \vert)  \vert F_2 - F_1 \vert^3   .
	\end{align*}
	Let $\vphi$ be as in the statement. The previous calculation, \eqref{uniform_def_conv1}, and \eqref{uniform_def_conv2} thus imply that
	\begin{align}\label{helpfullater}
		 & \quad \ \left\vert \int_I   \int_\Omega \Big( \tau^{-2} D^2(\nabla \underline{y}_\tau, \nabla \overline{y}_\tau , \underline{\theta}_\tau) - \tau^{-2} \drate(\nabla \underline{y}_\tau, \nabla \overline{y}_\tau -\nabla \underline{y}_\tau , \underline{\theta}_\tau) \Big)  \varphi \di x \di t  \right\vert \notag \\
		 & \quad \quad \quad  \leq C ( \Vert \py \Vert_{L^\infty(I; L^\infty(\Omega))} + \Vert \ny \Vert_{L^\infty(I; L^\infty(\Omega))} )  \Vert \partial_t \nabla \ay   \Vert^2_{L^2(I\times \Omega)} \Vert \ny - \py \Vert_{L^\infty(I; L^\infty(\Omega))} \to 0
	\end{align}
	as $\tau \to 0$. Concerning the coupling term, using  \eqref{calculationhelpful}  once again, we employ  \eqref{eq: strange deriv} and derive
	\begin{align*}
		 & \quad \tau^{-1} \pl_F \cplpot(\nabla \underline{y}_\tau
		, \underline{\theta}_\tau) : \big( \nabla \overline{y}_\tau -\nabla \underline{y}_\tau \big)   - \tau^{-1} \partial_C \hat{W}^{\rm cpl} \big( (\nabla \underline{y}_\tau)^T \nabla \underline{y}_\tau, \underline{\theta}_\tau \big) : \big( (\nabla \overline{y}_\tau )^T \nabla \overline{y}_\tau - (\nabla \underline{y}_\tau)^T \nabla \underline{y}_\tau \big) \\
		 & =   - \tau^{-1} \frac{1}{2} ( \nabla \underline{y}_\tau )^{-1}  \partial_F {W}^{\rm cpl} \big(  \nabla \underline{y}_\tau, \underline{\theta}_\tau \big) :   \big(\nabla \overline{y}_\tau -\nabla \underline{y}_\tau\big)^{T}\big(\nabla \overline{y}_\tau-\nabla \underline{y}_\tau\big)
	\end{align*}
	a.e.\ in $I \times \Omega$.  In view of \eqref{pointwise_temp_conv2}, we have $\pt \to \theta $ pointwise a.e.~in $I \times \Omega$, up to a subsequence.
	Invoking additionally \eqref{C_locally_lipschitz} and \eqref{pos_det}, the previous calculation allows us to proceed as in \eqref{helpfullater} to show that the term vanishes as $\tau \to 0$. Thus, the discrete weak formulations \eqref{el_thermal_step} and \cite[(3.11)]{RBMFMK}  differ only by a lower order term. The other terms are handled as in \cite{RBMFMK}. (At this point, we also explicitly use the convergences of $w_\tau$ in \eqref{pointwise_temp_conv}.)
\end{proof}

  Collecting the results of the  previous two sections, we are able to give the proof of Theorem~\ref{thm:van_tau}.
\begin{proof}[Proof of Theorem~\ref{thm:van_tau}]
	The existence of time-discrete solutions immediately follows from Proposition~\ref{prop:existence_mechanical_step} and Proposition~\ref{prop:existence_thermal_step}. This shows (i).
	The a priori estimates are addressed in Theorem~\ref{thm:apriori_toten_velo_bound} and Theorem \ref{thm:thermal}, allowing us to
	extract converging subsequences in the sense of \eqref{van_tau_y_conv} and \eqref{van_tau_theta_conv}, see Lemma~\ref{lem:uniform_def_conv} and Lemma \ref{lem:strong_strain_rates_conv}.    Finally, Proposition~\ref{thm:vanishing_tau_mech_nonlinear} and Proposition~\ref{thm:convergence_heat_vanishing_tau} guarantee that the limiting variables are a weak solution in the sense of Definition~\ref{def:weak_formulation}. This shows (ii).  
\end{proof}

\section{Numerical experiments}\label{sec:Example}
In order to compare the above results with experimental settings, we have also carried out numerical simulations by implementing the staggered time-discretization scheme in \eqref{mechanical_step} and \eqref{thermal_step}.
Those are performed for $d = 3$ considering plane-strain assumptions, i.e., we assume that the specimen is infinitely thick in the direction perpendicular to the plane. The higher-order regularization   $H$ in \eqref{mechanical_energy} has been neglected, i.e., $H = 0$, as this significantly simplifies the code.   This allows us to resort to   Q1 finite elements, approximating both the deformation and the temperature by piecewise-affine continuous functions. The model has been implemented in the C$++$ finite element code OOFEM, see, e.g., \cite{horak-etal, patzak} for details. Throughout this section, the variable $t$ represents integer values, smoothly interpolated for visualizations.

\subsection{Experiment 1 (Rigid body motions generate heat in \ref{Case1-oldDissipation})}\label{sec:Experiment1}
In this experiment, we compare  \ref{Case2-newDissipation} and \ref{Case1-oldDissipation} regarding the amount of energy that is dissipated between two consecutive mechanical minimization problems,  see \eqref{mechanical_step}.  We show that the (time-discrete) dissipation which is not dynamically frame indifferent leads to unphysical results, namely to heat production during a rigid motion.  

This effect is already visible in a simplified setting:
we suppose that the elastic energy is described in terms of the compressible neo-Hookean potential, i.e., 
\begin{equation}\label{neohookean}
	W^{\rm el} (F) = W_{\rm NH}(F) \defas 
	\begin{cases}
		\mu\Big(   \det(F)^{-2/3}| F|^2   - 3 \Big)^2 + \lambda\Big(   \det(F) + \frac{1}{\det(F)} - 2\Big)^2 & \text{if } \det(F) >0, \\
		+ \infty                                                                                                  & \text{otherwise},
	\end{cases}
\end{equation}
and  consider the coupling potential $W^{\rm cpl}  (F,\theta) = C_1\theta(1-\log(\theta))$ for a constant $C_1>0$,   which is independent of $F$. Here, $\mu>0$ is a shear modulus  and $\lambda>0$ is a bulk modulus of the material.   
We highlight that, although  the growth of $W_{\rm NH}$ does not satisfy the conditions in \ref{W_lower_bound}, it satisfies all qualitative features of the model,   i.e., $W_{\rm NH}$ is frame indifferent and penalizes strong compressions of the material.  
By  construction $\neohooke$ is minimized precisely on $SO(3)$ with the minimal value being zero.
As for the dissipation, we consider a viscosity tensor that is independent of $(\nabla y)^T \nabla y$ and $\theta$ of the form
\begin{align}\label{viscositytensor}
	V_{ijkl} = \frac{\nu }{2} ( \delta_{ik} \delta_{jl} + \delta_{il} \delta_{jk}),
\end{align}
where $\nu>0$ is the viscosity constant. Consequently, we have  
$R(F,\dot F, \theta) = \frac{  \nu}{2} \vert \dot{F}^T F + F^T \dot F \vert^2$.

We consider time-dependent Dirichlet-boundary conditions that prescribe a time-dependent rotation of the boundary of our reference domain $\Omega$.
More precisely, we  take a radially symmetric reference configurations $\Omega$, see Figure~\ref{fig:example_frameNonIvariance}, and define $h$ in \eqref{Yid} as
\begin{align}\label{boundary-rotation}
	h(t,x) = \begin{pmatrix}
		         \cos(t) & - \sin(t) & 0 \\
		         \sin(t) & \cos(t)   & 0 \\
		         0       & 0         & 1 \\
	         \end{pmatrix}
	\begin{pmatrix}
		x_1 \\
		x_2 \\
		x_3 \\
	\end{pmatrix}
	= \begin{pmatrix}
		  \cos(t) x_1 - \sin(t) x_2 \\
		  \sin(t) x_1 + \cos(t) x_2 \\
		  x_3                       \\
	  \end{pmatrix},
\end{align} representing a rotation of the material with rotation axis  $e_3$. Further, we assume that the external temperature $\theta_\flat$ is constant and neglect body forces  by setting $f = g = 0$ in \eqref{main_mechanical_eq} and \eqref{main_bc}.    The initial configuration  is assumed to be stress-free, i.e., $y_0(x) = h(0,x)$ with  $  W^{\rm el}  (\nabla y_0) = W_{\rm NH}({\rm Id}) = 0  $.
Given the initial temperature $\theta_0 = \theta_\flat$,   the deformation $y(t, x)=h(t,x)$ for $x\in\Omega$ and the constant temperature $\theta(x,t)=\theta_\flat(x)$   is  a solution to the system \eqref{main_mechanical_eq}--\eqref{main_thermal_eq}, i.e., the temperature of the material does not change during   time    and the body performs a rigid rotation in time. Here, we use that the system   does not dissipate energy (all terms on the right-hand sided of \eqref{main_mechanical_eq}--\eqref{main_thermal_eq} involving  $\nabla \partial_t y$ vanish)  since the right Cauchy-Green strain is always $\nabla y^T \nabla y =\Id$, and thus  independent of time. Moreover,    $\pl_F \felpot(\nabla y, \theta) = \partial_F W_{\rm NH}(\nabla y) =0$  as well.

Concerning the discrete schemes, we  first present the corresponding theoretical result.

\begin{figure}[htbp!]
	\centering
    \includegraphics[width=0.5\textwidth]{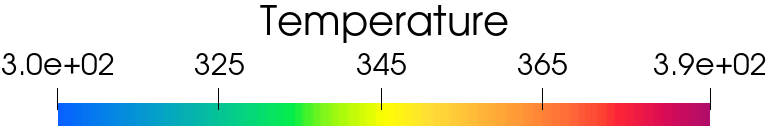} 
    \\
    \vspace{6mm}
	\begin{tabular}{ccc}
		\includegraphics[width=0.3\textwidth]{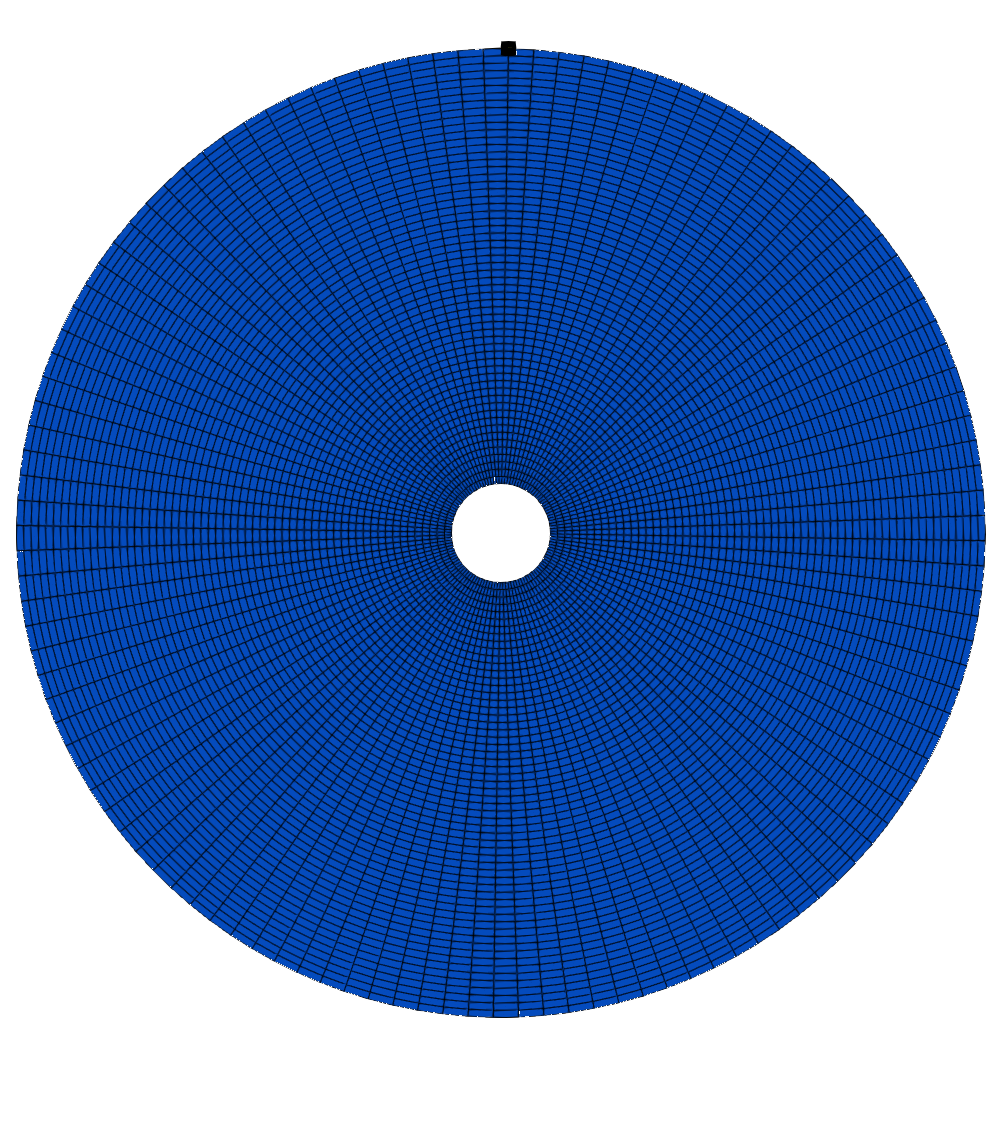} &
		\includegraphics[width=0.3\textwidth]{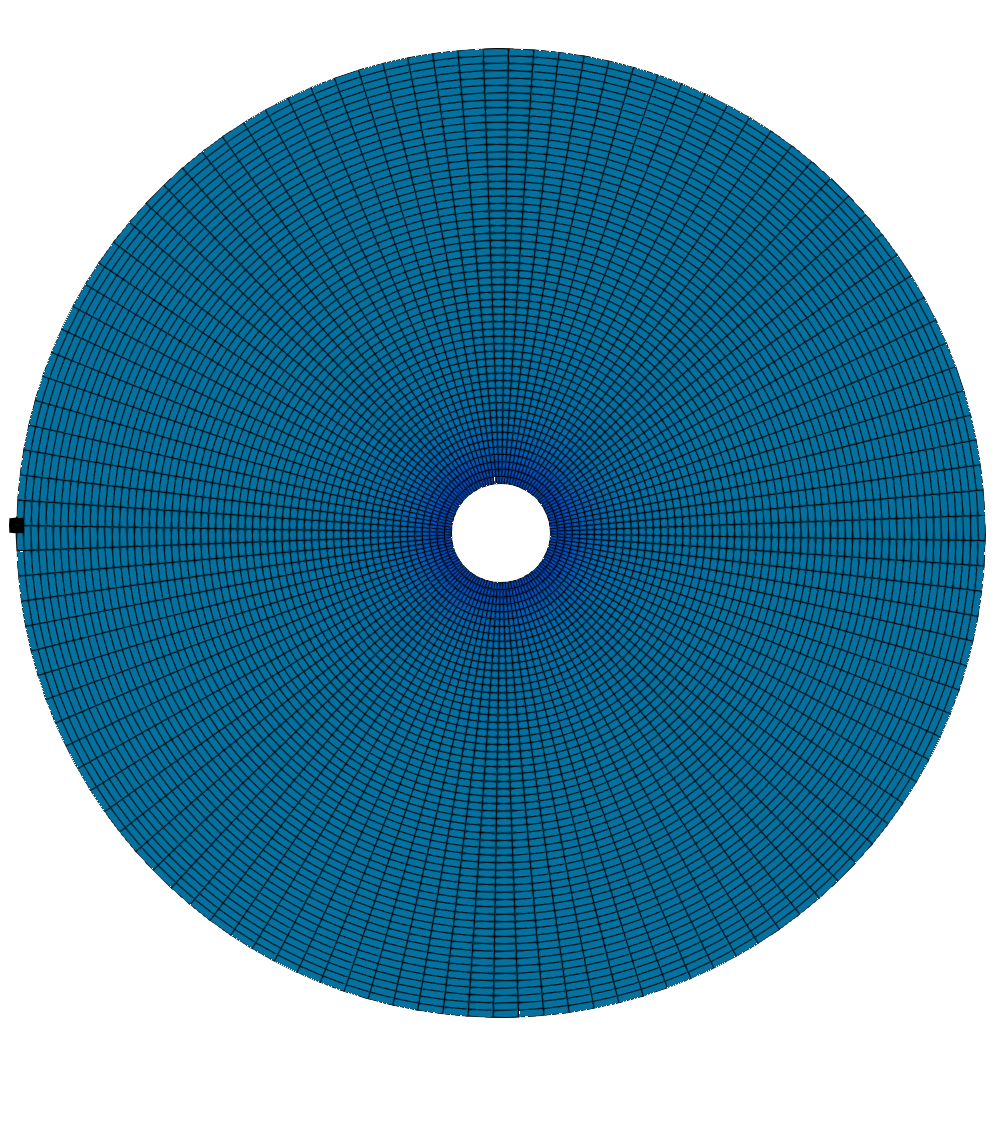} &
		\includegraphics[width=0.3\textwidth]{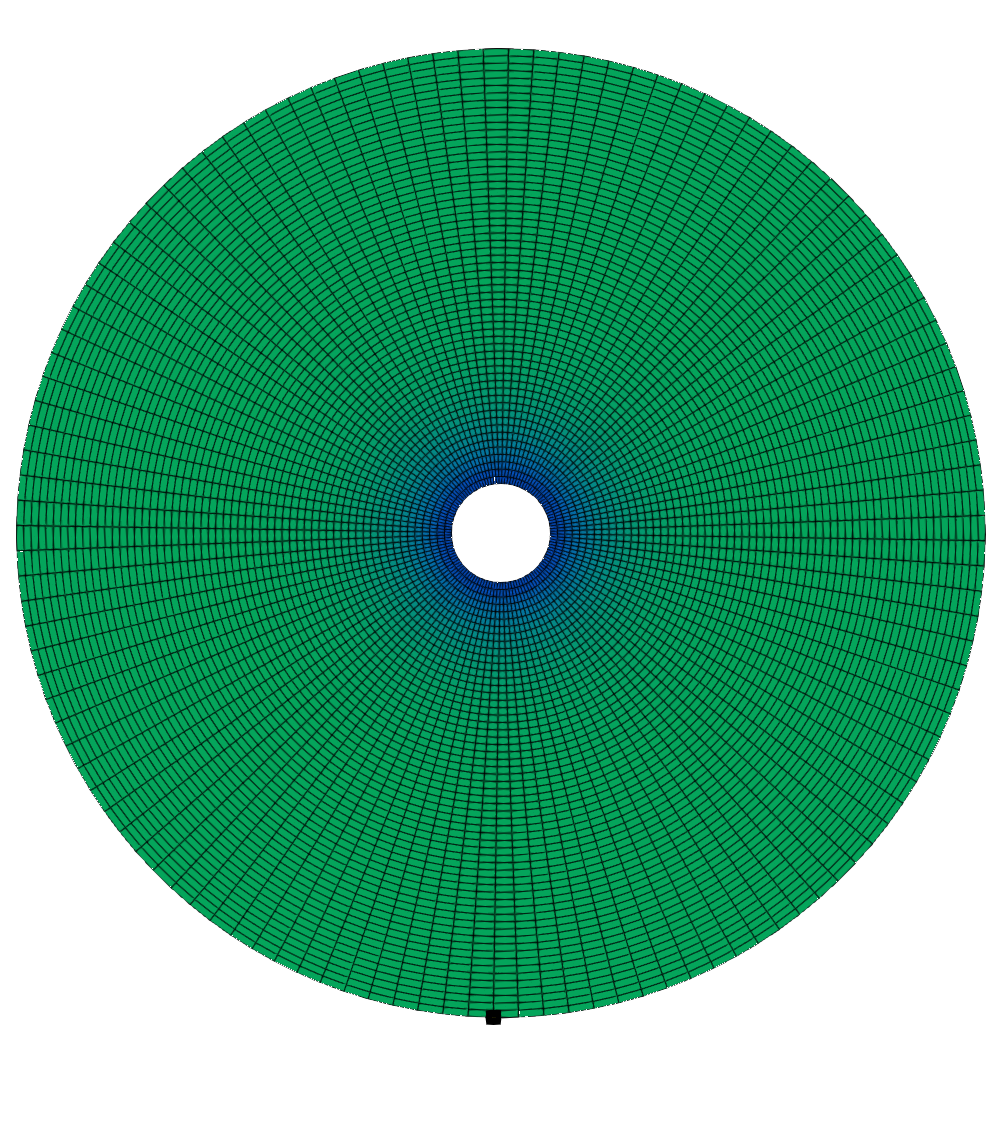}             \\
		(a)                                                           & (b) & (c) \\[6pt]

		\includegraphics[width=0.3\textwidth]{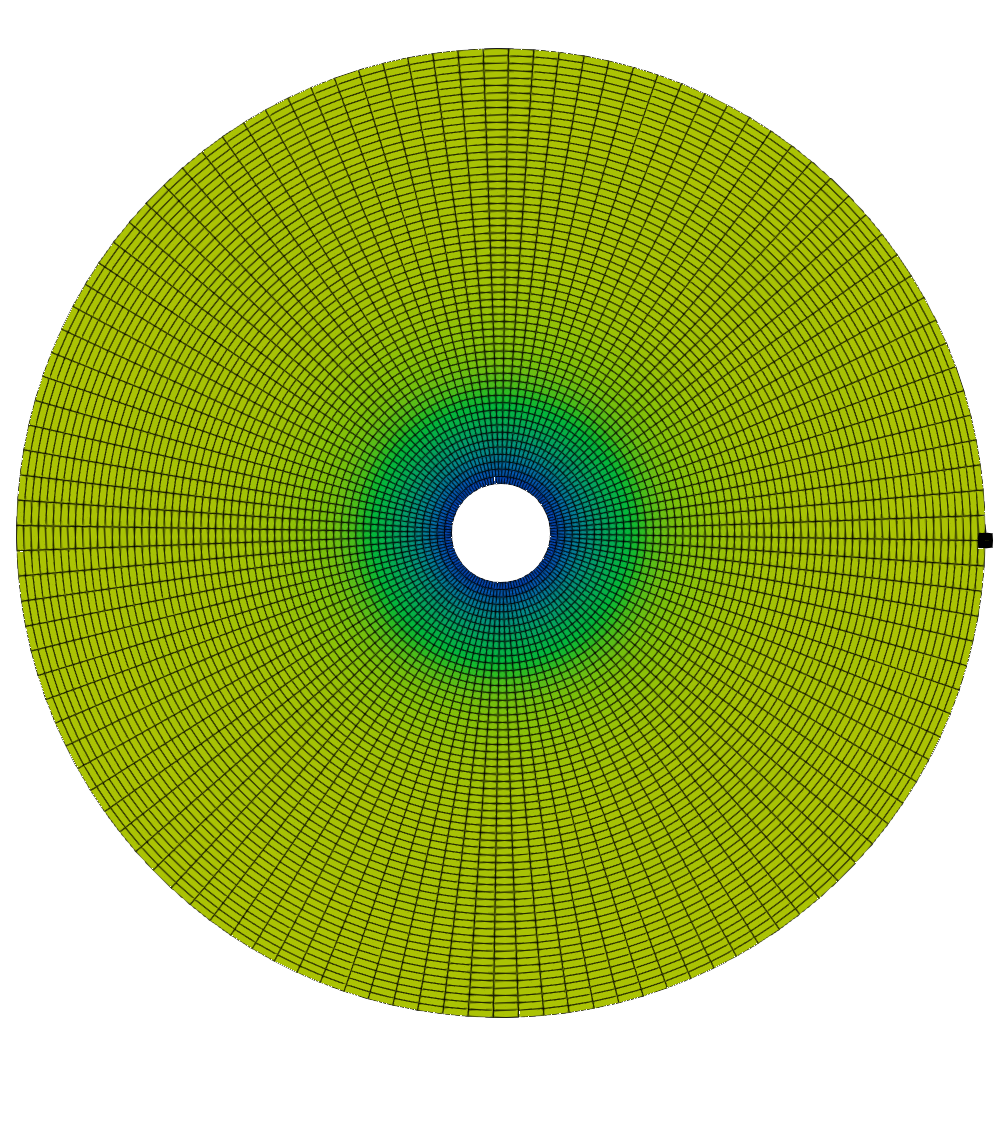} &
		\includegraphics[width=0.3\textwidth]{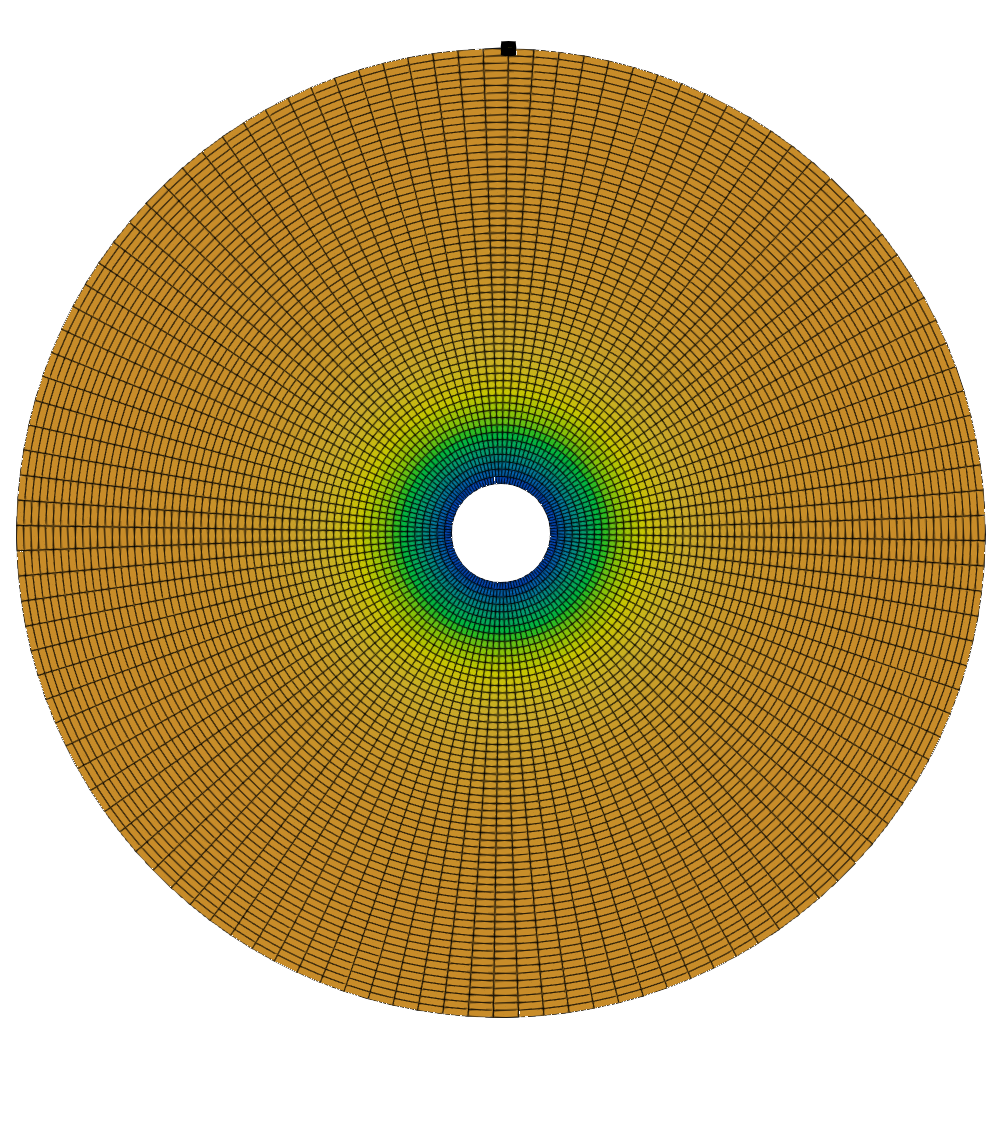} &
		\includegraphics[width=0.3\textwidth]{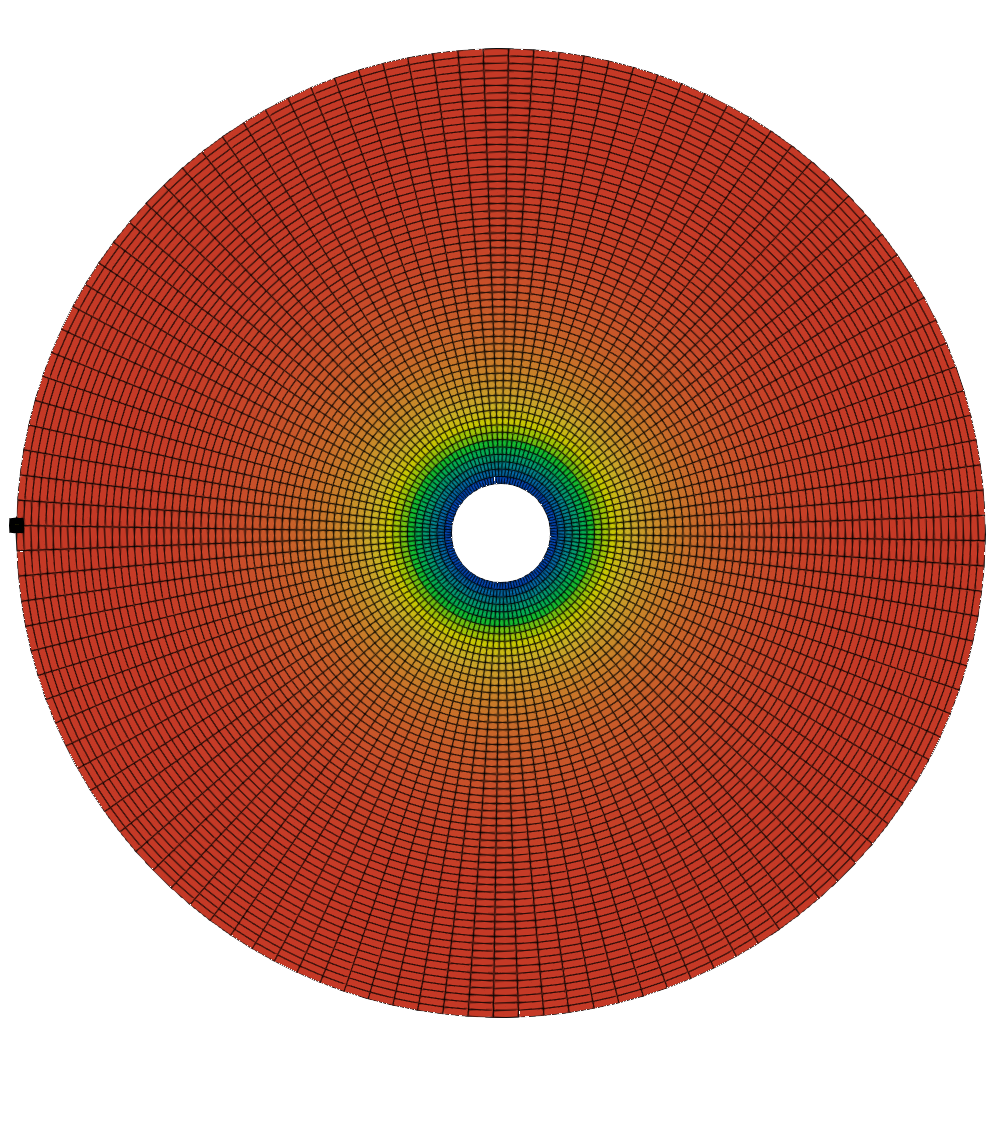}             \\
		(d)                                                           & (e) & (f)
	\end{tabular}
	\caption{The non-frame-indifferent dissipation \ref{Case1-oldDissipation}  in Experiment 1. We observe that the temperature increases over the time instants (a)--(f), where the rotation is highlighted via the black square on the outer boundary of the annulus. This is clearly unphysical since the body merely rotates by a time-dependent rigid body motion.  The simulation is conducted for zero Neumann boundary conditions at the outer boundary of the annulus, i.e., $\kappa = 0$ in   \eqref{main_bc_heat},   and for constant Dirichlet boundary conditions at the inner boundary, formally corresponding to $\kappa \to \infty$ in \eqref{main_bc_heat}. (Strictly speaking, such mixed-type boundary conditions are not covered by our theoretical result, but could be considered in principle.)    }
	\label{fig:example_frameNonIvariance}
\end{figure}

\begin{theorem}\label{thm:quantitativeerror}
	Assume that $W^{\rm el}$ is quasiconvex with $W^{\rm el}(F) = 0$ if $F \in SO(3)$ and $W^{\rm el}(F) >0$ otherwise. \
	Let $h$ be given by \eqref{boundary-rotation}, and suppose that at a previous time step $(k-1)\tau$ the deformation is given by $\yst{k-1} = h((k-1)\tau,\cdot)$.  
	\begin{enumerate}

		\item[(i)] Consider \ref{Case2-newDissipation}.
		      Then, the minimizer in \eqref{mechanical_step2} is given by
		      \begin{equation*}
			      \argmin_{y \in \Wid{k\tau}} \Big\{
			      \mechen(y)
			      + \frac{1}{\tau} \diss(\yst{k-1}, y , \theta_\flat) \Big\}=   h(k\tau,\cdot)
		      \end{equation*}
		      with   $ \diss(\yst{k-1}, h(k\tau,\cdot) , \theta_\flat)   =0$.
		\item[(ii)]  If \ref{Case1-oldDissipation} holds,  then   for $\tau>0$ sufficiently small, the minimizer in \eqref{mechanical_step2} is given by
		      \begin{equation*}
			      \argmin_{y \in \Wid{k\tau}} \Big\{
			      \mechen(y)
			      + \frac{1}{\tau} \diss(\yst{k-1}, y , \theta_\flat) \Big\}=  h(k\tau,\cdot).  
		      \end{equation*}
		      Moreover, it holds that $ \diss(\yst{k-1}, h(k\tau,\cdot) , \theta_\flat)   > \tau^2  c_h$   for a constant $c_h>0$.
	\end{enumerate}
\end{theorem}
In the present setting, we assume that the second-order contribution $H$ vanishes, but we note that the statement of Theorem \ref{thm:quantitativeerror}  also holds for convex $H$ with $H(0) = 0$ and $H\geq 0$.

\begin{proof}
	(i)   In this case, we directly check that $ \diss(\yst{k-1}, h(k\tau,\cdot) , \theta_\flat)   =0$, i.e., $h(k\tau,\cdot)$ is clearly a minimizer and the dissipation is zero. 
	 The uniqueness follows from the following observation: If a minimizer $\yst{k}$ satisfies  $ \diss( h((k-1)\tau,\cdot) ,\yst{k}, \theta_\flat)   =0$, we have $\nabla \yst{k} \in SO(3)$ with $\yst{k} \in \Wid{k\tau}$. Thus, Liouville's theorem gives $\yst{k} =  h(k\tau)$. 

	(ii) In   case  \ref{Case1-oldDissipation}, let us check that the functional $y \mapsto\frac{1}{\tau} \diss(\yst{k-1}, y , \theta_\flat)$ is \emph{strictly} convex on $\Wid{k\tau}$, where we restrict ourselves to constant tensors $V$ for easier presentation.
	Given $u,v \in \Wid{k\tau}$ with $u \neq v$, a computation reveals
	\begin{align*}
		 & \int_\Omega \tfrac{1}{2} \Big( \left\vert \sym \big((\nabla \yst{k-1})^T (\nabla u  - \nabla \yst{k-1}) \big) \right\vert^2 + \left\vert \sym \big((\nabla \yst{k-1})^T (\nabla v - \nabla \yst{k-1} ) \big) \right\vert^2 \Big) \di x                                   \\
		 & \quad - \int_\Omega \left\vert \sym \big((\nabla \yst{k-1})^T \big(\tfrac{1 }{2} (\nabla u + \nabla v) - \nabla \yst{k-1} \big) \big) \right\vert^2 \di x   = \frac{1}{4} \int_\Omega \left\vert \sym \big( (\nabla \yst{k-1})^T \nabla (u-v) \big) \right\vert^2 \di x.
	\end{align*}
	By the generalized version of Korn inequality, see e.g.\ \cite[Theorem 3.1  (i)]{RBMFLM}, the right-hand side is positive, showing the strict convexity.  Due to the quasiconvexity (see \cite{dacorogna}, for instance), we immediately get that
	$h(k\tau,\cdot)$ is a solution, i.e., (ii) holds. Indeed, it follows from the definition of quasiconvexity that the energy functional is minimized by  the affine map for given affine boundary conditions.  Further, we check that  $ \diss(h((k-1)\tau,\cdot) , h(k\tau,\cdot) , \theta_\flat)   >  \tau^2 c_h$ for some $c_h >0$.  
\end{proof}

Due to the fact that the dissipation \ref{Case1-oldDissipation}  is not invariant with respect to rigid motions, the theoretical result predicts that energy is dissipated. This is confirmed by the numerical solution depicted in Figure~\ref{fig:example_frameNonIvariance}, where we see  that the temperature increases during the rigid body rotation.   As this  is clearly unphysical, this experiment demonstrates the necessity of dynamically  frame-indifferent discretizations.

\subsection{Experiment 2 (Long-time behavior)}\label{sec:Experiment2}
The long-time behavior of the isothermal model has been analyzed analytically in \cite[Theorem 2.3]{Machillpvisco} for body forces constant in time under smallness assumptions by exploiting a metric gradient flow formulation. An exponential convergence rate towards the state with minimal elastic energy has been predicted.  
In this experiment, we illustrate such a convergence result in a simulation, considering the compressible neo-Hookean potential from \eqref{neohookean}. Here, we focus on traction loadings $g$ instead of body forces $f$:
a \emph{creep test} is designed to   identify   the time-dependent deformation of the material subjected to a constant applied stress.  We consider a rectangular body of length $1$ fixed at its left boundary and subjected to uniaxial traction loading  which is  applied in the
$x_1$-direction at the right boundary of the specimen. This traction is kept constant at all times. The initial deformation $y_0$ is given by the identity, and the initial temperature ~$\theta_0$ is constant, chosen as $293K$ (Kelvin).
 The experiment is conducted for the potentials from Section~\ref{sec:Experiment1}, where we note that the heat capacity becomes independent of $F$ and $\theta$ and takes the form $-\theta \pl_\theta^2 \cplpot( F , \theta) = C_1$. 
The temperature at the left boundary is prescribed and kept fixed at the same value for all times. 

The loading induces a
uniaxial stress state. For different ratios between $\nu$ and $\mu$ and the fixed Poisson ratio of $0.125$ (corresponding to $\lambda = \mu$ in \eqref{neohookean}), the resulting deformation is recorded as a
function of time, depicted in Figure~\ref{fig:creep}(a). The axial strain, quantified by the deformation
gradient component $\partial_1 y_1(t)  $, increases over time  and is constant in space
for the constant
applied load, reflecting the material’s time-dependent response.  
The measured strain--time curve thus provides direct insight into the   rheological parameters and serves as a benchmark for validating constitutive formulations describing long-term deformations under sustained loads.  
Figure~\ref{fig:creep}(b) shows the evolution of the temperature field for the case $ \nu/\mu = 0.5$. 
The heat conductivity tensor $\mathbb{K}(\theta)$ is assumed to be independent of $\theta$, namely $ \mathbb{K}(\theta) = k  \Id$, where $k$ satisfies $C_1/k = 5$.
  Qualitatively, as both equations are of parabolic type, the speed of convergence is related to these ratios. 
 We observe that the mechanical response evolves on a significantly faster time scale than thermal equilibration. Consequently, the deformation relaxes before the temperature field becomes spatially uniform.

\begin{figure}[htbp!]
	\centering
	\begin{tabular}{cc}
		\includegraphics[width=0.47\textwidth]{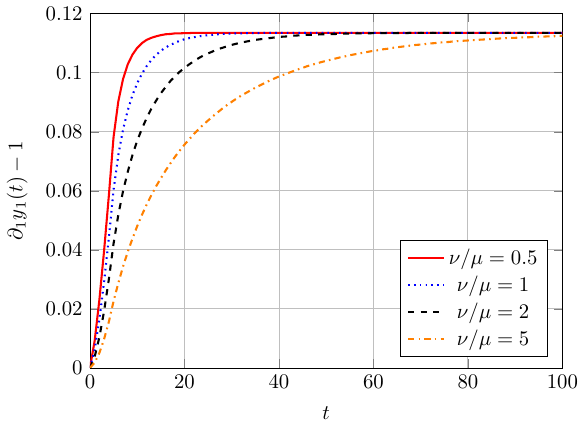}
		    &
		\includegraphics[width=0.46\textwidth]{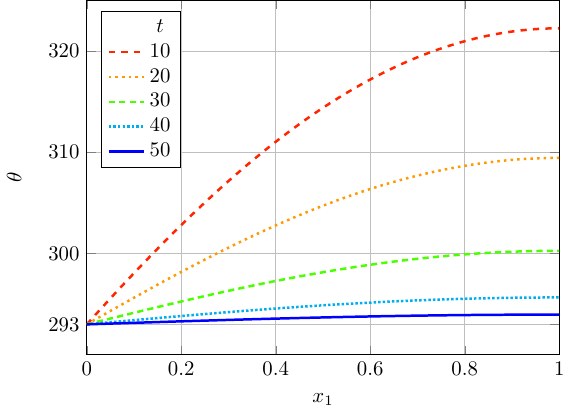} \\
		(a) & (b)                                                             \\
	\end{tabular}
	\caption{Creep test in Experiment 2 for an elastic energy density as in \eqref{neohookean}  with $\lambda = \mu$ and a viscous potential as in \eqref{viscositytensor}.   
     Subfigure (a) shows that the material response to stress is slower for  increasing ratio  $\nu/\mu$, and Subfigure (b) displays the evolution of the temperature for $\nu/\mu = 0.5$.  
	}
	\label{fig:creep}
\end{figure}

\subsection{Experiment 3 (Shape memory alloys)}\label{sec:Experiment3}
As shown in \cite{positivity24}, there exist solutions in the sense of \eqref{weak_limit_mechanical_equation}--\eqref{weak_limit_heat_equation} that preserve the positivity of the temperature throughout the evolution. This is a nontrivial issue
due to the presence of the adiabatic heat-absorbing term in the nonlinear heat equation~\eqref{main_thermal_eq}.
Thus, potential cooling effects are possible while the material undergoes  deformation.

In this experiment, we show that applied loadings may induce cooling effects
of the material.
To this end, we need to recover a situation, where the term $\theta \pl_{F \theta} \cplpot(\nabla y, \theta) : \nabla \partial_t y$ becomes negative, while the dissipation rate
$\drate(\nabla y, \nabla \partial_t y, \theta) \geq 0$ is comparably small. The latter condition is guaranteed for a small viscosity constant $\nu>0$, while the coupling potential necessarily needs to depend on both $F$ and $\theta$ as the mixed derivative $ \pl_{F \theta} \cplpot$ vanishes otherwise.

 Moreover, with the experiment we also want to show the complex temperature distribution and evolution in such systems, induced by  diffusive effects and heat sources that are
directly coupled to the strain rate. Note that the quadratic dissipation leads to a rate-dependent generation of heat. 
 In particular, variations in viscosity affect both the time scale of the mechanical response and the magnitude of the internal heat production, which may lead to pronounced temperature changes even when the characteristic diffusion time remains unchanged.

In \cite{MielkeRoubicek20Thermoviscoelasticity},   
the authors provide a family of free energy potentials modeling austenite-martensite transformations in so-called shape-memory alloys, where the free energy potential in \eqref{eq: free energy}  takes the form
\begin{equation}\label{freeenergyexample}
	W(F, \theta) = (1 - a(\theta)) W_M(F) + a(\theta) W_A(F) + C_1 \theta (1 -  \log\theta).
\end{equation}
Here, $C_1 > 0$ denotes a fixed constant as before, and $a\colon \R_+ \to [0,1]$  represents  the volume fraction between austenite and  martensite, which we assume to depend only on temperature.
Moreover, $W_M$  is a frame-indifferent
multi-well potential modeling the martensite state while $W_A$ denotes a frame-indifferent
single-well potential with minimum on $SO(3)$, corresponding to the austenite state. We consider here a prototypical example, building on the potential in \eqref{neohookean}:
given parameters $\mu$ and $\lambda$, we first  define
\begin{align*}
	Z_\eps(F) \defas \mu\Big(   \det(F)^{-2/3}| F G_\varepsilon  |^2   - 3 \Big)^2 + \lambda\Big(   \det(F) + \frac{1}{\det(F)} - 2\Big)^2  \quad \text{for} \quad G_\varepsilon \defas \left(  \begin{array}{lll}1 & \varepsilon &0 \\ 0 & 1 & 0 \\ 0& 0& 1\end{array}\right) .
\end{align*} 
Clearly, the minimum is attained for any \ $F=RG_{-\varepsilon}$ for $R\in {\rm SO}(3)$.
We set $$ W_A(F) = Z_0 (F)  $$ if $\det(F) >0$ and $W_A(F) = +\infty$ otherwise.
For $W_M$ instead, we consider a function which is minimized on two wells representing shear deformations along the $x_1$-direction, namely,
$$	W_M(F)=  \min \{ Z_{\varepsilon}(F), Z_{-\varepsilon}(F) \} $$ if $\det(F) >0$. 
Given the above potentials, for any $F \in GL^+(  3  )$ and $\theta > 0$,  we set  $a(\theta)      \defas \theta (1+\theta)^{-1}$,  
\begin{align} \label{specialcase}
	W^{\rm el} (F)         \defas W_M(F),   \qquad \text{and} \qquad
	\cplpot(F, \theta)  \defas a(\theta) (W_A(F) - W_M(F)) + C_1 \theta (1 - \log(\theta)).
\end{align}
  Once again, we consider a viscosity tensor with viscosity constant $\nu>0$ from \eqref{viscositytensor}  and the fixed choice $\lambda = \mu$.  
We perform numerical experiments with $\varepsilon=0.01$  on a rectangular domain with fixed
displacements at both the left and right ends, and a prescribed cyclic
traction $g$ applied on the top surface in the $x_2$-direction, i.e., only the second component of $g$ in \eqref{main_bc} is nonzero, displayed in Figure \ref{fig:speed}. 
 This corresponds to a classical
engineering configuration of a clamped--clamped beam subjected to a
distributed transverse load. \begin{figure}[htbp!]
	\centering
    \includegraphics[width=0.4\textwidth]{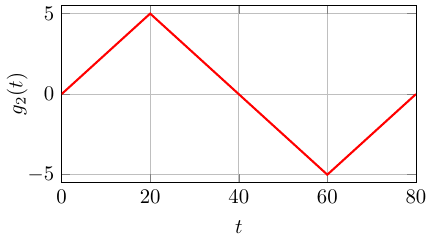}
	\caption{The evolution of the (vertical) boundary traction in Experiment $3$.
	}
	\label{fig:speed}
\end{figure} 
  During the loading process, the temperature at
the left and right end is kept fixed at $293K$. Moreover, the evolution of the deformation starts from the identity with the constant initial temperature $293K$.
As in Section~\ref{sec:Experiment2}, parameters are chosen such that time scale of the thermal diffusion is slower than the mechanical response:
More precisely,  \eqref{specialcase} leads to a nonlinear heat capacity, taking the form $-\theta \pl_\theta^2 \cplpot( F , \theta)  = C_1 - \theta a''(\theta) (W_A(F) - W_M(F))$.
The heat conductivity tensor $\mathcal{K}(\theta) = k \Id$ satisfies $- 293\pl_\theta^2 \cplpot( \Id , 293) /k = 10$, and we consider ratios $\mu/\nu$ significantly smaller than $10$.

\begin{figure}[htbp!]
	\centering
	\begin{tabular}{cc}
		\includegraphics[width=0.45\textwidth]{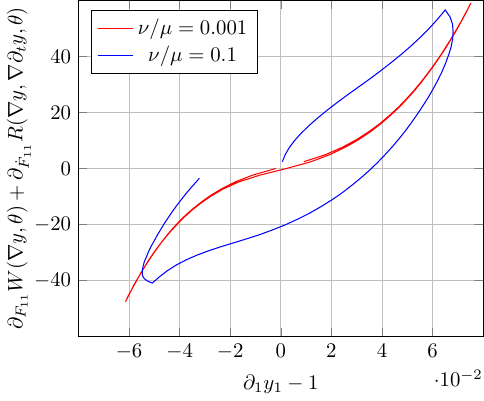}
		    &
		\includegraphics[width=0.518\textwidth]{temperature.pdf} \\
		(a) & (b)                                                             \\
	\end{tabular}
	\caption{Experiment $3$ for different ratios of $\nu/\mu$, recorded at the midpoint of the bottom surface. Subfigure (a) shows the stress-strain curve for one loading circle under the cyclic loading $g$ in Figure~\ref{fig:speed},
     i.e., $80$ time steps.  Subfigure (b) records the evolution of the temperature. 
	}
	\label{fig:speed-new}
\end{figure}

\begin{figure}[htbp!]
	\centering
	\resizebox{0.9\textwidth}{!}{
		\begin{tabular}{cc}
			\includegraphics[width=0.5\textwidth]{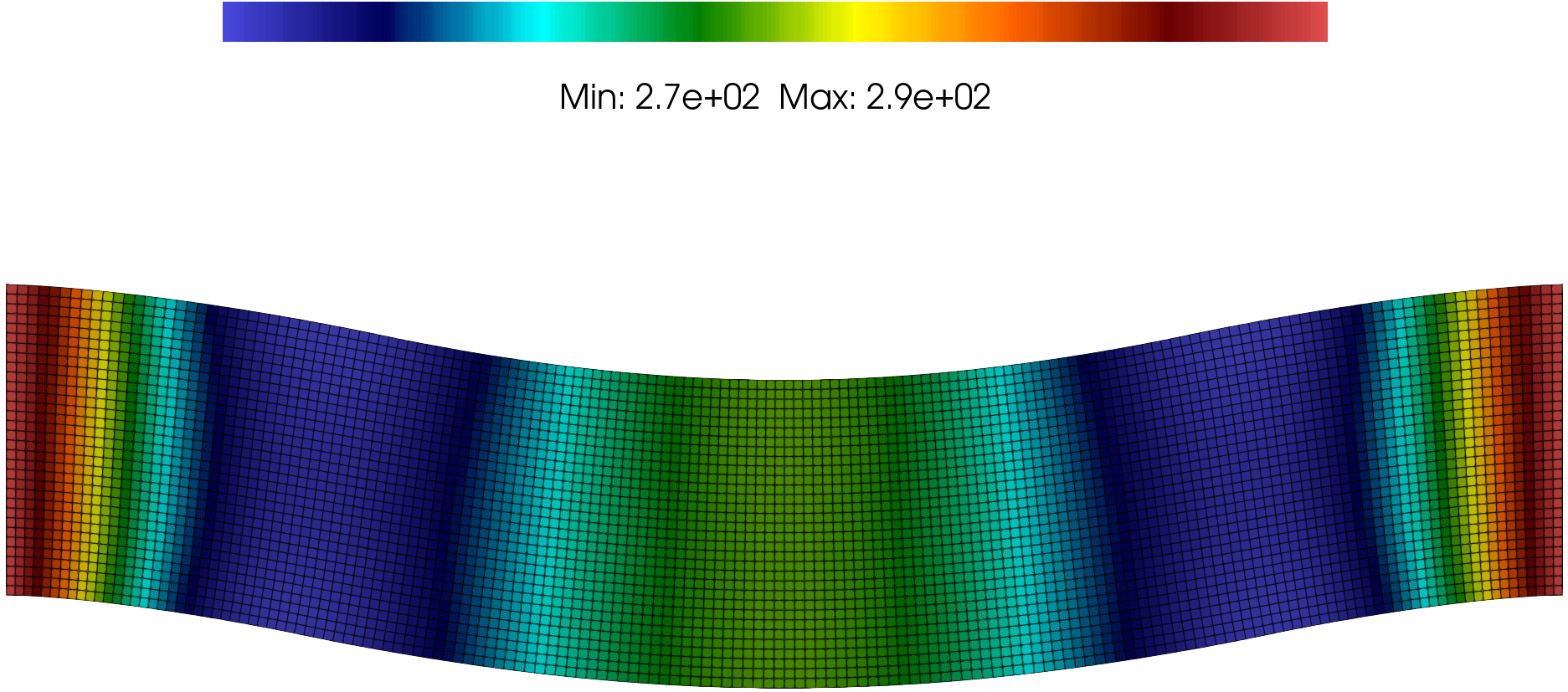}
			    &
            \includegraphics[width=0.5\textwidth]{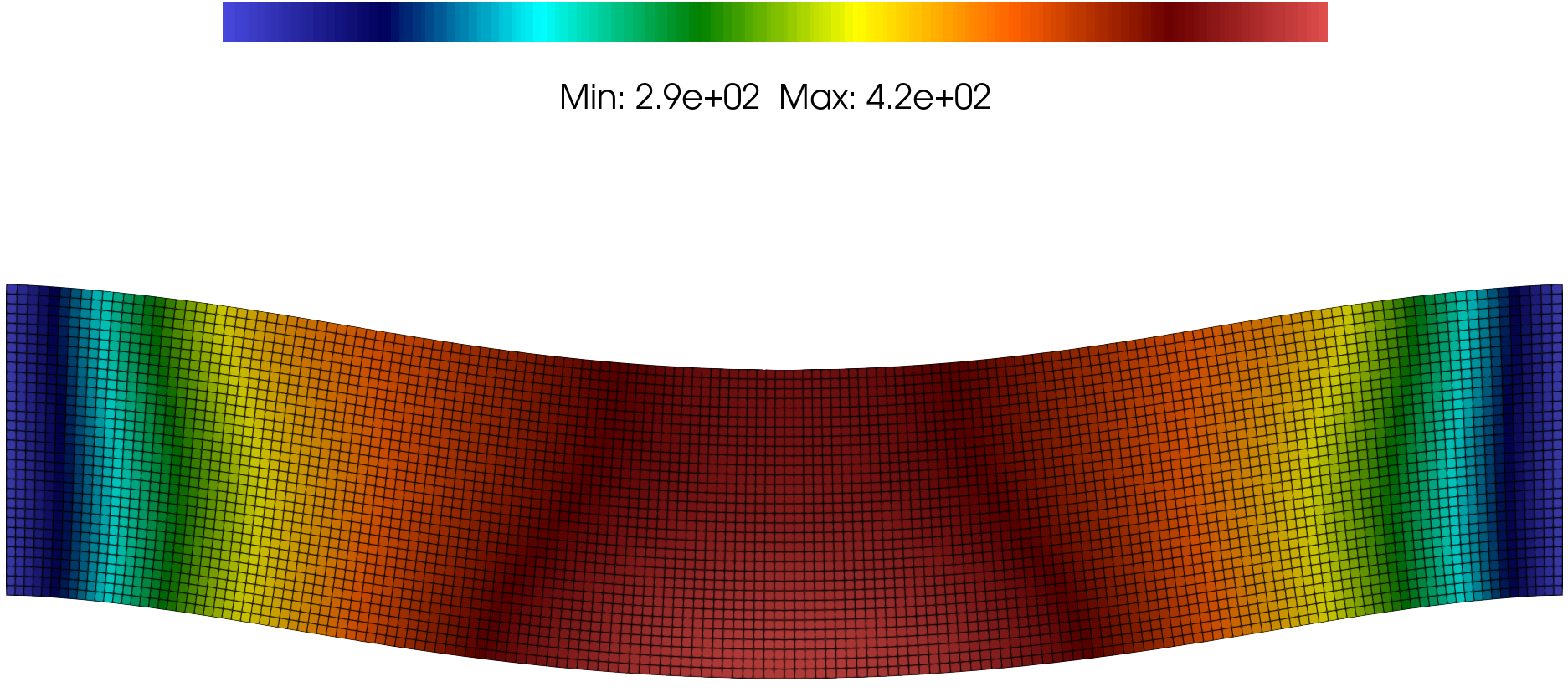}  
            \\
			(a)    & (b)   
            \\
            \\
            			\includegraphics[width=0.5\textwidth]{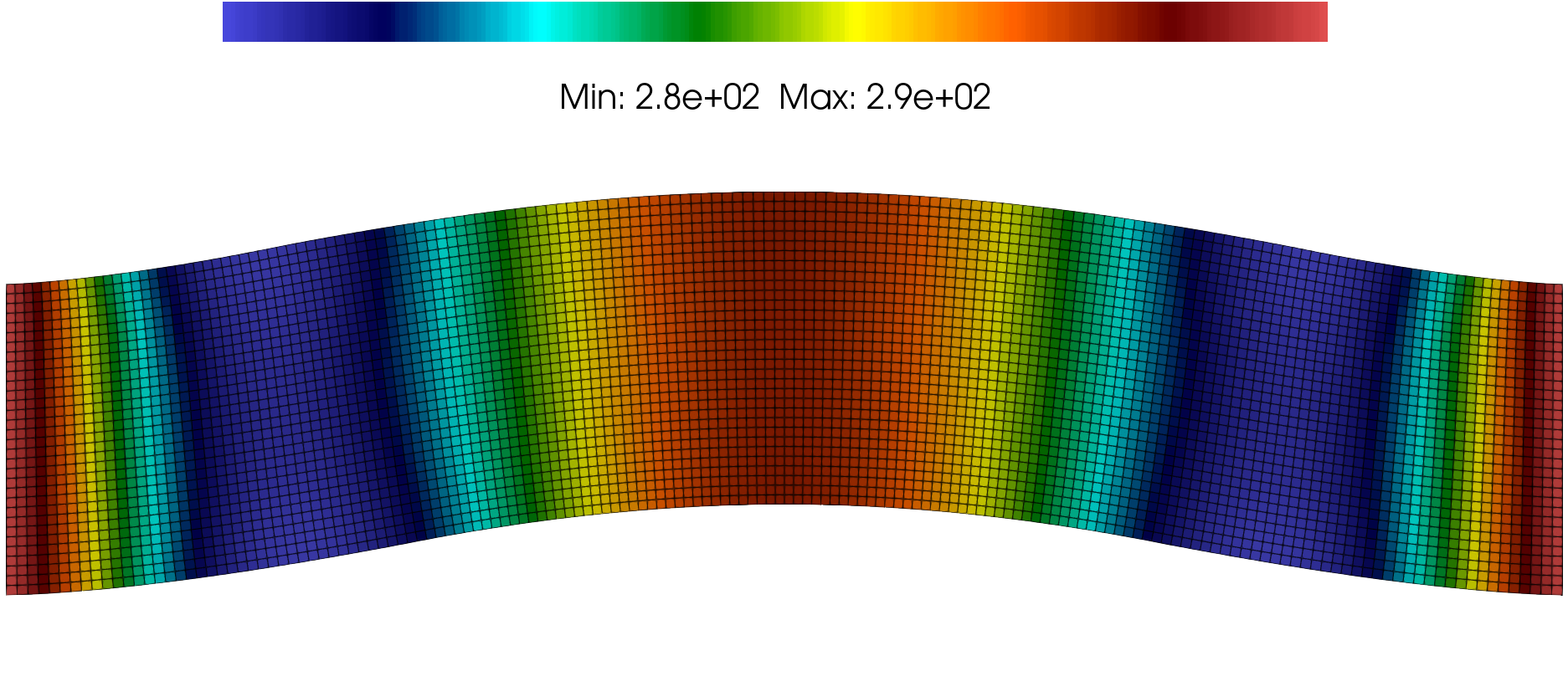}
			    &
            \includegraphics[width=0.5\textwidth]{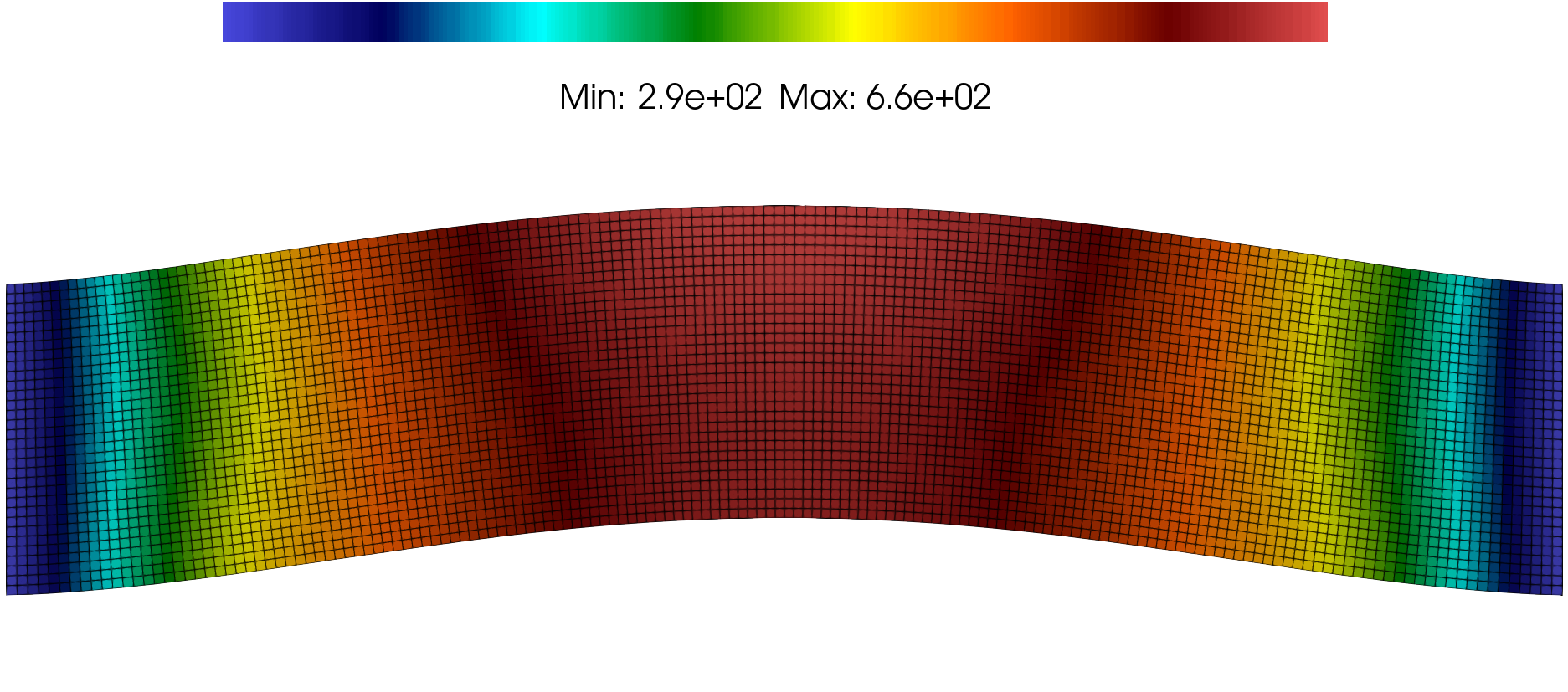}  
            \\
			(c)     & (d)                                                                      \\
            \end{tabular}
	}
	\caption{Spatial temperature distribution at two distinct time steps in Experiment 3. Subfigures (a) and (b) show results at time step 20, whereas (c) and (d) correspond to time step 60. The left column represents the case $\nu/\mu = 0.001$, and the right one displays $\nu/\mu=0.1$. }
	\label{fig:figure5}
\end{figure}

 Figure~\ref{fig:speed-new}(a) shows the corresponding stress–strain curves.
 Here, the curves first remain close to the origin as the loading $g$ is comparably small, see Figure~\ref{fig:speed}. Notice that once the (weak) derivative of $g$ changes its sign, the stress displays nonsmooth kinks after $20$ and $60$ time steps, respectively. After one loading cycle, i.e., $80$ time steps, the specimen does not reach its initial position for a large viscosity (blue curve) due to the `delayed' behavior.
The figure hence illustrates the influence of viscosity on the mechanical behavior.
  While the mechanical energy which is dissipated is almost negligible for the small viscosity parameter (red curve), it becomes substantially pronounced for the higher viscosity (blue curve), corresponding to significantly larger hysteresis. Consequently, in the latter case, a much higher heat is generated during the loading cycle which can be seen in the associated evolution of the temperature  depicted in Figure~\ref{fig:speed-new}(b). Here, the temperature evolution of  the midpoint of the bottom surface is recorded. In particular, the red curve shows a cooling effect through the adiabatic term $\theta \pl_{F \theta} \cplpot(\nabla y, \theta) : \nabla \partial_t y$ as temperatures below the initial and boundary temperature are attained. For a large ratio $\nu/\mu$, the heat production due to the viscous dissipation $\drate(\nabla y, \nabla \partial_t y, \theta)$ dominates. Here, the cooling effects after time steps $20$ and $60$ are caused by   diffusion of the temperature at the boundary and smallness of the dissipation as the loading direction changes. 
Figure~\ref{fig:figure5} shows the deformed material and its temperature distribution in the whole specimen at selected time steps $20$ and $60$. The latter correspond to the time steps where the slope of the loading changes sign, see Figure~\ref{fig:speed}. 
In Subfigures (a) and (c), the material points in the middle do not cool down as fast as the points which are closer to the left and right boundary.

 In general, these cooling effects can also be justified heuristically from the equation, by discussing the sign of the adiabatic term $\theta \pl_{F \theta} \cplpot(\nabla y, \theta) : \nabla \partial_t y = a'(\theta) (\partial_F W_A(\nabla y) - \partial_F W_M (\nabla y)): \nabla \partial_t y$.
As we apply vertical loading, the component $\partial_t \partial_2 y_2$ plays a significant role in the strain rate $\nabla \partial_t y$, and the sign of this quantity changes once the motion of the deformation changes the direction. At the same time, the deformation gradients $\nabla y$ in the motion downwards and upwards are comparable, which is why the adiabatic term changes its sign once the motion changes its direction.

\begin{figure}[htbp!]
	\centering
    \resizebox{0.9\textwidth}{!}{
	\begin{tabular}{cc}
		\includegraphics[width=0.49\textwidth]{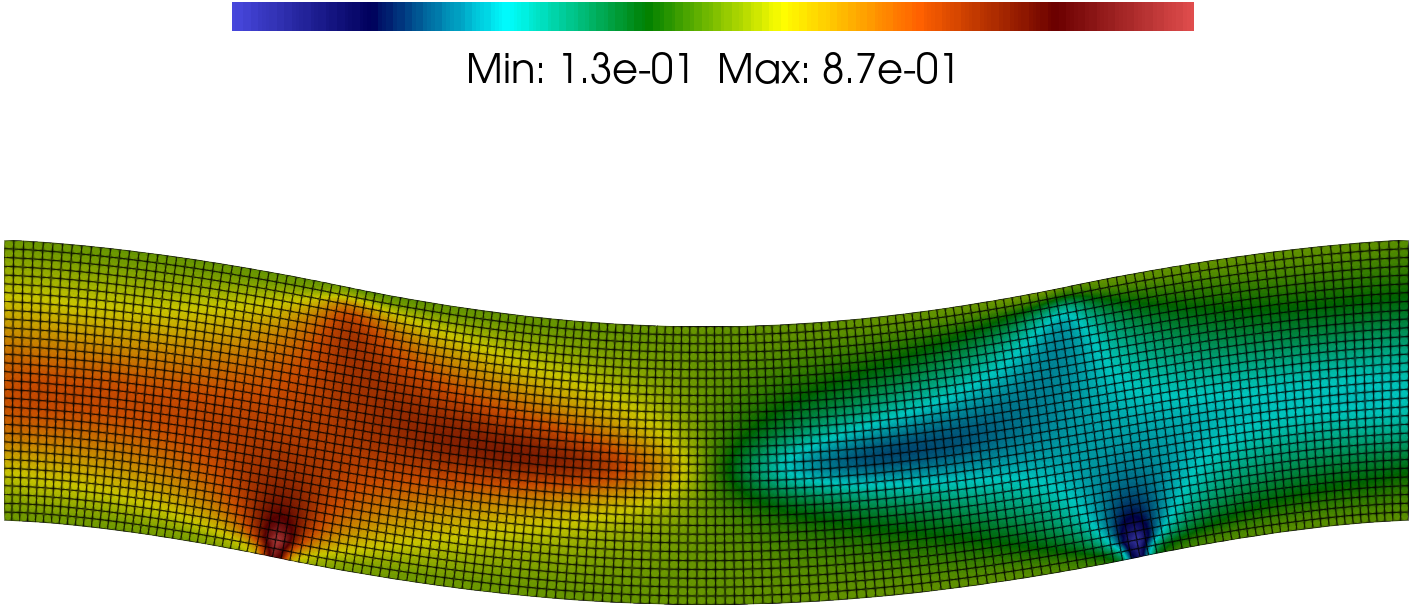}
		    &
		\includegraphics[width=0.49\textwidth]{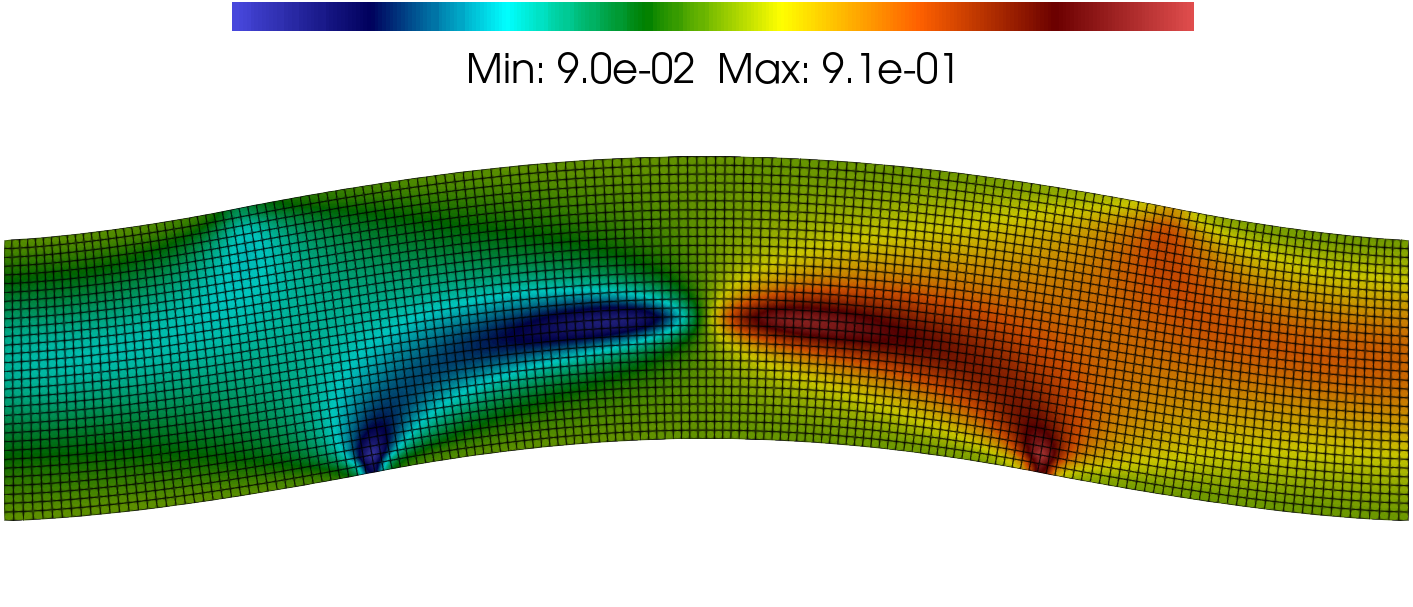}  \\
		(a) & (b)                            
	\end{tabular}
    }
	\caption{ Phase evolution in Experiment $3$ for time steps (a) 20 and (b) 60. The value on the color bar corresponds to the function \eqref{phasefunction} evaluated at material points.   }
	\label{fig:phases}
\end{figure}

 Finally, in Figure~\ref{fig:phases}, we visualize the phase-indicator function
\begin{align}\label{phasefunction}
	F\mapsto \frac{|F^T F-F_{-\varepsilon}^T F_{-\varepsilon} |^2}{|F^T F-F_{- \varepsilon}^T F_{-\varepsilon} |^2+|F^T F-F_{ \varepsilon}^T F_{\varepsilon} |^2},
\end{align}
 which measures the distance of the right Cauchy-Green strain to the martensitic wells. This shows that the adiabatic cooling effect in this simulation becomes visible once the deformation gradient is close to the martensitic phases.

\subsection*{Acknowledgements}
M.H. and M.K. gratefully acknowledge the support by CSF grants 23-04766S and 24-10366S. Moreover, they are indebted to the Departments of Mathematics of  the FAU and the JKU for hospitality. L.M.\ gratefully acknowledges support by the Deutsche Forschungsgemeinschaft (DFG, German Research Foundation) under Germany's Excellence Strategy -- EXC-2047/1 -- 390685813, and through the CRC 1720 -- 539309657. The work was further supported
by the DAAD project 57702972 and CAS Mobility Plus project DAAD-24-10.

\typeout{References}

\end{document}